\documentclass[a4paper,oneside]{amsart}

\usepackage{amsmath,amsthm,amssymb,enumerate}
\usepackage{epsfig}
\usepackage{amssymb}
\usepackage{amsmath}
\usepackage{amssymb}
\usepackage{amsmath,amsthm}
\usepackage[latin1]{inputenc}
\usepackage[T1]{fontenc}
\usepackage{ae,aecompl}
\usepackage{amsfonts}
\usepackage{amsxtra}
\usepackage{bbm,euscript,mathrsfs}
\usepackage{color}
\usepackage[left=3.3cm,right=3.3cm,top=4cm,bottom=3cm]{geometry}
\allowdisplaybreaks

\usepackage{enumitem}
\setenumerate{label={\rm (\alph{*})}}

\usepackage{amsfonts}
\usepackage{amsxtra}

\usepackage[frak,theorem_section]{paper_diening}
\numberwithin{equation}{section}

\DeclareMathOperator{\Test}{\mathscr{F}}

\newcommand{\Testzeta}{\Test_{\eta}^{\mathrm div}}

\newcommand{\Div}{\divergence}
\newcommand{\ep}{\varepsilon}

\newcommand{\R}{\mathbb R}
\newcommand{\N}{\mathbb N}

\newcommand{\vu}{\mathbf{u}}

\newcommand{\Grad}{\nabla}

\newcommand{\dd}{\mathrm{d}}
\newcommand{\dx}{\,\mathrm{d}x}
\newcommand{\dt}{\,\mathrm{d}t}
\newcommand{\dxt}{\,\mathrm{d}x\,\mathrm{d}t}
\newcommand{\ds}{\,\mathrm{d}\sigma}
\newcommand{\dxs}{\,\mathrm{d}x\,\mathrm{d}\sigma}
\newcommand{\dH}{{\, \dd y}}

\DeclareMathOperator{\tr}{tr}

\newcommand{\dif}{\mathrm{d}}

\newcommand{\mr}{\mathbb{R}}

\newcommand{\mt}{\mathbb{R}^3}

\DeclareMathOperator{\dist}{dist}

\DeclareMathOperator{\diver}{div}

\newcommand{\bu}{\mathbf u}
\newcommand{\bq}{\mathbf q}

\allowdisplaybreaks

\newcommand{\vr}{\varrho}
\newcommand{\vt}{\vartheta}
\newcommand{\vrd}{\varrho_\delta}
\newcommand{\vtd}{\vartheta_\delta}
\newcommand{\vud}{\bfu_\delta}

\newcommand{\weaktostar}{\weakto^*}
\newcommand{\toetai}{\to^\eta}
\newcommand{\weaktoetai}{\weakto^\eta}

\newcommand{\rdiss}{{\mathscr{D}_\delta}}

\newcommand{\seb}[1]{\textcolor[rgb]{0.00,0.00,1.00}{  #1}}

\begin{document}

%\begin{frontmatter}

\title[Compressible heat-conducting fluid-structure interactions] {Navier--Stokes--Fourier fluids interacting with elastic shells} %{Compressible heat-conducting fluids interacting with a nonlinear-elastic shell}
%{On compressible fluids interacting with a linear-elastic Koiter-type shell}

\author{Dominic Breit}
\address[D. Breit]{
Department of Mathematics, Heriot-Watt University, Riccarton Edinburgh EH14 4AS, UK}
\email{d.breit@hw.ac.uk}

\author{Sebastian Schwarzacher}
\address[S. Schwarzacher]{Departement of Analysis, Faculty of Mathematics and Physics, Charles University, Sokolovsk\'a 83, 186 75 Praha, Czech Republic}
\email{schwarz@karlin.mff.cuni.cz}

\begin{abstract}
We study the motion of a compressible heat-conducting fluid in three dimensions interacting with a nonlinear flexible shell. The fluid is described by the full Navier--Stokes--Fourier system. The shell constitutes an unknown part of the boundary of the physical domain of the fluid and is changing in time. The solid is described as an elastic non-linear shell of Koiter type; in particular it possesses a non-convex elastic energy.
% Its deformation is modelled by a (non-linear) elastic Koiter shell.
We show the existence of a weak solution to the corresponding system of PDEs which exists until the moving boundary approaches a self-intersection or the non-linear elastic energy of the shell degenerates. It is achieved by compactness results (in highest order spaces) for the solid-deformation and fluid-density. Our solutions comply with the first and second law of thermodynamics: the total energy is preserved and the entropy balance is understood as a variational inequality.
\end{abstract}

\subjclass[2020]{76N10, 76N06, 35Q40, 35R35}%checken!!
\keywords{Compressible fluids, Navier--Stokes equations, weak solution, Koiter Shell, time dependent domains, moving boundary}

\date{\today}

\maketitle

%\end{frontmatter}

\tableofcontents

\section{Introduction}
The interactions of fluids with elastic structures are important for many applications
ranging from hydro- and aero-elasticity \cite{Do} over bio-mechanics \cite{BGN} to hydrodynamics
\cite{Su}. Motivated by these applications and the scientific foundations from engineers and physicists also mathematicians became interested in the field. Nowadays there exists a vast body of literature on incompressible fluid structure interaction, where a part of the boundary of the underlying domain is the mid-section of a flexible shell.\\
The mathematical analysis of continuum mechanical models in fluid mechanics
%The presented work here is in the field of the theory on weak solutions in continuum mechanics. A field that 
reaches back to the pioneering work of Leray on the existence of weak solutions for the incompressible Navier--Stokes equations \cite{Ler}. Based on this, various fluid-structure interaction results have been achieved already; we will explain this in more detail below.
A similar foundational work in the compressible case is due to Lions \cite{Li} with important extensions by Feireisl et al. \cite{feireisl1,fei3}. Compressible fluids are
important for applications in aero-dynamics and mathematical results on their interactions with elastic structures appeared in this context recently in~\cite{BrSc,Tr}. A next natural step is to study the thermodynamics
of fluid structure interactions. In fact, the assumption that a physical process is isentropic can only be valid
for a very short period of time. In general it is indispensable to take into account the transfer of heat. Similarly, the linearisation of the shell model, often applied in the mathematical literature,
looses its validity as soon as
the displacement of the boundary is not on a small scale any more. The treatment of non-linear shell models in the context of weak solutions is very recent~\cite{MuSc} and (up to date) only available for incompressible fluids. In this work we progress on the theory of weak solutions by showing the existence for systems that take into account 1) heat conduction and compression effects for the fluid and 2) a non-linear elastic respond for the solid.
% we study a general viscous, compressible and heat-conducting fluid which interacts with a nonlinear elastic shell.
More specifically, we use the classical model by Koiter to describe the shell movement which yields a fully nonlinear fourth order hyperbolic
equation with a non-convex energy. 
The main result of this paper is the existence of a global-in-time weak solution to the Navier--Stokes--Fourier system coupled to the motion of a solid shell of Koiter type. This means that a fourth order PDE for the solid is coupled (via the geometry) to a viscous fluid. A special feature of the Navier--Stokes--Fourier system is that even weak solutions can satisfy an energy equality. We produce a respective equality for the energy of the coupled fluid-structure interaction; this includes the full Koiter energy of the solid deformation. In this context it is noteworthy that we consider perfect elastic shells. This means that no heat is produced by the solid, or reversely entropy is only increased via the fluid. Still some viscous effects can be shown to hold for the elastic solid due to the tight coupling between the solid and the fluid. It is this key observation (and the respective estimate in Subsection~\ref{sec:shelleps}) that allows to show that the elastic part of the energy has the necessary compactness in order to prove that the system is indeed closed (energy is preserved).
We note that the interval of existence for our weak solutions could be arbitrarily large. In fact, the time of existence is only restricted
once either the topology of the fluid domain changes, namely if a self-intersection of the variable boundary (of the elastic shell) is approached, or if the solid energy reaches a point of degeneracy.
% is prescribed via
%the two-dimensional mid-section of the 
%flexible Koiter shell whose energy is a nonlinear function of the
%first and second fundamental forms of the moving boundary.

\subsection{State of art}
Incompressible viscous fluids interacting with lower-dimensional linear elastodynamic equations were studied, for instance, in \cite{Ch, Gr, LeRu, MuCa1, MuCa, MuSc, SchSro20}. All but the last result are concerned with the existence of weak solutions which exist as long as the moving part of the structure does not touch the fixed part of the fluid boundary.
The analysis in \cite{Ch} is concerned with a three-dimensional viscous incompressible fluid modelled by the Navier--Stokes equations, which is interacting with a flexible elastic plate located on one
part of the fluid boundary. The shell equations is linearised and the shell is assumed to be one-dimensional.
The existence of a weak solution to the incompressible Navier--Stokes equation coupled with a plate in flexion was constructed in \cite{Gr}.
The authors in \cite{LeRu} studied the interaction of an incompressible fluid which interacts with a linear elastic shell of Koiter-type. Here, the middle surface of the shell serves as the mathematical boundary of the three-dimensional fluid domain. 
%The weak solution is shown to exist so long as the magnitude of the shell's displacement stays below %a bound that rules out self-intersection. 
In \cite{MuCa1} the incompressible Navier--Stokes equations are studied in a cylindrical wall which is moving in time. Its elastodynamics is modelled by the one-dimensional cylindrical linearised Koiter shell model. The authors apply a numerical approach and show the existence a of weak solution
based on semi-discrete operator splitting scheme. The elastodynamics of the cylinder wall in \cite{MuCa} is governed by the one-dimensional linear wave equation modelling the thin structural layer, and by the two-dimension equations of linear elasticity modelling the thick structural layer. In \cite{SchSro20} the analysis of uniqueness properties of weak-solution has been initiated in this field. There the authors show a weak-strong uniqueness result for elastic plates interacting with the incompressible Navier--Stokes equations.
%The compressible fluid-structure problem on the other hand, is studied in \cite{breitSchw2018compressible} yielding an analogous result to \cite{lengeler2014weak}. 
As far as we know, the only result on the analysis of weak solutions to fluid-structure interaction, where the original Koiter model (to be described below in Section \ref{subsec:model}) with a leading order nonlinear shell energy is considered, is the recent paper \cite{MuSc} by Muha and the second author. Results regarding the short-time existence of strong solutions can be found in \cite{ChSk}.\\
There are much less results concerning the compressible case.
In \cite{BrSc} the authors of the present paper showed the existence of a weak solution to the compressible Navier--Stokes equations coupled with a linear elastic shell of Koiter type. Eventually, a similar result has been shown by a time-stepping method~\cite{Tr}, where the interaction of a compressible fluid with a thermoelastic plate is studied (compare also with with the numeric results from~\cite{SchShe20}).
Results on the short-time existence of strong solutions for compressible fluid models coupled with one-dimensional linear elastic structures can be found in \cite{Ma,Mi}.
In \cite{Bo} the author studies an elastic structure (with a regularised elasticity law) which is immersed into a compressible fluid and proves the existence of weak solutions to the underlying system. 
Results concerning the long-time existence of weak solutions about structure interactions with heat conducting fluids are missing so far - even in the incompressible case. The existence of a unique local-in-time strong solution to compressible Navier--Stokes--Fourier system coupled with a damped linear
plate equation has been established very recently in~\cite{Ma2}.

\subsection{The model}
\label{subsec:model}
We consider the full Navier--Stokes--Fourier system of a heat-conducting compressible
fluid interacting with a nonlinear elastic Koiter shell in $\mathbb{R}^3$ of thickness $2\varepsilon_0 > 0$ (see \cite{Koi,Koi2} and also \cite{Ci1,Ci2}).  Here,  $\omega \subset \R^2$ can be associated to the middle surface of the shell and for simplicity, we take $\omega = \R^2\setminus \mathbb{Z}^2$ to be the flat torus.
Following \cite{ ciarlet2001justification} (see also \cite{MuSc} and \cite{BrMe}) we suppose that $\partial\Omega$ can be parametrised by an injective mapping ${\bfvarphi}\in C^4(\omega;\R^3)$ such that for all points $y=(y_1,y_2)\in \omega$, the pair of vectors  
%$\mathbf{a}_i(\by):= 
$\partial_i {\bfvarphi}(y)$, $i=1,2,$ are linearly independent. Simply put, $\bfvarphi$ is an injective map on the mid-section of the shell of the domain $\Omega$.
This vector pair $[\partial_1 {\bfvarphi}(y), \partial_2 {\bfvarphi}(y)]$ is the covariant basis of the tangent plane to the middle surface ${\bfvarphi}(\omega)$ of the reference configuration  at each point ${\bfvarphi}(y)$ and
\begin{align*}
{\nu}(y)=\frac{\partial_1 {\bfvarphi}(y) \times \partial_2 {\bfvarphi}(y)}{\vert \partial_1 {\bfvarphi}(y) \times \partial_2 {\bfvarphi}(y) \vert}
\end{align*}
is a well-defined unit vector  normal to the surface ${\bfvarphi}(\omega)$ at ${\bfvarphi}(y)$. We now assume that the shell (and in particular, its middle surface) only deforms along the normal direction with a displacement field $\eta {\nu} : \omega \rightarrow\R^3$ where $\eta : \omega \rightarrow\R$ is considerably smooth. Then, we can parametrized the deformed boundary by the coordinates
\begin{align}\label{eq:covar}
{\bfvarphi}_\eta( y)={\bfvarphi}(y) + \eta(y){\nu}(y), \quad \,y \in \omega,
\end{align}
yielding the deformed middle surface 
%$\Gamma_\eta(t):= 
${\bfvarphi}_\eta(\omega)$. 
%Now for
%\begin{align*}
%%\mathbf{a}_i(\eta):=
%\partial_i{\bfvarphi}_\eta(t, \bfy)= \partial_i {\bfvarphi}( \bfy) +  \partial_i \eta(t, \bfy) {\bfnu}(\bfy) +   \eta(t, \bfy) \partial_i{\bfnu}(\bfy), \quad i =1,2,
%\end{align*}
%The covariant components of the first fundamental form of  the deformed middle surface  $ {\bfvarphi}_\eta(t, \omega)$ is given by
The covariant components of the ``modified'' change of metric tensor $\mathbb{G}(\eta)$ are given by
\begin{align*}
%a_{ij}(\eta):= 
G_{ij}(\eta)=\partial_i{\bfvarphi}_\eta \cdot \partial_j {\bfvarphi}_\eta -\partial_i {\bfvarphi} \cdot \partial_j {\bfvarphi},
\end{align*}
where $\partial_i{\bfvarphi}_\eta \cdot \partial_j {\bfvarphi}_\eta$ are the covariant components of the first fundamental form of  the deformed middle surface  $ {\bfvarphi}_\eta(\omega)$.
%where
%\begin{align*}
%G_{ij}(\eta) &:= \partial_i \eta(t, \bfy)  \partial_j \eta(t, \bfy) + \eta(t, \bfy)\big[\partial_i {\bfvarphi}( \bfy) \cdot\partial_j  {\bfnu}(\bfy) + \partial_j{\bfvarphi}( \bfy) \cdot\partial_i {\bfnu}(\bfy) \big] 
%+ \eta^2(t, \bfy)\partial_i{\bfnu}(\bfy) \cdot \partial_j  {\bfnu}(\bfy)  
%\end{align*}
We denote by $\nu_\eta$ the normal-direction to the deformed middle surface  $ {\bfvarphi}_\eta(\omega)$ at the point ${\bfvarphi}_\eta(y)$ (which is in general not a unit vector). It is given by
\begin{align*}
{\nu}_\eta(y)
&=\partial_1 {\bfvarphi}_\eta(y) \times \partial_2 {\bfvarphi}_\eta(y)
%=
% {\bfnu}(\bfy) \big\vert \partial_1 {\bfvarphi}(\bfy) \times \partial_2 {\bfvarphi}(\bfy) \big\vert
% +
%  \partial_2 \eta(t,\bfy)\big(\partial_1 {\bfvarphi}(\bfy) \times {\bfnu}(\bfy) 
% \\&+
% \eta(t,\bfy) \partial_1{\bfnu}(\bfy) \times {\bfnu}(\bfy) \big)
%  +
%  \partial_1 \eta(t,\bfy)\big({\bfnu}(\bfy) \times \partial_2 {\bfvarphi}(\bfy)  +
% \eta(t,\bfy) {\bfnu}(\bfy) \times \partial_2{\bfnu}(\bfy) \big)
% \\&
%   +
%\eta(t,\bfy)\big(\partial_1{\bfvarphi}(\bfy) \times \partial_2  {\bfnu}(\bfy)  +
%\partial_1 {\bfnu}(\bfy) \times \partial_2{\bfvarphi}(\bfy) \big)
%+
%\eta^2(t,\bfy)\big( \partial_1 {\bfnu}(\bfy) \times \partial_2{\bfnu}(\bfy) \big)
\end{align*}
and
\begin{align*}
R_{ij}^\sharp(\eta) :=\frac{\partial_{ij} {\bfvarphi}_\eta\cdot {\nu}_\eta}{\vert \partial_1 {\bfvarphi} \times \partial_2 {\bfvarphi} \vert}
-
\partial_{ij} {\bfvarphi}\cdot {\nu} , \quad i,j=1,2,
\end{align*}
are the covariant components of the change of curvature tensor $\mathbb{R}^\sharp(\eta)$. The elastic energy $K(\eta):=K(\eta, \eta)$ of the deformation is then given by 
\begin{equation}
\begin{aligned}
\label{koiterEnergy}
K(\eta)&= \frac{1}{2}\epsilon_0 \int_\omega \mathbb{C}:\mathbb{G}(\eta) \otimes \mathbb{G}(\eta ) \dH
+
\frac{1}{6}\epsilon_0^3 \int_\omega \mathbb{C}:\mathbb{R}^\sharp(\eta ) \otimes\mathbb{R}^\sharp(\eta ) \dH
\\&
:=\sum_{i,j,k,l=1}^2 \frac{1}{2}\epsilon_0 \int_\omega C^{ijkl}G_{kl}(\eta )G_{ij}(\eta ) \dH
+
\frac{1}{6}\epsilon_0^3 \int_\omega C^{ijkl}{R}^\sharp_{kl}(\eta ){R}^\sharp_{ij}(\eta ) \dH
\end{aligned}
\end{equation} 
where $\mathbb{C}=( C^{ijkl})_{i,j,k,l =1}^2$ is a fourth-order tensor whose entries are the contravariant components of the shell elasticity, see \cite[Page 162]{Ci3}. We remark that for simplicity, we have normalized the measure $\dH$ in \eqref{koiterEnergy} which should have actually been the weighted measure $=\vert \partial_1 {\bfvarphi} \times \partial_2 {\bfvarphi} \vert \dH$ with the weight $\vert \partial_1 {\bfvarphi} \times \partial_2 {\bfvarphi} \vert$. 
Next, given the geometric quantity
\begin{align}
\begin{aligned}
\label{geomQuan}
\overline{\gamma}(\eta) :=
1&+
\frac{\eta}{\vert \partial_1 {\bfvarphi}\times \partial_2{\bfvarphi} \vert}
\Big[
{\nu} \cdot\big(  \partial_1 {\bfvarphi}\times \partial_2 {\nu}+\partial_1{\nu} \times \partial_2 {\bfvarphi} \big) 
\Big]\\&+
\frac{\eta^2}{\vert \partial_1 {\bfvarphi} \times \partial_2 {\bfvarphi} \vert}
{\nu}\cdot \big(\partial_1{\nu} \times \partial_2{\nu} \big),
\end{aligned}
\end{align}
one deduces the $W^{2,2}(\omega)$-coercivity of the Koiter energy \eqref{koiterEnergy} as long as $\gamma(\eta)\neq0$, cf. \cite[Lemma 4.3
and Remark 4.4]{MuSc}.
Finally, we remark that
the Koiter energy is continuous on $W^{2,p}(\omega)$ for all $p>\seb{\beta}>2$ due to the Sobolev embedding
$W^{2,p}(\omega)\hookrightarrow W^{1,\infty}(\omega)$.
\\
%As is derived in~\cite{LenRuz} we find that the shell described via $\eta$, which describes the deformation of an initial shell in normal direction, has the following bending force
%\begin{align}
%\label{k-EL}
%\partial_t^2\eta+m\Delta^2\eta+B\eta%\text{ in } I\times M.
%\end{align} 
For a given  function $\eta:I\times\omega\rightarrow \R$ with an interval $I=(0,T)$ we denote by $\Omega_{\eta(t)}$ the variable in time domain. With a slight abuse of notation we denote by $I\times\Omega_\eta=\bigcup_{t\in I}\set{t}\times\Omega_{\eta(t)}$ the deformed time-space cylinder, defined via its boundary
\[
\partial\Omega_{\eta(t)}=\set{{\bfvarphi}(y) + \eta(t,y){\nu}(y):y\in \omega}.
\]
 Along such a cylinder we observe the flow of a heat-conducting compressible fluid subject to the volume force $\bff:I\times \Omega_\eta\rightarrow\R^3$ and the heat source $H:I\times \Omega_\eta\rightarrow\R$. We seek the velocity field $\bu:I\times \Omega_\eta\rightarrow\R^3$, the density $\varrho:I\times \Omega_\eta\rightarrow\R$
 and the temperature $\vartheta:I\times \Omega_\eta\rightarrow\R$ solving the following system
%\begin{equation}
 \begin{align}
 \partial_t\varrho+\diver(\varrho\bu)&=0,&\text{ in } &I\times \Omega_\eta,\label{eq1}
 \\ \partial_t(\varrho\bu)+\diver(\varrho\bu\otimes\bu)&=\diver\bfS(\vt,\nabla\bfu)-\nabla p(\varrho,\vt)+{\varrho}\bff&\text{ in } &I\times\Omega_\eta,\label{eq2}
 \\
\nonumber\partial_t(\varrho e(\vr, \vt))+\Div(\varrho e(\vr, \vt) \bfu)&= \bfS(\vt, \Grad \vu) :\nabla\bfu-p(\vr, \vt) \Div\bfu\\&-\Div\bfq(\vt, \Grad \vt) + \vr H &\text{ in } &I\times\Omega_\eta,\label{eq2b}\\
 \bfu(t,x+\eta(t,x)\nu(x))&=\partial_t\eta(t,x)\nu(x)&\text{ in } &I\times \omega,\label{eq3a}
 \\
% \bfu&=0&\text{ on } &I\times \Gamma,\label{eq3aa}\\
 \partial_{\nu_\eta}\vt&=0&\text{ on } &I\times \Omega_\eta,\label{eq3aa'}
 \\
  \varrho(0)=\varrho_0,\quad (\varrho\bu)(0)&=\bq_0,\quad\vt(0)=\vt_0&\text{ in } &\Omega_{\eta_0}.\label{eq3}%&\qquad\text{ in }&\mt.\label{eq3}
 \end{align}
% Here, $p(\varrho)$ is the pressure which is assumed to follow the $\gamma$-law,  for simplicity $p(\varrho)=a\varrho^\gamma$, where $a>0$ and $\gamma>1$. Note that
In \eqref{eq2} we suppose Newton's rheological law
\begin{align*}
\bfS(\vt,\nabla\bfu)=\mu(\vt)\Big(\frac{\nabla\bfu+\nabla\bfu^T}{2}-\frac{1}{3}\Div\bfu\,\mathcal{I}\Big)+\lambda(\vt)\Div\bfu\,\mathcal{I}
\end{align*} 
with strictly positive viscosity coefficients $\mu,\,\lambda$
(see Remark 1.3 in \cite{BrSc} for the case $\lambda\geq0$).
The internal energy (heat) flux is determined by Fourier's law
\begin{align}\label{eq:fl}
\bfq(\vartheta,\nabla\vartheta)=-\varkappa(\vartheta)\nabla\vartheta=-\nabla\mathcal K(\vartheta),\quad \mathcal K(\vartheta)=\int_0^\vartheta\varkappa(z)\,\dd z
\end{align}
with strictly positive heat-conductivity $\varkappa$.
The thermodynamic functions $p$ and $e$ are related to the (specific) entropy $s$ through {Gibbs' equation}
\begin{align}\label{m97} \vartheta D s (\varrho, \vartheta) = D e (\varrho,
\vartheta) + p (\varrho, \vartheta) D \Big( \frac{1}{\varrho} \Big)
\quad \mbox{for all} \quad \varrho, \vartheta > 0.
\end{align}
The model case is given by
\begin{align*}
p(\varrho,\vartheta)=\varrho^\gamma+\varrho\vartheta+\frac{a}{3}\vartheta^4,\quad e(\varrho,\vartheta)=\frac{1}{\gamma-1}\varrho^{\gamma-1}+c_v\vartheta+a\frac{\vartheta^4}{\varrho},\quad
s(\varrho,\vartheta)=\frac{4a}{3}\frac{\vartheta^3}{\varrho}+\log (\vartheta^{c_v})-\log\varrho,
\end{align*}
for $a,c_v>0$ and $\gamma>1$.
 In view of Gibb's relation (\ref{m97}),
the internal energy equation \eqref{eq2b} can be rewritten in the form of the entropy balance
\begin{align}\label{eq:13'}
\partial_t(\varrho s(\vr,\vt))+\Div(\varrho s(\vr,\vt)\bfu)=-\Div\Big(\frac{\bfq(\vt,\nabla\vt)}{\vartheta}\Big)+\sigma+ \vr \frac{{H}}{\vt} 
\end{align}
with the entropy production rate
\begin{align}\label{eq:0202}
\sigma=\frac{1}{\vartheta}\Big(\bfS(\vt,\nabla\bfu):\nabla\bfu-\frac{\bfq(\vt,\nabla\vt)\cdot\nabla\vartheta}{\vartheta}\Big).
\end{align}
In the weak formulation \eqref{eq:13'} will be replaced by a variational inequality.\\
 The shell should response optimally with respect to the forces, which act on the boundary. Therefore we have
 \begin{align}\label{eq:4}
  \varepsilon_0\varrho_S\partial_t^2\eta+K'(\eta)&=
  g+\nu \cdot \bfF\quad\text{in}\quad I\times \omega,
\end{align}
where $\varrho_S>0$ is the density of the shell.
Here, $g:I\times \omega\rightarrow\R$ is a given force and $\bfF$ is given by
\begin{align*}
\bfF&:=\big(-\bftau\nu_\eta \big)\circ \bfvarphi_{\eta(t)}|\det D\bfvarphi_{\eta(t)}|,
\quad
\bftau:=\bfS(\nabla\bfu) -p(\varrho,\vt)\mathcal{I}. 
 \end{align*} 
Here, $\bfvarphi_{\eta(t)}:\omega\rightarrow\partial\Omega_{\eta(t)}$ is the change of coordinates from
\eqref{eq:covar} and $\bftau$ is the Cauchy stress. To simplify the presentation in \eqref{eq:4}
we will assume 
\[
\varepsilon_0\varrho_S=1
\]
 throughout the paper. 
We assume the following boundary and initial values for $\eta$
\begin{align}
\label{initial}
\eta(0,\cdot)=\eta_0,\quad \partial_t\eta(0,\cdot)=\eta_1\quad\text{in}\quad \omega,
% \label{initial}\eta=0,\quad\nabla \eta=0\quad\text{in}\quad\partial M,
\end{align}
where $\eta_0,\eta_1:\omega\rightarrow\R$ are given functions. 
Here, we assume that
\[
\text{Im}(\eta_0)\subset (a,b).
\]
In view of \eqref{eq3a} we have to suppose the compatibility condition\footnote{Note that the above condition is necessary for strong solutions only and hence is not effecting the rest of the paper. In particular, since the assumptions on the initial data are only in Lebesgue spaces, the compatibility condition is void. It does, however, implicitly appear in the construction of the Galerkin bases, where a smooth approximation of the initial values is considered.}
\begin{align}\label{eq:compa}
\eta_1(y)\nu(y)=\frac{\bfq_0}{\varrho_0}(y+\eta(y)\nu(y))\quad\text{in}\quad \omega. 
\end{align}

%By the canonical extension of $\eta,\partial\eta$ by $0$ to $\partial\Omega$ we can unify \eqref{eq3a} and \eqref{eq3aa} to
%\begin{align}
%\label{boundary:unified}
% \bfu(t,y+\eta(t,y)\nu(y))=\partial_t\eta(t,y)\nu\quad\text{on}\quad I\times \partial\Omega.
%\end{align}
Our main result is the following existence theorem. The system \eqref{eq1}--\eqref{eq:compa} can be written in a natural way as a weak solution. The concept is introduced in the next section, (see \eqref{eq:apufinal}--\eqref{EI20final}), where also the precise formulation of our main result is presented (see Theorem \ref{thm:MAIN}). It is concerned with the existence of a weak solution up to degeneracy of the geometry and reads in a simplified version as follows.
 
\begin{theorem} \label{thm:MAINa}
Under natural assumptions on the data there exists a weak solution $(\eta,\bfu,\varrho,\vartheta)$ to \eqref{eq1}--\eqref{initial} with satisfies the energy balance
\begin{align}\label{eq:apriori0}
\begin{aligned}
\mathcal E(t) &=
\mathcal E(0) + \int_{\Omega_\eta} \vr H  \dx+\int_{\Omega_{\eta}}\varrho\bff\cdot\bfu\dx+\int_\omega g\,\partial_t\eta\dH,\\
\mathcal E(t)&= \int_{\Omega_\eta(t)}\Big(\frac{1}{2} \varrho(t) | {\bf u}(t) |^2 + \varrho(t) e(\varrho(t),\vartheta(t))\Big)\dx+\int_\omega \frac{|\partial_t\eta(t)|^2}{2}\,\dH+ K(\eta(t)).\end{aligned}
\end{align} 
The interval of existence is of the form $I=(0,t)$, where $t<T$ only in case $\Omega_{\eta(s)}$ approaches a self-intersection when $s\to t$ or the Koiter energy degenerates (namely, if $\lim_{s\to t}\overline{\gamma}(s,y)=0$ for some point $y\in \omega$).
\end{theorem}
The function space of existence for a weak solution to \eqref{eq1}--\eqref{initial} is determined by the total energy $\mathcal E$ in 
\eqref{eq:apriori0} as well as the quantity $\sigma$ in \eqref{eq:0202} taking into account the variable domain. Theorem \ref{thm:MAIN} extends the results from \cite{BrSc} to the case of a heat-conducting fluid but also applies to nonlinear structure equations.
As in the case of fixed domains studied in \cite{F} (see also \cite{DuFe} and \cite{FN}) the heat-conducting model allows (different to the isentropic equations) the striking feature of an energy equality. Energy, which is lost by dissipation, is transfered into heat, cf. \eqref{eq2b}.

%\begin{remark}[Elastic solids vs.\ damped elastic solids]
% Equation \eqref{eq:4} implies that we consider a perfectly elastic solid. In particular, no damping is assumed and hence dissipative effects are not assumed to act on the solid. Consequently, no effects stemming from the solid are appearing in the entropy relation. While this decoupling might seem like an advantage at first glance the opposite is actually true. Indeed, a damped solid naturally yields better a priori estimates for the solid evolution. As a consequence the damping effects would simplify the convergence analysis of the energy equality in the latter case.
%
%In conclusion, the strategy developed here assumes with no loss of generality that the solid has no damping effect. Indeed, it seems to be perfectly applicable for damped shells or plates as well. These damping terms would, however, create naturally non-homogeneous Neumann boundary values for the temperature since the motion of the solid could produce heat.\footnote{The Neumann boundary values would naturally be preserved in the limit passage and lead to a respective inequality for the weak solutions considered here.}
%\end{remark}

\subsection{Mathematical strategy}
In this paragraph we provide an overview of the developed methodologies. Further we aim to explain all technical novelties and their potential significance.\\
As is common in the existence theory for weak solutions, the first step is to understand how to prove sequential compactness. Let us assume there is a given sequence of weak solutions $(\eta_n,\bfu_n,\varrho_n,\vartheta_n)$ to \eqref{eq1}--\eqref{initial} possessing suitable regularity properties. Deriving a priori estimates using the entropy balance one can control, in addition to the total energy defined in \eqref{eq:apriori0}, first order spatial derivative of
$\bfu_n$ and $\vartheta_n$
using \eqref{eq:13'} and \eqref{eq:0202}.
Unlike in the steady domain case, these estimates are not sufficient to show that a subsequence is again converging to a solution. One problem is to derive energy equality (that is expected for closed systems like the Navier--Stokes--Fourier equations considered here).  Critical are the {\em kinetic and elastic part of the solid energy}. 
%However,
%energy equality seems currently out of reach even for fixed domains. In contrast to this,
%the full system \eqref{eq1}--\eqref{initial} is energetically closed. So, we aim at energy equality. 
To prove their compactness which does not follow from the energy estimates. In fact, the functional $K$ is not even well-defined on $W^{2,2}(\omega)$ recalling the discussion from Section \ref{subsec:model}. Our strategy is to
derive fractional estimates for $\nabla^2\eta_n$ as a consequence of a testing procedure for \eqref{eq:4} with difference quotients. Testing the shell equation with suitable test-functions requires in the weak formulation to choose an appropriate test-function for the full momentum equation as well.
Technically, this means we have to ``extend'' functions defined on $\omega$ to functions defined on the time dependent domain $\Omega_{\eta_n}$. An obstacle here is that the pressure is only expected to belong to $L^1$ in space near the moving boundary (compare with~\cite{BrSc}).\footnote{As is explained in~\cite{BrSc} the usual test with the \Bogovskii-operator (that implies higher integrability of the density) fails and we are only able to prove uniform integrability, cf. Lemmas \ref{prop:higherb} and \ref{prop:higher'}.}
%In a first step we prove
%  prove strong convergence of $\partial_t\eta_n$ in $L^2$. Eventually,
%  we aim to test the equation by $\eta_n$ to pass to the limit
%  in $K(\eta_n)$.
%A standard extension fails since the pressure only belongs to $L^1$ globally (the usual test with the \Bogovskii-operator fails and we are only able to prove uniform integrability, cf. Lemmas \ref{prop:higherb} and \ref{prop:higher'})
%and $\nabla\eta_n$ fails being bounded (as we only have estimates on $W^{2,2}$ at hand).
To circumvent the irregularity of the pressure we work with a solenoidal extension $\mathscr F^{\Div}_{\eta_n}$ that was recently constructed in~\cite{MuSc} (compare also with~\cite{LeRu}).\\
A second related problem is the strong convergence of $\partial_t\eta_n$ (which is a part of the kinetic energy). Here we use a modified version of the classical argument by Aubin-Lions. Non-standard are uniform continuity estimates in time of the underlying sequence, which rely on the weak coupled momentum equation.  
Again a carefully chosen test-function is needed.
 Here, however, we use an extension which has (different to $\mathscr F^{\Div}_{\eta_n}$) a regularizing effect but no solenoidality is needed. What turns out to be the most sensitive point is that the extension is depending on the variable geometry. In particular, the extension of a constant in time function still possesses a non-trivial time-derivative. The essential term is
\begin{align*}
\int_I\int_{\Omega_{\eta_n}}\varrho_n\bfu_n\cdot \partial_t(\mathscr F_{\eta_n}b)\dxt
\end{align*}
using the notation from the next section.
We observe that $\partial_t(\mathscr F_{\eta_n}b)$ (the time-derivative of the extension) is expected to behave like $\partial_t\eta_n$.
Based on the a priori estimates $\vr_n\in L^\infty_t(L^\gamma_x)$, $\bfu_n\in L^2_t(L^6_x)$, we find that $\partial_t\eta_n\in L^2_t(L^{r}_x)$ for all $r<4$ uniformly by the trace theorem (see~Lemma~\ref{lem:2.28}).  Consequently the bound $\gamma>\frac{12}{7}$ naturally appears. It is interesting to note that
the same bound was needed in \cite[Lemma 7.4]{BrSc} in order to avoid concentrations of the approximate pressure at the boundary (an argument that we will use later in Lemma~\ref{prop:higher'}).\\
%The approach just explained has another striking feature compared to \cite{BrSc}: we are able to consider fully nonlinear shell models.\\
In order to prove Theorem \ref{thm:MAIN} we have to work with a multi-layer approximation scheme.
As is nowadays standard in the theory of compressible fluids
we follow \cite{feireisl1} and use an artificial pressure (replace $p(\varrho,\vt)$ by
$p_\delta(\varrho,\vt)=p(\vr,\vt)+\delta\varrho^\beta$ where $\beta$ is chosen large enough) as well as an artificial viscosity (add $\varepsilon\Delta\varrho$ to the right-hand side of \eqref{eq1}). 
The resulting system is solved by means of a Galerkin approximation. More specifically,
we have to solve a finite-dimensional system of ODEs and eventually pass to the limit in the dimension $N$.
%For technical reasons the three limit passages $N\rightarrow\infty$, $\varepsilon\rightarrow0$ and $\delta\rightarrow0$ have to be performed separately (see Sections \eqref{sec:3} and \ref{sec:5} and \ref{sec:6}). 
It turns out that existence on the basic level, where the parameters $\ep$ and $\delta$ are fixed, is quite involving. Troublesome is the derivation of the entropy balance \eqref{eq:13'} (in form of a variational inequality): Though it is suitable to pass to the limit it is not appropriate for the direct construction of solutions due to its highly involving non-linearities.
%\footnote{We explain is not possible to start directly with the entropy balance on the regularized level. After the limit procedure it only yields an inequality which has the wrong sign to recover even an energy inequality. The necessary a priori estimates can only be obtained by combining energy and entropy balance to the total dissipation balance (see Subsection \ref{subsec:teb}).} 
Hence the entropy balance is derived a posteriori by dividing the internal energy equation \eqref{eq2b} by $\vt$. In order to do this rigorously it has to be shown that the temperature is strictly positive - a property which can only be expected from strong solutions to \eqref{eq2b}. One of the main efforts of this paper is consequently to construct strong solutions to \eqref{eq2b} for regularized velocity and smooth pressure. New a priori estimates for \eqref{eq2b} and \eqref{eq1} in variable domains are shown that go well beyond the results from~\cite[Sec. 3]{BrSc} and form one if the main achievements of this paper.
%Several difficulties studying an equation with such a complicated structure in a moving domain are noteworthy.
%Obsevre, that \eqref{eq2b} is not a parabolic equation in $\vt$ but rather in $\vt^4$. However, the elliptic part becomes degenerate from this point of view.\\
%We remark that it is not possible to start directly with the entropy balance on the regularized level. After the limit procedure it only yields an inequality which has the wrong sign to recover even an energy inequality. The necessary a priori estimates can only be obtained by combining energy and entropy balance to the total dissipation balance (see Subsection \ref{subsec:teb}).\\
%Having understood the strict positivity of the temperature is the  the regularized system
%in Section \ref{sec:3}.
%The Galerkin solution itself is produced by decoupling
%We approximate the latter one by a finite dimensional problem which can be solved by means of a fixed point argument (here the shell and the convective terms are decoupled).
Finally, we wish to note that we can shorten the approach from \cite{BrSc} considerably. Different to \cite{BrSc} we decouple the geometry from the fluid system on the Galerkin level and apply the fixed point argument to the resulting semi-discrete problem directly. This allows to remove one regularization level in which the moving boundary and the convective terms are regularised. % by a parameter $\kappa$.
% This differs from the approach (in \cite{LeRu}
%and \cite{BrSc} the is applied on the continuous level and the decoupled system is solved by a Galerkin approximation). The formulation of the Galerkin approximation in our case is more involved since the basis functions are defined on the a priori unknown moving domain. The fixed point argument on the Galerkin level is, however, much easier. Also, the $\nabla\bfu$-term on the right-hand side of \eqref{eq2b} is regular on this level hence the temperature strictly positive.

\subsection{Overview of the paper} 
In Section \ref{sec:2} we present basics concerning variable domains as well as the functional analytic set-up. In its last subsection the concept of weak solutions for the coupled system and the main theorem are introduced. The preliminary section is rather significant. Indeed, many standard tools of the analysis need an appropriate adaptation to the variable geometry set-up, as well as to the particular non-linear coupling of the PDE system. In particular, in Subsection~\ref{sec:ext} we introduce two different extension operators that are needed for the analysis performed later. 
 In Section \ref{sec:3a} we study the (regularized) continuity equation as well as the (regularized) internal energy equation in a time dependent domain. These are non-trivial extensions from the analysis presented in \cite[Section 3]{BrSc}. In particular, we provide regularity estimates and minimum and maximum principles. 
 Section \ref{sec:3} is dedicated to the construction of an approximate solution. Different to previous fixed point approaches (see e.g.~\cite{BrSc} and \cite{LeRu}) we construct a fixed point on the Galerkin level which we believe to be appropriate also for future applications. A further achievement is the derivation of the entropy inequality which sensitively relies on Section~\ref{sec:3a}.
 Finally, in Section \ref{sec:5} the two limit passages $\varepsilon\rightarrow0$ and $\delta\rightarrow0$ are performed which leads to the proof of Theorem~\ref{thm:MAIN} and the existence of a weak solution is shown. 
 Of particular importance is here Subsection~\ref{subsec:shell} where the derivation of an {\em energy equality} is performed. Critical is the strong convergence of the elastic energy of the solid deformation. Here we adapt a recent regularity argument for the shell displacement derived in \cite{MuSc}.  As shown in \cite{MuSc} these estimates are crucial to involve non-linear Koiter shell laws in the weak existence theory for incompressible fluids. In the here considered Navier-Stokes-Forier system the regularity is needed even for linear shell models. Since, even for linear Koiter shell models an energy equality cannot be derived without additional regularity estimates and the related compactness properties.

\section{Preliminaries}
\label{sec:2} 
\subsection{Structural and constitutive assumptions}
\label{H}

We impose several restrictions on the specific shape of the thermodynamic functions
$p=p(\vr,\vt)$, $e=e(\vr,\vt)$ and $s=s(\vr,\vt)$ which are in line with Gibbs' relation (\ref{m97}). 
We consider the pressure $p$ in the form
\begin{align}\label{m98}
p(\vr, \vt) = p_M(\vr) + p_R(\vr,\vt),\quad
p_M(\vr) = \vr^\gamma+\varrho\vartheta,\quad p_R(\vr,\vt)=\frac{a}{3} \vt^4, \ a > 0,
\end{align}
the specific internal energy
\begin{equation} \label{mp8a}
e(\vr, \vt) = e_M(\vr) + e_R(\vr,\vt),\quad e_M(\vr) =\frac{1}{\gamma-1}\vr^{\gamma-1}+c_v\vartheta,\quad e_R(\vt,\vr)={a} \frac{\vt^4}{\vr},\ c_v>0,
\end{equation}
and the specific entropy
\begin{equation}\label{md8!}
s(\vr, \vt) = \frac{4a }{3} \frac{\vt^3}{\vr}+\log (\vartheta^{c_v})-\log\varrho.
\end{equation}
This is model case for the set-up in \cite[Chapter 1]{F}, to which we refer
for the physical background and the relevant discussion.

%In fact, in \cite[Chapter 1]{F}
%a weak temperature dependence of $p_M$ is allowed which vanishes asymptotically
%such that  $p_M(\vr, \vt)\sim \varrho^\gamma$ for large $\varrho$.\\
%\begin{equation}\label{md8}
%S = S(Z),\
%S'(Z)=-\frac{1}{\gamma-1} \frac{\gamma P(Z)-ZP'(Z)}{Z^2}<0,
%\end{equation}
%where
%\begin{equation}\label{m103-}
%P \in C^1[0, \infty) \cap C^2(0, \infty),\; P(0)= 0,\;P'(Z)>0, \ \mbox{for all}\ {Z \geq 0},
%\end{equation}
%\begin{equation}\label{md7}
%0<\frac{1}{\gamma-1} \frac{\gamma P(Z)-ZP'(Z)}{Z}  <c,\; \mbox{for all}\ Z > 0,
%\end{equation}
%and
%\begin{equation}\label{md1}
%\lim_{Z\to\infty}\frac {P(Z)} {Z^{\gamma}}=p_\infty>0.
%\end{equation}
%\\
The viscosity coefficients $\mu$, $\lambda$ are continuously
differentiable functions of the absolute temperature $\vt$, more
precisely $\mu, \ \lambda  \in C^1([0,\infty))$, satisfying
\begin{align}\label{m105}  \underline{\mu}(1 + \vartheta) \leq
\mu(\vartheta) \leq \overline{\mu}(1 + \vartheta), \end{align}
\begin{align}\label{*m105*} \sup_{\vartheta\in [0, \infty)}\big(|\mu'(\vartheta)|+|\lambda'(\vartheta)|\big)\le
\overline m, \end{align}
\begin{align}\label{m106}  \underline{\lambda}(1 +
\vartheta) \leq \lambda (\vartheta) \leq \overline{\lambda}(1 +
\vartheta), \end{align}
with positive constants $\underline{\mu},\overline{\mu},\overline{m},\underline{\lambda},\overline{\lambda}$.
The heat conductivity coefficient $\varkappa \in C^1[0, \infty)$ satisfies
\begin{align}\label{m108} 0 < \underline{\varkappa} (1 + \vartheta^3) \leq \varkappa(
\vartheta) \leq \overline{\varkappa} (1 + \vartheta^3)
\end{align}
with some positive $\underline{\varkappa},\overline{\varkappa}$.
We introduce the following regularizations
\begin{align} \label{neu}
\begin{aligned}
&p_\delta(\vr,\vt):= p_{R}(\vr,\vt)+p_{M,\delta}(\vr), \quad 
p_{M,\delta} (\vr):=p_M(\vr) + \delta\vr^\beta,
\\
&e_\delta(\vr,\vt)=:= e_{R}(\vr,\vt)+e_{M,\delta}(\vr), \quad 
e_{M,\delta} (\vr):=e_M(\vr) + \frac{\delta}{\beta-1}\vr^{\beta-1}, 
\\
&\varkappa_\delta(\vartheta )=\varkappa(\vartheta )+\delta\big(\vt^\beta+\frac{1}{\vartheta}\big),\quad \mathcal{K}_\delta(\vt) = \int_0^\vt \varkappa_\delta(z) \ {\rm d}z,
\\
&\bfS^{\ep}(\vt,\nabla\bfu)=\bfS(\vt,\nabla\bfu)+\ep(1+\vt)|\nabla\bfu|^{p-2}\nabla\bfu,
\end{aligned}
\end{align}
for some $p>\beta>2$.

%\seb{\begin{remark}\label{rem:pressurelaw}
%The pressure law in \eqref{m98} is a simplification of the one
%in \cite{F}, where the molecular pressure $p_M$ contains the additional term $\vr\vt$.
%This term is important for physical reasons and, in particular, implies that the temperature is positive (see also Remark \ref{rem:korn}).
%The simpler structuer allows to write the internal energy equation as an equation for a function of the temperature (using that $\vt e_R(\vt,\vr)=a\vt^4$ by \eqref{mp8a}). This is in our regularity analysis in Section \eqref{sec:int} which strongly differs from \cite{F} due ot the moving domain.
%We currently do not know how to construct a solution at the basic level (see Section \ref{sec:3}) without this restriction. On the other hand,
%the sequential compactness of solutions can be proved also for the pressure law from \cite{F}.
%\end{remark}}

\subsection{Function spaces on variable domains}
\label{ssec:geom}
 The spatial domain $\Omega$ is assumed to be a non-empty bounded subset of $\mathbb{R}^3$ with $C^4$-boundary and an outer unit normal ${\nu}$. We recall from Section \ref{subsec:model} that we assume that
 $\partial\Omega$ can be parametrised by an injective mapping ${\bfvarphi}\in C^4(\omega;\R^3)$ such that for all points $y=(y_1,y_2)\in \omega$, the pair of vectors  
%$\mathbf{a}_i(\by):= 
$\partial_i {\bfvarphi}(y)$, $i=1,2,$ are linearly independent. For a point $x$ in the neighbourhood
or $\partial\Omega$ we can define
\begin{align*}
 y(x)=\arg\min_{y\in\omega}|x-\bfvarphi(y)|,\quad s(x)\text{ is defined such that }s(x)\nu(y(x))+y(x)=x.
 \end{align*}
Moreover, we define the projection $\bfp(x)=\bfvarphi(y(x))$. We define $L>0$ to be the largest number such that $s,y$ and $\bfp$ are well-defined on $S_L$, where
\begin{align}
\label{eq:boundary1}
S_L=\{x\in\R^3:\,\mathrm{dist}(x,\partial\Omega)<L\},
\end{align}
see also Remark \ref{rem:simple} in connection with this.
We remark that due to the $C^2$ regularity of $\Omega$ for $L$ small enough we find that $\abs{s(x)}=\min_{y\in\omega}|x-\bfvarphi(y)|$ for all $x\in S_L$. This implies that $S_L=\{s\nu(y)+y:(s,y)\in [-L,L]\times \omega\}$.
%
%Moreover, it has been shown, that if $\overline{\gamma}(\eta_0)\neq 0$, then there exists a $L>0$ such that for all $\eta\in H^{2,2}(\omega)
%\begin{align}
%\label{eq:boundary2}
%S_L=\{x\in\R^3:\,\mathrm{dist}(x,\partial\Omega)<L\}.
%\end{align}
For a given function $\eta : I \times \omega \rightarrow\R$ we parametrise the deformed boundary by
\begin{align*}
{\bfvarphi}_\eta(t,y)={\bfvarphi}(y) + \eta(t,y){\nu}(y), \quad \,y \in \omega,\,t\in I,
\end{align*}
and the deformed space-time cylinder $I\times\Omega_\eta=\bigcup_{t\in I}\set{t}\times\Omega_{\eta(t)}$ through
$$\partial\Omega_{\eta(t)}=\set{{\bfvarphi}(y) + \eta(t,y){\nu}(y):y\in \omega}.$$
The corresponding function spaces for variable domains are defined as follows.
\begin{definition}{(Function spaces)}
For $I=(0,T)$, $T>0$, and $\eta\in C(\overline{I}\times\omega)$ with $\|\eta\|_{L^\infty_{t,x}}< L$ we set $I\times\Omega_\eta:=\bigcup_{t\in I}\{t\}\times\Omega_{\eta(t)} \subset\R^4$. We define for $1\leq p,r\leq\infty$
\begin{align*}
L^p(I;L^r(\Omega_\eta))&:=\big\{v\in L^1(I\times\Omega_\eta):\,\,v(t,\cdot)\in L^r(\Omega_{\eta(t)})\,\,\text{for a.e. }t,\,\,\|v(t,\cdot)\|_{L^r(\Omega_{\eta(t)})}\in L^p(I)\big\},\\
L^p(I;W^{1,r}(\Omega_\eta))&:=\big\{v\in L^p(I;L^r(\Omega_\eta)):\,\,\nabla v\in L^p(I;L^r(\Omega_\eta))\big\}.
\end{align*}
\end{definition}
For various purposes it is useful to relate the time dependent domains and the fixed domain. This can be done by the means of the Hanzawa transform. Its construction can be found in
\cite[pages 210, 211]{LeRu}. Note that variable domains in \cite{LeRu} are defined via functions $\zeta:\partial\Omega\rightarrow\R$ rather than functions $\eta:\omega\rightarrow\R$ (clearly, one can link them by setting
$\zeta=\eta\circ\bfvarphi^{-1}$). For any  $\eta:\omega\rightarrow(-L,L)$ we define the Hanzawa transform $\bfPsi_\eta: \Omega  \to \Omega_\eta$
by
\begin{align}
\label{map}
\begin{aligned}
\bfPsi_\eta(x)&= \begin{cases}\bfp(x)+\Big(s(x)+\eta(y(x))\phi(s(x))\Big)\nu(y(x)),&\text{ if } \dist(x,\partial\Omega)<L,\\
\quad x,\quad &\text{elsewhere}
\end{cases}.
\end{aligned}
\end{align}
Here $\phi\in C^\infty\big((-\frac{3L}{4},\infty),[0,1]\big)$ is such that 
$\phi\equiv 0$ in $[-\frac{3L}{4},-\frac{L}{2}]$ and $\phi\equiv 1$ in $[-\frac{L}{4},\infty)$.
Due to the size of $L$, we find that $\bfPsi_\eta$ is a homomorphism such that $\bfPsi_\eta|_{\Omega\setminus S_L}$ is the identity. Moreover, $\eta\in C^k(\omega)$ for $k\in\N$ implies that $\bfPsi_\eta$ is a $C^k$-diffeomorphism.

%\begin{lemma}
%\label{lem:diffeo}
%Let $\eta:\omega\rightarrow(-L,L)$ be a continuous function.
%\begin{itemize}
%\item[a)] There is a homomorphism $\bfPsi_\eta:\overline\Omega\rightarrow\overline{\Omega}_\eta$ such that $\bfPsi_\eta|_{\Omega\setminus S_L}$ is the identity.
%\item[b)] If $\eta\in C^k(\omega)$ for $k\in\{1,2,3\}$ then $\bfPsi_\eta$ is a $C^k$-diffeomorphism.
%\end{itemize}
%\end{lemma}
We collect a few properties of the above mapping $\bfPsi_\eta$.
\begin{lemma}
\label{lem:sobdif} Let $1<p\leq\infty$ and $\sigma\in (0,1]$.
\begin{itemize}
\item[a)]  If $\eta\in W^{2,2}(\omega)$ with $\|\eta\|_{L^\infty_x}<L$, then the linear mapping $\bfv\mapsto \bfv\circ \bfPsi_\eta$ ($\bfv\mapsto \bfv\circ \bfPsi_\eta^{-1}$) is continuous from $L^p(\Omega_\eta)$ to $L^r(\Omega)$ (from $L^p(\Omega)$ to $L^r(\Omega_\eta)$) for all $1\leq r<p$. 
\item[b)]If $\eta\in W^{2,2}(\omega)$ with $\|\eta\|_{L^\infty_x}<L$, then the linear mapping $\bfv\mapsto \bfv\circ \bfPsi_\eta$ ($\bfv\mapsto \bfv\circ \bfPsi_\eta^{-1}$) is continuous from 
$W^{1,p}(\Omega_\eta)$ to $W^{1,r}(\Omega)$ (from $W^{1,p}(\Omega)$ to $W^{1,r}(\Omega_\eta)$) for all $1\leq r<p$.
\item[c)] If $\eta\in C^{0,1}(\omega)$with $\|\eta\|_{L^\infty_x}<L$, then the linear mapping $\bfv\mapsto \bfv\circ \bfPsi_\eta$ ($\bfv\mapsto \bfv\circ \bfPsi_\eta^{-1}$) is continuous from 
$W^{\sigma,p}(\Omega_\eta)$ to $W^{\sigma,p}(\Omega)$ (from $W^{\sigma,p}(\Omega)$ to $W^{\sigma,p}(\Omega_\eta)$).
\item[d)] If $\eta\in W^{2,2}(\partial\Omega)$with $\|\eta\|_{L^{\infty}_x}<L$, then the linear mapping $\bfv\mapsto \bfv\circ \bfPsi_\eta$ ($\bfv\mapsto \bfv\circ \bfPsi_\eta^{-1}$) is continuous from 
$W^{\sigma,p}(\Omega_\eta)$ to $W^{\theta,r}(\Omega)$ (from $W^{\sigma,p}(\Omega)$ to $W^{\theta,r}(\Omega_\eta)$) for all $\theta\in(0,\sigma)$ and all $1<r<p$.
\end{itemize}
The continuity constants depend only on $\Omega,p,r,\sigma,\theta$, the respective norms of $\eta$.
% and $\tau(\eta)$.
\end{lemma}

The following lemma is a modification of \cite[Cor. 2.9]{LeRu}.
\begin{lemma}\label{lem:2.28}
Let $1<p<3$, $\sigma\in(\frac{1}{p},1]$ and $\eta\in W^{2,2}(\omega)$ with $\|\eta\|_{L^\infty_x}<L$. 
The linear mapping
$\tr_\eta:v\mapsto v\circ\bfPsi_\eta\circ\bfvarphi|_{\partial\Omega}$ is well defined and continuous from $W^{\sigma,p}(\Omega_\eta)$ to $W^{\sigma-\frac1r,r}(\omega)$ for all $r\in (\frac{1}{\sigma},p)$ and well defined and continuous from  $W^{\sigma,p}(\Omega_\eta)$ to $L^{q}(\omega)$ for all $1<q<\frac{2p}{3-\sigma p}$.
The continuity constants depend only on $\Omega,p,\sigma,$ and $\|\eta\|_{W^{2,2}_x}$.
\end{lemma}
%The following lemma allows us o extend functions defined on the variable domain
%to the whole space $\R^3$. This is not trivial for $\eta\in W^{2,2}(\partial\Omega)$ because the boundary is not Lipschitz continuous. However, it requires the additional assumption $\|\eta\|_{L^\infty(\partial\Omega)}<\frac{L}{2}$.
%\begin{lemma}
%\label{lem:extension}
%Let $1\leq r<p<\infty$ and $\eta\in W^{2,2}(\partial\Omega)$ with $\|\eta\|_{L^\infty(\partial\Omega)}<\frac{L}{2}$. There is a continuous linear operator $\mathscr E_\eta:W^{1,p}(\Omega_\eta)\rightarrow W^{1,r}(\R^3)$ such that $\mathscr E_\eta\big|_{\Omega_\eta}=\mathrm{Id}$.
%\end{lemma}

\begin{remark}
If $\eta\in L^\infty(I;W^{2,2}(\omega))$ we obtain non-stationary variants of the results stated above.
\end{remark}
It will be convenient for our purposes to extend $\bfPsi_\eta$, originally defined only
on
 $$\Omega_{(L-\eta)_+}=\Omega\cup \{x\in S_L:s(x)<\min\{L,L-\eta(y(x))\}\}),$$
to $\Omega_L=\Omega\cup S_L$ by setting
\begin{align*}
%\label{mapL}
\begin{aligned}
\overline\bfPsi_\eta(x)&= \begin{cases}\bfp(x)+\Big(s(x)+\eta(y(x))\phi(s(x))\Big)\nu(y(x)),&\text{ if } \dist(x,\partial\Omega)<L,\ s(x)+\eta(\bfp(x))< L,
\\
\quad x,\quad &\text{elsewhere.}
\end{cases}
\end{aligned}
\end{align*}
All the above statements are also true for $\bfv\mapsto \bfv\circ\overline{\bfPsi}_\eta$ and $\bfv\mapsto \bfv\circ\overline{\bfPsi}_\eta^{-1}$ on their respective domains.

\subsection{Extensions on variable domains}\label{sec:ext}
Since $\Omega$ is assumed to be sufficiently smooth, it is well-know that there is an extension operator $\mathscr F_\Omega$ which extends functions from $\partial\Omega$ to $\R^3$ and satisfies
\begin{align*}
\mathscr F_\Omega:W^{\sigma,p}(\partial\Omega)\rightarrow W^{\sigma+1/p,p}(\R^3)
\end{align*}
for all $p\in(1,\infty)$ and $\sigma\in[0,1]$, all as well as $\mathscr F_\Omega v|_{\partial\Omega}=v$. Now we define $\mathscr F_\eta$ by 
\begin{align}\label{eq:2401b}
\mathscr F_\eta b=\mathscr F_\Omega ((b\nu)\circ\bfvarphi^{-1})\circ\overline{\bfPsi}_\eta^{-1},\quad b\in W^{\sigma,p}(\omega),
\end{align}
where $\bfvarphi$ is the $C^4$-function in the parametrisation of $\Omega$.
If $\eta$ is smooth $\mathscr F_\eta$ behaves as a classical extension by Lemma \ref{lem:sobdif}.
The following properties can all be easily derived from the formulas
\begin{align*}
\nabla\mathscr F_\eta b&=\nabla\mathscr F_\Omega ((b\nu)\circ\bfvarphi^{-1})\circ\overline{\bfPsi}_\eta^{-1}\nabla\overline{\bfPsi}_\eta^{-1},\\
\nabla^2\mathscr F_\eta b&=\nabla^2\mathscr F_\Omega ((b\nu)\circ\bfvarphi^{-1})\circ\overline{\bfPsi}_\eta^{-1}\nabla\overline{\bfPsi}_\eta^{-1}\nabla\overline{\bfPsi}_\eta^{-1}+\nabla\mathscr F_\Omega ((b\nu)\circ\bfvarphi^{-1})\circ\overline{\bfPsi}_\eta^{-1}\nabla^2\overline{\bfPsi}_\eta^{-1},\\
\partial_t\mathscr F_\eta b&=\nabla\mathscr F_\Omega ((b\nu)\circ\bfvarphi^{-1})\circ\overline{\bfPsi}_\eta^{-1}\partial_t\overline{\bfPsi}_\eta^{-1},
\end{align*}
{where $\nabla\overline{\bfPsi}_\eta^{-1}$, $\nabla^2\overline{\bfPsi}_\eta^{-1}$ and $\partial_t\overline{\bfPsi}_\eta^{-1}$ behave as $\nabla\eta$, $\nabla^2\eta$ and $\partial_t\eta$ respectively.}
\begin{lemma}\label{lem:3.8}
Let $\eta\in C^{0,1}(\omega)$ with $\|\eta\|_{L^\infty_x}<\alpha<L$. 
\begin{enumerate}
\item[(a)] The operator
$\mathscr F_\eta$ defined in \eqref{eq:2401b} satisfies for all $p\in[1,\infty)$ and $\sigma\in[0,1]$
\begin{align*}
\mathscr F_\eta: W^{\sigma,p}(\omega)\rightarrow W^{\sigma+1/p,p}(\Omega\cup S_\alpha)
\end{align*}
and $\tr_\eta\mathscr F_\eta b=b\nu$ for all $b\in W^{1,p}(\omega)$. In particular, we have
\begin{align*}
\|\mathscr F_\eta b\|_{W^{\sigma+1/p,p}(\Omega\cup S_\alpha)}\leq\,c\,\|b\|_{W^{\sigma,p}(\omega)}
\end{align*}
for all $b\in W^{1,p}(\omega)$,
where the constant $c$ depends only on $\Omega,p,\sigma$, $\|\nabla\eta\|_{L^\infty_x}$ and $L-\alpha$.
\item[(b)]
If $p=\infty$ we have
\begin{align*}
\|\mathscr F_\eta b\|_{W^{1,\infty}(\Omega\cup S_\alpha)}\leq\,c(1+\|\nabla\eta\|_{L^\infty(\omega)})\,\|b\|_{W^{1,\infty}(\omega)}
\end{align*}
for all $b\in W^{1,\infty}(\omega)$, where $c$ depends only on $\Omega,p$ and $L-\alpha$.
\end{enumerate}
\end{lemma}
\begin{corollary}\label{cor:2807}
Let $\eta\in C^1(\overline I\times\omega)$ with $\|\eta\|_{L^\infty_x}<\alpha<L$. Then we have for all $q<\infty$
\begin{align*}
\sup_{t\in I}\|\partial_t\mathscr F_\eta b\|_{L^q(\Omega\cup S_\alpha)}\leq\,c\,\|b\|_{W^{1,q}(\omega)}\|\partial_t\eta\|_{L^\infty(I\times\omega)}
\end{align*}
for all $b\in W^{1,q}(\omega)$, where the constant $c$ depends only on $\Omega,p$ and $L-\alpha$.
\end{corollary}

We now turn to the case of a less regular function $\eta$ and analyse the properties of
$\mathscr F_\eta$ given by \eqref{eq:2401b} in this case.
\begin{lemma}\label{lem:3.8'}
Let $p\in[1,\infty)$ and $\eta\in W^{2,2}(\omega)$ with $\|\eta\|_{L^\infty_x}<\alpha<L$ and let the operator
$\mathscr F_\eta$ by defined by \eqref{eq:2401b}.
\begin{enumerate}
\item[(a)]
We have for all $p\in(1,\infty)$ and $\sigma\in(0,1]$
\begin{align*}
\mathscr F_\eta: W^{\sigma,p}(\omega)\rightarrow W^{\sigma,q}(\Omega\cup S_\alpha)
\end{align*}
for all $q<\frac{3}{2}p$
and $\tr_\eta\mathscr F_\eta b=b\nu$ for all $b\in W^{1,p}(\omega)$. In particular, we have
\begin{align*}
\|\mathscr F_\eta b\|_{W^{\sigma,p}(\Omega\cup S_\alpha)}\leq\,c\,\|b\|_{W^{\sigma,p}(\omega)}
\end{align*}
for all $b\in W^{1,p}(\omega)$.
\item[(b)]
We have for all $r<2$
\begin{align*}
\mathscr F_\eta: W^{2,2}(\omega)\rightarrow W^{2,r}(\Omega\cup S_\alpha)
\end{align*}
and $\tr_\eta\mathscr F_\eta b=b\nu$ for all $b\in W^{2,2}(\omega)$. In particular, we have
\begin{align*}
\|\mathscr F_\eta b\|_{W^{2,r}(\Omega\cup S_\alpha)}\leq\,c\,\|b\|_{W^{2,2}(\omega)}
\end{align*}
for all $b\in W^{2,2}(\omega)$.
\end{enumerate}
The constants in (a) and (b) depend only on $\Omega,p,q$, $\|\eta\|_{W^{2,2}_x}$ and $L-\alpha$.
\end{lemma}

\begin{corollary}\label{cor:2807'}
Let $\eta\in L^2(I;W^{2,2}(\partial\Omega))$ with $\|\eta\|_{L^{\infty}_{t,x}}<\alpha<L$. Suppose that
$\partial_t\eta\in L^q(I\times \omega)$ for some $q>1$. Then we have uniformly in time
\begin{align*}
\|\partial_t\mathscr F_\eta b\|_{L^r(\Omega\cup S_\alpha)}\leq\,c\,\|b\|_{W^{1,p}(\omega)}\|\partial_t\eta\|_{L^q(\omega)}
\end{align*}
for all $b\in W^{1,p}(\omega)$, provided $\frac{1}{r}=\frac{1}{p}+\frac{1}{q}\leq 1$. The constant $c$ depends only on $\Omega,p$ and $L-\alpha$.
\end{corollary}
The following is proved in \cite[Prop. 3.3]{MuSc}. It provides a solenoidal extension. For that we introduce the solenoidal space $W^{1,1}_{\Div}(\Omega\cup S_{\alpha} ):=\{\bfw\in W^{1,1}(\Omega\cup S_{\alpha} )\, :\, \divergence\bfw=0\}$. The corrector $\mathscr K_\eta$ in the below preconditions the boundary data to be compatible with the interior solenoidality.
\begin{proposition}
\label{prop:musc}
For a given $\eta\in L^\infty(I;W^{1,2}( \omega ))$ with $\|\eta\|_{L^\infty_{t,x}}<\alpha<L$, there are linear operators
\begin{align*}
\mathscr K_\eta:L^1( \omega )\rightarrow\mathbb R,\quad
\Testzeta:\{\xi\in L^1(I;W^{1,1}( \omega )):\,\mathscr K_\eta(\xi)=0\}\rightarrow L^1(I;W^{1,1}_{\Div}(\Omega\cup S_{\alpha} )),
\end{align*}
such that the tuple $(\Testzeta(\xi-\mathscr K_\eta(\xi)),\xi-\mathscr K_\eta(\xi))$ satisfies
\begin{align*}
\Testzeta(\xi-\mathscr K_\eta(\xi))&\in L^\infty(I;L^2(\Omega_\eta))\cap L^2(I;W^{1,2}_{\Div}(\Omega_\eta)),\\
\xi-\mathscr K_\eta(\xi)&\in L^\infty(I;W^{2,2}( \omega ))\cap  W^{1,\infty}(I;L^{2}( \omega )),\\
\mathrm{tr}_\eta (\Testzeta&(\xi-\mathscr K_\eta(\xi))=\xi-\mathscr K_\eta(\xi),\\
\Testzeta(\xi-\mathscr K_\eta&(\xi))(t,x)=0 \text{ for } (t,x)\in I \times (\Omega \setminus S_{\alpha})
\end{align*}
provided we have $\xi\in L^\infty(I;W^{2,2}( \omega ))\cap  W^{1,\infty}(I;L^{2}(\omega))$.
In particular, we have the estimates
\begin{align}\label{musc1}
\|\Testzeta(\xi-\mathscr K_\eta(\xi))\|_{L^q(I;W^{1,p}(\Omega \cup S_{\alpha}  ))}\lesssim \|\xi\|_{L^q(I;W^{1,p}( \omega ))}+\|\xi\nabla \eta\|_{L^q(I;L^{p}( \omega ))},\\
\label{musc2}\|\partial_t\Testzeta(\xi-\mathscr K_\eta(\xi))\|_{L^q(I;L^{p}( \Omega\cup S_{\alpha}))}\lesssim \|\partial_t\xi\|_{L^q(I;L^{p}( \omega ))}+\|\xi\partial_t \eta\|_{L^q(I;L^{p}( \omega ))},
\end{align}
for any $p\in (1,\infty),q\in(1,\infty]$ 
\end{proposition}

\subsection{Convergence in variable domains.}
Due to the variable domain the framework of Bochner spaces is not available. Hence, we cannot use the classical Aubin-Lions compactness theorem. In this subsection we are concerned with the question of how to get compactness anyway.
We start with the following definition of convergence in variable domains.
\begin{definition}
\label{def:conv}
Let $(\eta_i)\subset C(\overline{I}\times \omega;[-\theta L, \theta L])$, $\theta\in(0,1)$, be a sequence with $\eta_i\rightarrow \eta$ uniformly in $\overline I\times\omega$.
 Let $p,q\in [1,\infty]$ and $k\in\N_0$.
\begin{enumerate}
%\begin{align}
%g_i\to g \in L^p(I,L^q(\Omega_{\eta_i}))\text{ if }
%%\sup_{\set{\phi\in C^\inty([0,T]\times \setR^3),\norm{\phi}_p\leq 1}}
%\lim_{i\to \infty}\bigg(\int_I\bigg(\int_{\setR^3}\abs{g_i\chi_{\Omega_{\eta_i}}-g\chi_{\Omega_\eta}}^q\dx\bigg)^\frac{p}{q}\dt\bigg)^\frac{1}{p}=0.
%\end{align}
\item We say that  a sequence $(g_i) \subset L^p(I,L^q(\Omega_{\eta_i}))$ converges to $g$ in $L^p(I,L^q(\Omega_{\eta_i}))$ strongly with respect to $(\eta_i)$, in symbols
$
g_i\toetai g \,\text{in}\, L^p(I,L^q(\Omega_{\eta_i})),
$
%\sup_{\set{\phi\in C^\inty([0,T]\times \setR^3),\norm{\phi}_p\leq 1}}
if
\begin{align*}
\chi_{\Omega_{\eta_i}}g_i\to \chi_{\Omega_\eta}g \quad\text{in}\quad L^p(I,L^q(\R^3)).
\end{align*}
\item Let $p,q<\infty$. We say that  a sequence $(g_i) \subset L^p(I,L^q(\Omega_{\eta_i}))$ converges to $g$ in $L^p(I,L^q(\Omega_{\eta_i}))$ weakly with respect to $(\eta_i)$, in symbols
$
g_i\weaktoetai g \,\text{in}\, L^p(I,L^q(\Omega_{\eta_i})),
$
%\sup_{\set{\phi\in C^\inty([0,T]\times \setR^3),\norm{\phi}_p\leq 1}}
if
\begin{align*}
\chi_{\Omega_{\eta_i}}g_i\weakto \chi_{\Omega_\eta}g \quad\text{in}\quad L^p(I,L^q(\R^3)).
\end{align*}
\item Let $p=\infty$ and $q<\infty$. We say that  a sequence $(g_i) \subset L^\infty(I,L^q(\Omega_{\eta_i}))$ converges to $g$ in $L^\infty(I,L^q(\Omega_{\eta_i}))$ weakly$^*$ with respect to $(\eta_i)$, in symbols
$
g_i\weakto^{\ast,\eta} g \,\text{in}\, L^\infty(I,L^q(\Omega_{\eta_i})),
$
%\sup_{\set{\phi\in C^\inty([0,T]\times \setR^3),\norm{\phi}_p\leq 1}}
if
\begin{align*}
\chi_{\Omega_{\eta_i}}g_i\weakto^* \chi_{\Omega_\eta}g \quad\text{in}\quad L^\infty(I,L^q(\R^3)).
\end{align*}
%and we say for $l,k\in \setN_0$ that
%\begin{align*}
%g_i\to g\text{ in }&W^{k,p}(I,W^{l,q}(\Omega_{\eta_i}))\text{ if }\nabla^\alpha\partial_t^ig_i\to g\text{ in }L^p(I,L^q(\Omega_{\eta_i}))\text{ for all }i\in \set{0,...,k}
%\\
%&\text{ and } \alpha\in \setN_0^n\text{ such that }\abs{\alpha_1}+...+\abs{\alpha_n}\leq k.
%\end{align*}
\end{enumerate}
\end{definition}
Note that in the case of one single $\eta$ (i.e. not a sequence) the space
$L^p(I,L^q(\Omega_{\eta}))$ (with $1\leq p<\infty$ and $1<q<\infty$) is reflexive and we have the usual duality pairing
\begin{align}\label{0903}
L^p(I,L^q(\Omega_{\eta}))\cong L^{p'}(I,L^{q'}(\Omega_{\eta}))
\end{align} 
provided $\eta$ is smooth enough, see \cite{NRL}.
Definition \ref{def:conv} can be extended in a canonical way to Sobolev spaces:
A sequence $(g_i) \subset L^p(I,W^{1,q}(\Omega_{\eta_i}))$ converges to $g$ in  $L^p(I,W^{1,q}(\Omega_{\eta_i}))$ strongly with respect to $(\eta_i)$, in symbols
\begin{align*}
g_i\toetai g \quad\text{in}\quad L^p(I,W^{1,p}(\Omega_{\eta_i})),
\end{align*}
if both $g_i$ and $\nabla g_i$ converges (to $g$ and $\nabla g$ respectively) in $L^p(I,L^{q}(\Omega_{\eta_i}))$ strongly with respect to $(\eta_i)$ (in the sense of Definition \ref{def:conv} a)).
We also define weak and weak$^*$ convergence in Sobolev spaces with respect to $(\eta_i)$ with an obvious meaning. Note that also an extension to higher order Sobolev spaces is possible but not needed for our purposes.

For the next compactness lemma (see \cite[Lemma 2.8]{BrSc}) we require the following assumptions on the functions describing the boundary 
\begin{enumerate}[label={\rm (A\arabic{*})}, leftmargin=*]
\item\label{A1} The sequence $(\eta_i)\subset C(\overline I\times \omega;[-\theta L, \theta L])$, $\theta\in(0,1)$,  satisfies
\begin{align*}
\eta_i&\rightharpoonup^\ast\eta\quad\text{in}\quad L^\infty(I,W^{2,2}(\omega)),\\
\partial_t\eta_i&\rightharpoonup^\ast\partial_t\eta\quad\text{in}\quad L^\infty(I,L^{2}(\omega)).
\end{align*}
\item\label{A2} Let $(v_i)$ be a sequence such that for some $p,s\in[1,\infty)$ and $\alpha\in(0,1]$ we have
$$v_i\weaktoetai v\quad\text{in}\quad L^p(I;W^{\alpha,s}(\Omega_{\eta_i})).$$
\item\label{A3} Let $(r_i)$ be a sequence such that for some $m,b\in[1,\infty)$ we have
$$r_i\weaktoetai r\quad\text{in}\quad L^m(I;L^{b}(\Omega_{\eta_i})).$$
Assume further that there are sequences $(\bfH^1_i)$, $(\bfH^2_i)$ and $(h_i)$, bounded in $L^m(I;L^{b}(\Omega_{\eta_i}))$, such that $\partial_t r_i=\Div\Div \bfH_i^1+\Div \bfH_i^2+h_i$ in the sense of distributions, i.e.,
\begin{align*}
\int_I\int_{\Omega_{\eta_i}}r_i\,\partial_t\phi\dxt=\int_I\int_{\Omega_{\eta_i}}\bfH^1_i:\nabla^2\phi\dxt
+\int_I\int_{\Omega_{\eta_i}}\bfH^2_i\cdot\nabla\phi\dxt
+\int_I\int_{\Omega_{\eta_i}}h_i\,\phi\dxt
\end{align*}
for all $\phi\in C^\infty_c(I\times\Omega_{\eta_i})$.
\end{enumerate}
In \cite[Lemma 2.8]{BrSc} the corresponding version of \ref{A2} assumes $\alpha=1$. But the very same argument is also valid in case $\alpha\in(0,1)$ due to compact embeddings for fractional Sobolev spaces.
\begin{lemma}
\label{thm:weakstrong}
Let $(\eta_i)$, $(v_i)$ and $(r_i)$ be sequences satisfying \ref{A1}--\ref{A3}
 where $\frac{1}{s^*}+\frac{1}{b}=\frac{1}{a}<1$ (with $s^*=\frac{3s}{3-s\alpha}$ if $s\in (1,3/\alpha)$ and $s^*\in(1,\infty)$ arbitrarily otherwise) and $\frac{1}{m}+\frac{1}{p}=\frac{1}{q}<1$. Then there is a subsequence with
%\begin{align*}
%\tilde{v}_i&\weakto \tilde{v}\text{ weakly in }L^{p}([0,T],W^{1,s}(\Omega))
%\\
%\tilde{r}_i&\weakto \tilde{r}\text{ weakly in } L^{m}([0,T],L^b(\Omega))
%\\
%{v}_i&\weakto {v}\text{ weakly in } L^{p}([0,T],W^{1,s}(\Omega_{\eta_i}))
%\\
%{r}_i&\weakto {r}\text{ weakly in } L^{m}([0,T],L^b(\Omega_{\eta_i}))
%\\
%{r}_i&\to {r}\text{ strongly in } L^{p'}([0,T],W^{-1,s}(\Omega_{\eta_i}))
%\end{align*}
%Moreover,
\begin{align}
\label{al}
v_i r_i\weakto^\eta v r\text{ weakly in }L^{q}(I,L^a(\Omega_{\eta_i})).
\end{align}
\end{lemma}
%\begin{remark}
%Assumption \ref{A3} in Lemma \ref{thm:weakstrong} can be extended in an obvious way to the case of higher order distributional derivatives. We have chosen the version above as it is most suitable for our applications.
%\end{remark}

\begin{corollary}
\label{rem:strong}
In the case $r_i=v_i$ we find that
%\[
%\int_0^T\int_{\Omega_{\eta_i}}\abs{v_i}^2\, dx\, dt\to \int_0^T\int_{\Omega_{\eta}}\abs{v}^2\, dx\, dt.
%\]
%Since weak convergence and norm convergence implies strong convergence, we find (by interpolation) that
\[
v_i\to^\eta v\text{ strongly in }L^{2}(I, L^{2}(\Omega_{\eta_i})).
\]
\end{corollary}
We finish this section be repeating  the following Aubin-Lions type lemma which is shown in~\cite[Theorem 5.1. \& Remark 5.2.]{MuSc}.
\begin{theorem}
\label{thm:auba}
Let $X,Z$ be two Banach spaces, such that $X'\subset Z'$.
Assume that $f_n:(0,T)\to X$ and $g_n: (0,T)\to X'$, such that $g_n\in L^\infty(0,T;Z')$ uniformly. Moreover assume the following: 
\begin{enumerate}
\item The {\em weak convergence}: for some $s\in [1,\infty]$ we have that $f_n\weaktostar f$ in $L^s(X)$ and $g_n\weaktostar g$ in $L^{s'}(X')$.
%\item The {\em uniform bound} on one sequence
%\[
%\sup_n\norm{f_n}_{L^p(0,T;Y)}\leq C.
%\]
%moreover we suppose that $f_n\toweakstar f$ in $L^p(0,T;X)$.
\item The {\em approximability-condition} is satisfied: For every $\kappa\in (0,1]$ there exists a $f_{n,\kappa}\in L^s(0,T;X)\cap L^1(0,T;Z)$,
%
%and every $t\in [0,T]$ there exists a mollifying operator $\bor{X\ni}\phi\mapsto (\phi)_{\kappa,t}\ni Z$.
such that for every $\epsilon\in (0,1)$ there exists a $\kappa_\epsilon\in(0,1)$ (depending only on $\epsilon$) such that %for a.e. $t\in [0,T]$
\[
%\sup_{k,n}
\norm{f_n-f_{n,\kappa}}_{L^s(0,T;X)}\leq \epsilon\text{ for all } \kappa\in (0,\kappa_\epsilon]%\text{ and }\norm{f_{n,\kappa}}_{Z}\leq  C(\kappa)(1+\norm{f_n}_Y^p).
\]
and for every $\kappa\in (0,1]$ there is a $C(\kappa)$ such that
\[
\norm{f_{n,\kappa}}_{L^1(0,T;Z)}\, dt\leq C(\kappa).
\]
Moreover, we assume that for every $\kappa$ there is a function $f_\kappa$, and a subsequence such that $f_{n,\kappa}\weaktostar f_\kappa$ in $L^s(0,T;X)$.

\item The {\em equi-continuity} of $g_n$. We require that there exists an $\alpha\in (0,1]$ a functions $A_n$ with $A_n\in L^1(0,T)$ uniformly, such that for every $\kappa>0$ that there exist a $C(\kappa)>0$ and an $n_\kappa\in \N$ such that for $\tau>0$ and a.e.\ $t\in [0,T-\tau]$
\[
\sup_{n\geq n_{\kappa}} \Big|\dashint_{0}^\tau\skp{g_n(t)-g_n(t+s)}{f_{n,\kappa}(t)}_{X',X}\,\dd s\Big| \leq C(\kappa)\tau^\alpha(A_n(t)+1).
\]

\item The {\em compactness assumption} is satisfied: $X'\hookrightarrow \hookrightarrow Z'$. More precisely, every uniformly bounded sequence in $X'$ has a strongly converging subsequence in $Z'$. 
\end{enumerate}
Then there is a subsequence, such that
\[
\int_0^T\skp{f_n}{g_n}_{X,X'}\dt\to \int_0^T\skp{f}{g}_{X,X'}\dt.
\]
\end{theorem}

\subsection{Weak solutions and main theorem}
\label{sec:weak}
In accordance with the current state of the art for weak solutions of Navier-Stokes-Fourier law fluids and fluid-structure interactions, we introduce here our concept of weak solutions.
For that we introduce the following function spaces, where $\mathcal D(\bfu)=\frac{1}{2}(\nabla \bfu+\nabla \bfu^T)$ denotes the symmetric gradient of a given function:
\begin{enumerate}[label={(S\arabic{*})}]
\item For the solid deformation $\eta:I\times \omega\to \mathbb{R}$, 
$Y^I:=\{\zeta\in W^{1,\infty}(I;L^2(\omega))\cap L^\infty(I;W^{2,2}(\omega))\,\}$.
\item For the fluid velocity $\bfu:I\times \Omega_\eta\to\mathbb{R}^d$, $d=2,3$, $X_\eta^I:=\{\bfu\in L^2(I;L^2(\Omega_\eta))):\,\mathcal D(\bfu)\in L^2(I;L^{2}(\Omega_{\eta}))\}$.
\item For the fluid density $\varrho: I\times \Omega_\eta\to [0,\infty)$, $
 W_\eta^I:= C_w(\overline{I};L^\gamma(\Omega_\eta))$, where the subscript $w$ refers to continuity with respect to the weak topology.
\item For the temperature $\vt: I\times \Omega_\eta\to [0,\infty)$
\[
Z_\eta^I=\{\vt\in L^2(I;W^{1,2}(\Omega_\eta))\cap L^\infty(I;L^4(\Omega_\eta)):\log(\vt)\in L^2(W^{1,2}(\Omega_\eta))\}.
\]
\end{enumerate}
The definition of the function spaces above depending on $\eta$ only make sense provided $\|\eta\|_{L^\infty_{t,x}}<L$.
\begin{definition}
\label{def:weak-sol}
A weak solution to \eqref{eq1}--\eqref{initial} is a quadruplet $(\eta,\bfu,\varrho,\vt)\in \times Y^I\times X_\eta^I\times W_\eta^I\times Z_\eta^I$, which satisfies the following.
\begin{enumerate}[label={(O\arabic{*})}]
\item\label{O1} The momentum equation is satisfied in the sense that
\begin{align}\label{eq:apufinal}
\begin{aligned}
&\int_I\frac{\dd}{\dt}\int_{\Omega_{ \eta}}\varrho\bfu \cdot\bfphi\dxt-\int_{\Omega_{\eta}} \Big(\varrho\bfu\cdot \partial_t\bfphi +\varrho\bfu\otimes \bfu:\nabla \bfphi\Big)\dxt
\\
&+\int_I\int_{\Omega_\eta}\bfS(\vt,\nabla\bfu):\nabla\bfphi\dxt-\int_I\int_{\Omega_{ \eta }}
p(\vr,\vt)\,\Div\bfphi\dxt\\
&+\int_I\bigg(\frac{\dd}{\dt}\int_\omega \partial_t \eta b\dH-\int_\omega \partial_t\eta\,\partial_t b\dH + \int_\omega K'(\eta)\,b\dH\bigg)\dt
\\&=\int_I\int_{\Omega_{ \eta}}\varrho\bff\cdot\bfphi\dxt+\int_I\int_\omega g\,b\,\dd y\dt
\end{aligned}
\end{align} 
holds for all $(b,\bfphi)\in C^\infty(\omega)\times C^\infty(\overline{I}\times\R^3)$ with $\mathrm{tr}_{\eta}\bfphi=b\nu$. Moreover, we have $(\varrho\bfu)(0)=\bfq_0$, $\eta(0)=\eta_0$ and $\partial_t\eta(0)=\eta_1$. The boundary condition $\mathrm{tr}_\eta\bfu=\partial_t\eta\nu$ holds in the sense of Lemma \ref{lem:2.28}.
\item\label{O2} The continuity equation is satisfied in the sense that 
\begin{align}\label{eq:apvarrho0final}
\begin{aligned}
&\int_I\frac{\dd}{\dt}\int_{\Omega_{\eta}}\varrho \psi\dxt-\int_I\int_{\Omega_{\eta}}\Big(\varrho\partial_t\psi
+\varrho\bfu\cdot\nabla\psi\Big)\dxt=0
\end{aligned}
\end{align}
holds for all $\psi\in C^\infty(\overline{I}\times\R^3)$ and we have $\varrho(0)=\varrho_0$. 
\item \label{O2'} The entropy balance
\begin{align} \label{m217*final1}\begin{aligned}
 \int_I \frac{\dd}{\dt}\int_{\Omega_{\eta}} \vr s(\varrho,\vartheta) \, \psi\dxt
 &- \int_I \int_{\Omega_{\eta}} \big( \vr s(\varrho,\vartheta) \partial_t \psi + \varrho s (\varrho,\vartheta)\bfu \cdot \nabla\psi \big)\dxt
\\& \geq\int_I\int_{\Omega_{\eta}}
\frac{1}{\vartheta}\Big(\bfS(\vartheta,\nabla\bfu):\nabla\bfu+\frac{\varkappa(\vartheta)}{\vartheta}|\nabla\vartheta|^2\Big)
\psi\,\dif x\,\dif t \\
&+ \int_I\int_{\Omega_{\eta}} \frac{\varkappa(\vartheta)\nabla\vartheta}{\vartheta} \cdot \nabla\psi \dxt+ \int_I \int_{\Omega_{\eta}} \frac{\vr}{\vt} H \psi  \dxt
\end{aligned}
\end{align}
holds for all
$\psi\in C^\infty(\overline I\times \R^3)$ with $\psi \geq 0$; in particular, all integrals above are well defined.  Moreover, we have
$\lim_{r\rightarrow0}\vr s(\vr,\vt)(t)\geq \vr_0 s(\vr_0,\vt_0)$ and $\partial_{\nu_{\eta}}\vt|_{\partial\Omega_\eta}\leq 0$.
\item \label{O3} The total energy balance
\begin{equation} \label{EI20final}
\begin{split}
- \int_I \partial_t \psi \,
\mathcal E \dt &=
\psi(0) \mathcal E(0) + \int_I \psi \int_{\Omega} \vr H  \dxt+\int_I\psi\int_{\Omega_{\eta}}\varrho\bff\cdot\bfu\dxt\\&+\int_I\psi\int_\omega g\,\partial_t\eta\,\dd y\dt
\end{split}
\end{equation}
holds for any $\psi \in C^\infty_c([0, T))$.
Here, we abbreviated
$$\mathcal E(t)= \int_{\Omega_\eta(t)}\Big(\frac{1}{2} \varrho(t) | {\bf u}(t) |^2 + \varrho(t) e(\varrho(t),\vartheta(t))\Big)\dx+\int_\omega \frac{|\partial_t\eta(t)|^2}{2}\,\dd y+ K(\eta(t)).$$
\end{enumerate}
\end{definition}

As will be apparent by the analysis we will show that the renormalized continuity equation in the sense of DiPerna and Lions is satisfied, cf. \cite{DL,Li2}.
\begin{definition}[Renormalized continuity equation]
\label{def:ren}
Let $\eta\in Y^I$ and $\bfu\in X_\eta^I$. We say that the function $\varrho\in W_\eta^I$ solves the continuity equation~\eqref{eq1} in the renormalized sense if we have
\begin{align}\label{eq:final3b}
\begin{aligned}
\int_I\frac{\dd}{\dt}\int_{\Omega_\eta}\theta(\varrho)\psi\dxt&-\int_I\int_{\Omega_\eta}\Big(\theta(\varrho)\partial_t\psi +\theta(\varrho) \bfu\cdot \nabla \psi\Big)\dxt
\\
& =- \int_I\int_{\Omega_\eta}(\varrho\theta'(\varrho)-\theta(\varrho)) \diver\bfu \,\psi\dxt
\end{aligned}
\end{align}
for all $\psi\in C^\infty(\overline{I}\times\R^3)$ and all $\theta\in C^1(\R)$ with 
$\theta(0)=0$ and 
$\theta'(z)=0$ for $z\geq M_\theta$.
\end{definition}

We are now ready to formulate our main result.
\begin{theorem} \label{thm:MAIN}
%For regular data and $\gamma\in \big(\frac{12}7,\infty)$ ($\gamma>1$ in two dimensions)
Let $\gamma>\frac{12}{7}$ ($\gamma>1$ in two dimensions). Assume that we have
\begin{align*}
\frac{|\bfq_0|^2}{\varrho_0}&\in L^1(\Omega_{\eta_0}),\ \varrho_0\in L^\gamma(\Omega_{\eta_0}),\ \vt_0\in L^4(\Omega_{\eta_0}), \ \eta_0\in W^{2,2}(\omega),\ \eta_1\in L^2(\omega),\\
\bff&\in L^2(I;L^\infty(\R^3)),\ g\in L^2(I\times \omega),\ H\in L^\infty(I\times\R^3),\ H\geq0\ \text{a.e.}
\end{align*}
Furthermore suppose that $\varrho_0\geq0$ a.e., $\vt_0\geq0$ a.e. and that \eqref{eq:compa} is satisfied.
Then there exists a weak solution $(\eta,\bfu,\varrho,\vartheta)$ to \eqref{eq1}--\eqref{initial} in the sense of Definition~\ref{def:weak-sol}. 
%In particular the following energy balance is satisfied:
%\begin{align}\label{eq:apriori0}
%\begin{aligned}
%\mathcal E(t) &=
%\mathcal E(0) + \int_{\Omega_\eta} \vr H  \dx+\int_{\Omega_{\eta}}\varrho\bff\cdot\bfu\dx+\int_\omega g\,\partial_t\eta\dH,\\
%\mathcal E(t)&= \int_{\Omega_\eta(t)}\Big(\frac{1}{2} \varrho(t) | {\bf u}(t) |^2 + \varrho(t) e(\varrho(t),\vartheta(t))\Big)\dx+\int_\omega \frac{|\partial_t\eta(t)|^2}{2}\,\dH+ K(\eta(t)).\end{aligned}
%\end{align} 
The interval of existence is of the form $I=(0,t)$, where $t<T$ only in case $\Omega_{\eta(s)}$ approaches a self-intersection when $s\to t$ or the Koiter energy degenerates (namely, if $\lim_{s\to t}\overline{\gamma}(s,y)=0$ for some point $y\in \omega$). Moreover, the continuity equation is satisfied in the renormalized sense as specified in Definition~\ref{def:ren}.
\end{theorem}

\begin{remark}[Minimal interval of existence]
Let us mention that for any admissible initial conditions there is a minimal positive interval of existence. It follows from the fact that $\eta$  (and consequently also $\overline{\gamma}$, cf.~\eqref{geomQuan}) can be shown to be uniformly continuous in space-time (with bounds depending on the data only). Consequently, for some non-empty open time-interval no self-touching or point of degeneracy can be approached a-priori.
\end{remark}

\begin{remark}[Simplification of notation]\label{rem:simple}
We remark that we will assume without further mentioning that the initial conditions for the elastic deformation are within a neighbourhood of the reference configurations. This simplification is, however, without loss of generality. Indeed, by rephrasing the reference geometry accordingly, the existence procedure can be prolonged until a point of self touching or degeneracy (in case of non-linear Koiter energies) is approached.
\end{remark}

\begin{remark}[Properties of the solution]\label{rem:korn}
As can be seen from the proof, in particular \eqref{NTDB}, we can control in addition to the energy the quantity
\begin{align*}
\sigma&=\frac{1}{\vartheta}\bfS(\vartheta,\nabla\bfu):\nabla\bfu+\frac{\varkappa(\vartheta)}{\vartheta^2}|\nabla\vartheta|^2
\end{align*}
in $L^1$. This implies that
\begin{enumerate}
\item The symmetric gradient $\mathcal D(\bfu)$ belongs to $L^2$ due to \eqref{m105}--\eqref{m106}. Since the domain is not Lipschitz continuous (at least not uniformly in time) standard results on Korn-type inequalities do not apply. In our context of domains with less regularity, a corresponding inequality is shown in
\cite[Prop. 2.9]{Le} following ideas of \cite{Ac}. The integrability of the full gradient is, however, less than the one of the symmetric gradient, that is we only have $\nabla \bfu\in L^q$ for all $q<2$.
\item  
%Then, 
% and following~\cite[Section 2.2.4]{F} one find that 
All involved integrals in \ref{O2'} are finite. In particular,
 the temperature satisfies $\nabla\log \vt\in L^2$ using \eqref{m108}.
  Further one can deduce that $\log\vt\in L^2$ following~\cite[Section 2.2.4]{F}. This implies $\vt>0$ a.a. in $I\times\Omega_\eta$. 
%This is due to the missing term $\vr\vt$ in the pressure law, see also Remark \eqref{rem:pressurelaw}.
\end{enumerate}
\end{remark}

\section{Equations for density and temperature in variable domains}
\label{sec:3a}
In this section we study the continuity equation (with artificial viscosity) as well as the internal energy equation
in variable domains. In Theorems \ref{thm:regrho} and \ref{thm:regtheta} we prove the existence of classical solutions to both equations under the assumption the data (the velocity field as well and the variable boundary) are smooth. In particular, we prove that the temperature stays strictly positive on the regularised level. This is a key ingredient for the remainder of the paper.

\subsection{The continuity equation}
\label{subsec:continuity}
%We find the renormalized formulation as
%\begin{align}
%\label{eq:renorm}
%\begin{aligned}
%0=&\int_I\frac{\dd}{\dt}\int_{\Omega_\eta}\theta(\varrho)\,\psi\dxt-\int_I\int_{\Omega_\eta}\Big(\theta(\varrho)\partial_t\psi +\theta(\varrho) \bfu\cdot \nabla \psi\Big)\dxt
%\\
%&\quad + \int_I\int_{\Omega_\eta}(\varrho\theta'(\varrho)-\theta(\varrho)) \diver\bfu\, \psi\dxt
%\end{aligned}
%\end{align}
%for all $\psi\in C^\infty(\overline{I}\times\R^3)$.
In this subsection we are concerned with the regularised continuity equation in a (given) variable domain.
%We will need very explicit a-priori information about our approximation of the density $\varrho$. The necessary result is collected in Theorem~\ref{lem:warme} below. 
 We assume that the moving boundary is prescribed by a function $\zeta:\overline{I}\times \omega\rightarrow \R$.
 For a given function $\bfw\in L^2(I;W^{1,2}(\Omega_\zeta))$ with $\tr_{\zeta}\bfw=\partial_t\zeta\nu$ and $\varepsilon>0$ we consider the equation
\begin{align}\label{eq:warme}\begin{aligned}
\partial_t \varrho&+\Div(\varrho\bfw)=\varepsilon\Delta\varrho\quad\text{in}\quad I\times\Omega_\zeta,\\
\varrho(0)&=\varrho_0\text{ in }\Omega_{\zeta(0)},\quad\partial_{\nu_\zeta}\varrho\big|_{\partial\Omega_\zeta}=0\quad\text{on}\quad I\times\partial\Omega_\zeta.
\end{aligned}
\end{align}
 A weak solution to \eqref{eq:warme} satisfies
\begin{align}
\int_I\frac{\dd}{\dt}\int_{\Omega_\zeta}\varrho\psi\dx\dt
-\int_I\int_{\Omega_\zeta}\big(\varrho \partial_t\psi +\varrho \bfw\cdot\nabla \psi\big)\dxt=-\int_I\int_{\Omega_\zeta}\varepsilon\nabla\varrho\cdot\nabla\psi\dxt\label{eq:weak-masse}
\end{align}
for all $\psi\in C^\infty(\overline{I}\times \R^3)$. The following result has been proved in \cite[Thm. 3.1]{BrSc}(for the analogous results for fixed in time domains see~\cite[section 2.1]{feireisl1}).

\begin{theorem}\label{lem:warme}Let $\zeta\in C^{2,\alpha}(\overline{I}\times \omega,[-\frac{L}2,\frac{L}2])$ with $\alpha\in(0,1)$ be the function describing the boundary. Assume that $\bfw\in L^2(I;W^{1,2}(\Omega_\zeta))\cap L^\infty(I\times \Omega_\zeta)$ with $\tr_{\zeta}\bfw=\partial_t\zeta\nu$ and
%such that $\bfw(t,x)\cdot\nu(x)\circ \Psi_{\zeta}^{-1}(t,x) =\zeta(t,x)$ for all $(t,x)\in I\times M$ and 
$\varrho_0\in L^{2}(\Omega_{\zeta(0)})$.
\begin{enumerate}
\item[a)] There is a unique weak solution $\varrho$ to \eqref{eq:warme} such that
$$\varrho\in L^\infty(I;L^2(\Omega_{\zeta}))\cap L^2(I;W^{1,2}(\Omega_{\zeta})).$$
\item[b)] Let $\theta\in C^2(\setR_+;\setR_+)$ be such that $\theta'(s)= 0$
%\sebcom{we should replace this condition by $\theta''(s)\geq 0$}
 for large values of $s$ and $\theta(0)=0$.%\footnote{The restriction $\theta(0)=0$ is only needed if we extend the equation to the whole space.} 
Then the following holds, for the canonical zero extension of $\varrho\equiv\varrho\chi_{\Omega_{\zeta}}$:
\begin{align}
\label{eq:renormz}
\begin{aligned}
\int_I\frac{\dd}{\dt}&\int_{\R^3} \theta(\varrho)\,\psi\dxt-\int_{I\times \R^3}\theta(\varrho)\,\partial_t\psi\dxt
\\
=&-\int_{I\times \R^3}\big(\varrho\theta'(\varrho)-\theta(\varrho)\big)\Div\bfw\,\psi\dx+\int_{I\times \R^3}\theta(\varrho) \bfw\cdot\nabla\psi\dxt\\ &-\int_{I\times\R^3}\varepsilon\chi_{\Omega_{\zeta}}\nabla \theta(\varrho)\cdot\nabla\psi\dxt-\int_{I\times \R^3}\varepsilon\chi_{\Omega_{\zeta}}\theta''(\varrho)|\nabla\varrho|^2\psi\dxt
\end{aligned}
\end{align}
for all $\psi\in C^\infty(\overline{I}\times\R^3)$.
\item[c)] Assume that $\varrho_0\geq0$ a.e. in $\Omega_{\zeta(0)}$. Then we have $\varrho\geq0$ a.e. in $I\times\Omega_\zeta$. 
\end{enumerate}
\end{theorem}
\begin{remark}

\begin{itemize}
Observe that:
\item The statement in \cite{BrSc} holds without the assumption $\tr_{\zeta}\bfw=\partial_t\zeta\nu$ under the boundary condition $\partial_{\nu_\zeta}\varrho\big|_{\partial\Omega_\zeta}=\tfrac{1}{\varepsilon}\varrho(\bfw-(\partial_t\zeta\nu)\circ\bfvarphi^{-1}_\zeta)\cdot\nu_\zeta$.
\item Theorem 3.1 in \cite{BrSc} is formulated with the stronger assumption $\zeta\in C^{3}(\overline{I}\times \omega,[-\frac{L}2,\frac{L}2])$.
 However, it can be checked that the condition $\zeta\in C^{2,\alpha}(\overline{I}\times \omega,[-\frac{L}2,\frac{L}2])$ is sufficient for the proof.
 \end{itemize}
\end{remark}
In the following we improve the result from Theorem \ref{lem:warme} and obtain a classical solution to \eqref{eq:warme}.
\begin{theorem}\label{thm:regrho}
Let the assumptions of Theorem \ref{lem:warme} be satisfied and suppose additionally that $\partial_t\nabla^2\zeta$ as well as $\nabla^3\zeta$ belong to the class $C^{\alpha}(\overline I\times \omega)$. Furthermore we assume that  $J_\zeta:=\mathrm{det}\nabla\bfPsi_\zeta$ is strictly positive, $\varrho_0\in C^{2,\alpha}(\overline\Omega_{\zeta(0)})$ and 
$\bfw\in C^{1,\alpha}(\overline I\times\overline\Omega_\zeta)$ such that $\partial_t\nabla\bfw$ and $\nabla^2\bfw$ belong to the class $C^{\alpha}(\overline I\times \overline\Omega_\zeta)$. 
\begin{enumerate}
\item
The solution $\varrho$ from Theorem \ref{lem:warme} satisfies \eqref{eq:warme} in the classical sense and belongs to the regularity class
\begin{align*}
\mathcal Z^{I}_\zeta:=\big\{z \in C^1(\overline I\times\overline\Omega_\zeta):\ \nabla^2z \in C(\overline I\times\overline\Omega_\zeta)\big\}.
\end{align*}
In particular, we have
\begin{align*}
\|\varrho\|_{C^1_{t,x}}+\|\nabla^2\varrho\|_{C_{t,x}}\leq\,c\Big(\varrho_0,\zeta,\sup J_\zeta^{-1},\bfw\Big),
\end{align*}
with dependence via the (semi-)norms in the affirmative function spaces.
\item Suppose that $\varrho_0\geq0$. Then we have the estimate
\begin{align*}
 C^{-1}\min_{\Omega_{\zeta(0)}}\varrho_0\leq \max_{I\times\Omega_\zeta} \varrho\leq C\max_{\Omega_{\zeta(0)}}\varrho_0,
\end{align*}
where $C=C(\zeta,\sup J_\zeta^{-1},\bfw)$ with dependence via the (semi-)norms in the affirmative function spaces.
%\item Suppose that $\varrho_0\geq0$ and that $\tr_{\zeta}\bfw=\partial_t\zeta\nu$. Then we have the estimate
%\begin{align*}
%\min_{I\times\Omega_\zeta} \varrho\geq
%\end{align*}
%with $C=C(\|\zeta\|_{C^{2}},\sup J_\zeta^{-1},\|\bfw\|_\infty,\|\nabla\bfw\|_\infty)$.
\end{enumerate}
\end{theorem}
\begin{proof}
We start by transforming \eqref{eq:weak-masse} to the reference domain. For $\overline\psi\in C^\infty(\overline I\times\R^3)$ we set $\psi=\overline\psi\circ\bfPsi_\zeta^{-1}$. Defining similarly
$\overline\varrho=\varrho\circ\bfPsi_\zeta$ and $\overline\bfw=\bfw\circ\bfPsi_\zeta$
we obtain from \eqref{eq:weak-masse}
\begin{align*}
\int_I\frac{\dd}{\dt}\int_{\Omega_\zeta}\overline\varrho\circ\bfPsi_\zeta^{-1}\,\overline\psi\circ\bfPsi_\zeta^{-1}\dx\dt
&=\int_I\int_{\Omega_\zeta}\overline\varrho\circ\bfPsi_\zeta^{-1} \big(\partial_t\overline\psi\circ\bfPsi_\zeta^{-1}+\nabla\overline\psi\circ\bfPsi_\zeta^{-1}\cdot\partial_t\bfPsi_\zeta^{-1}\big)\dxt\\
&+\int_I\int_{\Omega_\zeta}\overline\varrho\circ\bfPsi_\zeta^{-1} \overline\bfw\circ\bfPsi_\zeta^{-1}\cdot(\nabla \bfPsi_\zeta)^{-1}\nabla \overline\psi\circ\bfPsi_\zeta^{-1}\dxt
\\&-\int_I\int_{\Omega_\zeta}\varepsilon\big(\nabla \bfPsi_\zeta^{-1}\big)^{T}\nabla \bfPsi_\zeta^{-1}\nabla\overline\varrho\circ\bfPsi_\zeta^{-1}\cdot\nabla\overline\psi\circ\bfPsi_\zeta^{-1}\dxt
\end{align*}
such that
\begin{align*}
\int_I\frac{\dd}{\dt}\int_{\Omega}J_{\zeta}\overline\varrho\,\overline\psi\dx\dt
&=\int_I\int_{\Omega}J_{\zeta}\overline\varrho \big(\partial_t\overline\psi+\nabla\overline\psi\cdot\partial_t\bfPsi_\zeta^{-1}\circ\bfPsi_\zeta\big)\dxt\\
&+\int_I\int_{\Omega}J_\zeta\overline\varrho\overline\bfw\cdot\big(\nabla \bfPsi_\zeta\big)^{-1}\nabla \overline\psi\dxt
\\&-\int_I\int_{\Omega}\varepsilon J_\zeta\big(\nabla \bfPsi_\zeta\big)^{-T}\big(\nabla \bfPsi_\zeta\big)^{-1}\nabla\overline\varrho\cdot\nabla\overline\psi\dxt,
\end{align*}
where $J_\zeta=\mathrm{det}\nabla\bfPsi_\zeta$. Finally, we replace $\overline\psi$ by $\overline\psi/J_\zeta$ to obtain
\begin{align*}
\int_I\frac{\dd}{\dt}\int_{\Omega}\overline\varrho\,\overline\psi\dx\dt
&=\int_I\int_{\Omega}\big(\overline\varrho \partial_t\overline\psi+\overline\varrho J_\zeta\partial_t J_\zeta^{-1}\overline\psi\big)\dxt\\&+\int_I\int_{\Omega}\big(\overline\varrho \partial_t\bfPsi_\zeta^{-1}\circ\bfPsi_\zeta\cdot\nabla\overline\psi+\overline\varrho J_\zeta\nabla J_\zeta^{-1}\cdot\partial_t\bfPsi_\zeta^{-1}\circ\bfPsi_\zeta\overline\psi\big)\dxt\\
&+\int_I\int_{\Omega}\overline\varrho\overline\bfw\cdot\big(\nabla \bfPsi_\zeta\big)^{-1}\nabla \overline\psi\dxt+\int_I\int_{\Omega}\overline\varrho\overline\bfw\cdot J_\zeta \big(\nabla \bfPsi_\zeta\big)^{-1}\nabla J_\zeta^{-1}\overline\psi\dxt
\\&-\int_I\int_{\Omega}\varepsilon\big(\nabla \bfPsi_\zeta\big)^{-T}\big(\nabla \bfPsi_\zeta\big)^{-1}\nabla\overline\varrho\cdot\nabla\overline\psi\dxt\\
&-\int_I\int_{\Omega}\varepsilon J_\zeta\big(\nabla \bfPsi_\zeta\big)^{-T}\big(\nabla \bfPsi_\zeta\big)^{-1}\nabla\overline\varrho\cdot\nabla J_\zeta^{-1}\overline\psi\dxt.
\end{align*}
Now we set
\begin{align*}
g_\zeta&=J_\zeta\partial_t J_\zeta^{-1}+ J_\zeta\nabla J_\zeta^{-1}\cdot\partial_t\bfPsi_\zeta^{-1}\circ\bfPsi_\zeta+\overline\bfw\cdot J_\zeta \big(\nabla \bfPsi_\zeta\big)^{-1}\nabla J_\zeta^{-1}\\
\bfg_\zeta&=-\varepsilon J_\zeta\big(\nabla \bfPsi_\zeta\big)^{-T}\big(\nabla \bfPsi_\zeta\big)^{-1}\nabla J_\zeta^{-1},\\
\bff_\zeta&= \partial_t\bfPsi_\zeta^{-1}\circ\bfPsi_\zeta+\big(\nabla \bfPsi_\zeta\big)^{-T}\overline\bfw,\quad \bfA_\zeta=\varepsilon\big(\nabla \bfPsi_\zeta\big)^{-T}\big(\nabla \bfPsi_\zeta\big)^{-1},
\end{align*}
such that the equation reads as
\begin{align*}
-\int_I\int_{\Omega}\overline\varrho\,\partial_t\overline\psi\dx\dt
&=\int_I\int_{\Omega}\overline\varrho\,g_\zeta \overline\psi\dxt+\int_I\int_{\Omega}\nabla\overline\varrho\cdot\bfg_\zeta \overline\psi\dxt+\int_I\int_{\Omega}\overline\varrho \bff_\zeta\cdot\nabla\overline\psi\dxt
\\&-\int_I\int_{\Omega}\bfA_\zeta\nabla\overline\varrho\cdot\nabla\overline\psi\dxt
\end{align*}
for any $\overline\psi$ with $\overline\psi(0)=\overline\psi(T)=0$.
Choosing $\overline\psi\in C^\infty_c(I\times\Omega)$ arbitrarily we obtain
\begin{align*}
\partial_t\overline\varrho&=\overline\varrho g_\zeta+\nabla\overline\varrho\cdot\bfg_\zeta-\Div(\overline\varrho\bff_\zeta)+\Div\big(\bfA_\zeta\nabla\overline\varrho\big)\\
&=\overline\varrho\big( g_\zeta-\Div \bff_\zeta\big)+\nabla\overline\varrho\cdot\big(\bfg_\zeta-\bff_\zeta\big)+\Div\big(\bfA_\zeta\nabla\overline\varrho\big)
\end{align*}
and we have the boundary condition
\begin{align}\label{1611}
\nu\cdot\bfA_\zeta\nabla\overline\varrho=\overline\varrho\bff_\zeta\cdot\nu=0.
\end{align}
Here we use that $\partial_t\bfPsi_\zeta^{-1}\circ\bfPsi_\zeta=-\nabla \bfPsi_\zeta^{-T}\overline\bfw$ on $\partial\Omega$ due to the assumption $\tr_{\zeta}\bfw=\partial_t\zeta\nu$. In fact, we have
\begin{align*}
0=\partial_t\big(\bfPsi_\zeta\circ\bfPsi_\zeta^{-1}\big)=\partial_t\bfPsi_\zeta\circ\bfPsi_\zeta^{-1}+\nabla\bfPsi_\zeta^T\circ\bfPsi_\zeta^{-1}\partial_t\bfPsi_\zeta^{-1}
\end{align*}
such that
\begin{align*}
\nabla \bfPsi_\zeta^{T}\partial_t\bfPsi_\zeta^{-1}\circ\bfPsi_\zeta=-\partial_t\bfPsi_\zeta=-(\partial_t\zeta\nu)\circ\bfvarphi=-\bfw\circ\bfPsi_\zeta=-\overline\bfw
\end{align*}
on $I\times\partial\Omega$ due to the definition of $\bfPsi_\zeta$ from \eqref{map}.\\
%Since our data is smooth we can apply \cite[Thm. 7.35, page 185]{Li} to conclude with the claim (we fist obtain an estimate to $\overline\varrho$ from which we have to transform back using smoothness of $\zeta$ and $J_zeta$ strictly positive).
%Note that in Lieberman's notation we have $\beta=-\nu\bfA_\zeta$, $\beta_0=\bff_\zeta\cdot\nu$ and $\gamma=-\nu$ such that $\beta\cdot\gamma=\nu\bfA_\zeta\nu$ which is strictly positive due to definition of $\bfA_\zeta$.\\
We can rewrite the equation further as
\begin{align*}
\partial_t\overline\varrho
&=\overline\varrho\big( g_\zeta-\Div \bff_\zeta\big)+\nabla\overline\varrho\cdot\big(\bfg_\zeta-\bff_\zeta\big)+\Div\bfA_\zeta\cdot\nabla\overline\varrho+\bfA_\zeta:\nabla^2\overline\varrho
\end{align*}
such that we finally obtain
\begin{align}\label{eq:LSU1}
\begin{aligned}
\partial_t\overline\varrho
&+b_\zeta(t,x,\overline\varrho,\nabla\overline\varrho)=\bfA_\zeta:\nabla^2\overline\varrho\quad\text{in}\quad I\times\Omega,\\
&\qquad\nu\bfA_\zeta\cdot\nabla\overline\varrho=0\quad \text{on}\quad I\times\partial\Omega,
\end{aligned}
\end{align}
where
\begin{align*}
b_\zeta(t,x,u,\bfU)&=-u\big( g_\zeta-\Div \bff_\zeta\big)+\bfU\cdot\big(\bff_\zeta-\bfg_\zeta-\Div\bfA_\zeta\big).
%\Psi_\zeta(t,x,u)&=-u\bff_\zeta\cdot\nu.
\end{align*}
By the classical theory from
\cite[Thm.'s 7.2, 7.3 \& 7.4, Chapter V]{LSU} the claim of part (a) follows if we can control the following quantities:\footnote{Note that this gives the required regularity for $\overline\varrho$; however, transforming back by means of $\bfPsi_\zeta^{-1}$ does not alter it due to the regularity of $\zeta$.}
\begin{itemize}
\item The $C^{2,\alpha}$-norm of $\varrho_0$;
\item The $\alpha$-H\"older-semi-norms of $\nabla_x b_\zeta,\partial_u b_\zeta$ and $\partial_\bfU b_\zeta$ with respect to $x$; the constants in
\begin{align*}
-u b_\zeta(t,x,u,\bfU)\leq\,c_0u^2+c_1|\bfU|^2+c_2\quad \forall(t,x,u,\bfU)\in I\times\Omega\times\R\times\R^3;\\
|b_\zeta(t,x,u,\bfU)|+|\nabla_{(t,u)} b_\zeta(t,x,u,\bfU)|+(1+|\bfU|)|\nabla_\bfU b_\zeta(t,x,u,\bfU)|\leq\,c_3(1+u^2+|\bfU|^2);
\end{align*}
\item The coercivity constant of $\bfA_\zeta$ and its upper bound; the $\alpha$-H\"older-semi-norm of $\nabla_x\bfA_\zeta$ and $\partial_u\bfA_\zeta$ with respect to $x$.
%\item The $\alpha/2$-H\"older-semi-norm of $\partial_t\Psi$ and the $\alpha$-H\"older-semi-norm of $\nabla_x\Psi$ with respect to $x$; %the linear-growth constants of $\Psi_\zeta(t,x,u)$ with respect to the third variable;
%the constants in
%\begin{align*}
%-u\Psi_\zeta(t,x,u)\leq\,c_4u^2+c_5\quad \forall(t,x,u)\in I\times\Omega\times\R.
%\end{align*}
\end{itemize} 
{One readily checks that all these quantities can be controlled in terms of $\|\varrho_0\|_{C^{2,\alpha}_x}$, $\|\zeta\|_{C^{2,\alpha}_{t,x}}$, $\|\partial_t\nabla^2\zeta\|_{C^{\alpha}_{t,x}}$, $\|\nabla^3\zeta\|_{C^{\alpha}_{t,x}},\sup J_\zeta^{-1},\|\bfw\|_{C^{1,\alpha}_{t,x}}$ and $\|\nabla^2\bfw\|_{C^\alpha_{t,x}}$.}\\
In order to prove (b) we argue similarly to the classical arguments from
\cite[Thm. 7.3 Chapter V]{LSU} and set
\begin{align*}
v(t,x):=\varphi(x)e^{-\lambda_1 t}\overline\varrho
\end{align*}
to verify the estimate from above. Here $\varphi\in C^\infty(\overline\Omega)$ is constructed such that it satisfies
\begin{align}
\label{eq:testLSU1}\varphi(x)&\geq1\quad\text{in}\quad  \overline\Omega,\\
\frac{\nabla\varphi\cdot\bfA_\zeta\nu}{\varphi}&<0 \quad\text{on}\quad \overline{I}\times\partial\Omega\label{eq:testLSU2}.
\end{align}
Such a function $\varphi$ can be defined using the distance function to the boundary with respect to the direction $\bfA_\zeta\nu$. By the assumption that $\norm{\zeta}_{L^\infty_{t,x}}\leq \frac{L}2$ and the $C^2$ regularity of $\zeta$ this is a well defined function. Note that $\varphi$ is chosen independently of $\lambda_1$ (which we will fix below). We have by \eqref{eq:LSU1} using the linearity of b
\begin{align*}
&\partial_t v=\varphi e^{-\lambda_1 t}\partial_t\overline\varrho-\lambda_1 v=
-\varphi e^{-\lambda_1 t}b_\zeta(t,x,\overline\varrho,\nabla\overline\varrho)+\varphi e^{-\lambda_1 t}\bfA_\zeta:\nabla^2\overline\varrho-\lambda_1 v\\
&=
-b_\zeta(t,x,\varphi e^{-\lambda_1 t}\overline\varrho,\varphi e^{-\lambda_1 t}\nabla\overline\varrho)+\varphi e^{-\lambda_1 t}\bfA_\zeta:\nabla^2\overline\varrho-\lambda_1 v\\
&=
-b_\zeta(t,x,v,\nabla v-\overline\varrho e^{-\lambda_1 t}\nabla\varphi)+ \bfA_\zeta:\nabla^2 v-\bfA_\zeta:\big(2e^{-\lambda_1 t}\nabla\overline\varrho\otimes^{sym}\nabla\varphi+\overline\varrho e^{-\lambda_1 t}\nabla^2\varphi\big)-\lambda_1 v.
\end{align*}
Let us assume that there is a point $(t_0,x_0)\in I\times \Omega$
with $v(t_0,x_0)=\max_{t,x} v(t,x)$. 
%Due to $\nabla v=0$ we have
%Note that (in the maximum)
%\begin{align*}
%\varphi e^{-\lambda_1 t}\nabla^2\overline\varrho&=\nabla^2 v-2e^{-\lambda_1 t}\nabla\overline\varrho\otimes\nabla\varphi-\overline\varrho e^{-\lambda_1 t}\nabla^2\varphi\\
%&=\nabla^2 v+2e^{-\lambda_1 t}\frac{\overline\varrho}{\varphi}\nabla\varphi\otimes\nabla\varphi-\overline\varrho e^{-\lambda_1 t}\nabla^2\varphi.
%\end{align*}
We obtain in this point
\begin{align*}
0
&=
-b_\zeta(t,x,v,-\overline\varrho e^{-\lambda_1 t_0}\nabla\varphi)+\bfA_\zeta:\nabla^2 v+\bfA_\zeta:2e^{-\lambda_1 t_0}\overline\vr\frac{\nabla\varphi}{\varphi}\otimes\nabla\varphi-\bfA_\zeta:\overline\varrho e^{-\lambda_1 t_0}\nabla^2\varphi-\lambda_1 v\\
&\leq -b_\zeta(t,x,v,\overline\varrho e^{-\lambda_1 t_0}\nabla\varphi)+\overline\vr e^{-\lambda_1t_0}\bfA_\zeta:\Big(2\frac{\nabla\varphi}{\varphi}\otimes\nabla\varphi-\nabla^2\varphi\Big)-\lambda_1 v\\
&=\overline\varrho e^{-\lambda_1 t_0}\Big(\varphi\big( g_\zeta-\Div \bff_\zeta\big)-\lambda_1\varphi+\nabla\varphi\cdot\big(\bfg_\zeta-\bff_\zeta+\Div\bfA_\zeta\big)+\bfA_\zeta:\Big(2\frac{\nabla\varphi}{\varphi}\otimes\nabla\varphi-\nabla^2\varphi\Big)\Big).
\end{align*}
If we choose $\lambda_1$ large (depending on $\|g_\zeta\|_{L^\infty_{t,x}},\|\bfg_\zeta\|_{L^\infty_{t,x}},\|\nabla\bff_\zeta\|_{L^\infty_{t,x}},\|\nabla\bfA_\zeta\|_{L^\infty_{t,x}}$ and $\varphi$) this leads to a contradiction (note that $\overline\vr$ is non-negative by Theorem \ref{lem:warme} (b)).\\
Let us now assume that $x_0\in\partial\Omega$ and $t>0$. Then since $\bfA_\zeta(t_0,x_0)\nu(x_0)$ points outside $\Omega$ we have
\begin{align*}%\label{eq:0612}
0\leq \frac{\dd}{\dd s}v\big(t_0,x_0+s\bfA_\zeta(t_0,x_0)\nu(x_0)\big)\Big|_{s=0}=\nabla v(t_0,x_0)\cdot\bfA_\zeta(t_0,x_0)\nu(x_0).
\end{align*}
By \eqref{eq:LSU1} this implies
\begin{align}\label{eq:1412}
\begin{aligned}
0&\leq e^{-\lambda t_0}\big(\overline\varrho(t_0,x_0)\nabla\varphi(x_0)\cdot\bfA_\zeta(t_0,x_0)\nu(x_0)+\varphi\nabla\overline\varrho(x_0)\cdot\bfA_\zeta(t_0,x_0)\nu(x_0)\big)\\
&=\varphi(x_0)\overline\varrho(t_0,x_0) e^{-\lambda t_0}\frac{\nabla\varphi(x_0)\cdot\bfA_\zeta(t_0,x_0)\nu(x_0)}{\varphi(x_0)},
\end{aligned}
\end{align}
which yields a contradiction by \eqref{eq:testLSU2}. We conclude that the maximum of $v$ is attained at $(0,x_0)$ for some $x_0\in\overline\Omega$. By \eqref{eq:testLSU1} the estimate for the maximum follows. \\
%The proof for the minimum is completely analogous setting $v(t,x):=\varphi(x)e^{\lambda_1 t}\overline\varrho$ and replacing \eqref{eq:testLSU2} by $\frac{\nabla\varphi\cdot\bfA_\zeta\nu}{\varphi}\geq \lambda$.
Unfortunately, the approach above used for the maximum principle does not work to achieve a minimim principle. The reason is that we do not know a priori if $\overline\varrho$ is strictly positive at a potential minimum of $v$ at the boundary.
We multiply \eqref{eq:LSU1} by $-m(\xi+\overline\varrho)^{-m+1}$ where $0<\xi\ll1$ and $m\gg1$. This yields
\begin{align*}
\partial_t (\xi+\overline\varrho)^{-m}&=-m(g_\zeta-\Div\bff_\zeta)\overline\varrho(\xi+\overline\varrho)^{-m+1}-m\nabla \overline\varrho\cdot(\bfg_\zeta-\bff_\zeta)(\xi+\overline\varrho)^{-m+1}\\&-m\Div\big(\bfA_\zeta\nabla \overline\varrho\big)(\xi+\overline\varrho)^{-m+1}
\end{align*}
and
\begin{align*}
\frac{\dd}{\dt}&\int_\Omega (\xi+\overline\varrho)^{-m}\dx+m(m-1)\int_\Omega(\xi+\overline\varrho)^{-m}\bfA_\zeta\big(\nabla \overline\varrho,\nabla \overline\varrho\big)\dx\\&=-m\int_\Omega (g_\zeta-\Div\bff_\zeta)\overline\varrho(\xi+\overline\varrho)^{-m+1}\dx-m\int_\Omega(\bfg_\zeta-\bff_\zeta)\cdot\nabla \overline\varrho(\xi+\overline\varrho)^{-m+1}\dx=(I)+(II)
\end{align*}
using \eqref{eq:LSU1}$_2$. Using the boundedness of $\overline\varrho$ from (b) we obtain
\begin{align*}
(I)&\leq \,c\,m\int_\Omega \overline\varrho(\xi+\overline\varrho)^{-m+1}\dx\leq \,c\,m\int_\Omega \big(\xi+\overline\varrho)^{-m}\dx.
\end{align*}
The constant $c$ depends on $\|\zeta\|_{C^{2}_{t,x}},\sup J_\zeta^{-1}$, $\|\bfw\|_{L^\infty_{t,x}}$ and $\|\nabla\bfw\|_{L^\infty_{t,x}}$.
Similarly, we have for any $\kappa>0$
\begin{align*}
(II)&\leq\,c\,m\int_\Omega |\nabla \overline\varrho|(\xi+\overline\varrho)^{-m+1}\dx\\
%&\leq \,\frac{m(m-1)}{4}\int_\Omega  \bfA_\zeta(\nabla \overline\varrho,\nabla \overline\varrho)(\xi+\overline\varrho)^{-m}\dx+
%\,c\,\int_\Omega  \bfA_\zeta(\nabla J_\zeta^{-1},\nabla J_\zeta^{-1})(\xi+\overline\varrho)^{-m+2}\dx\\
&\leq \,\kappa\, m(m-1)\int_\Omega |\nabla \overline\varrho|^2(\xi+\overline\varrho)^{-m}\dx+
\,c(\kappa)\,\int_\Omega (\xi+\overline\varrho)^{-m}\dx\\
&\leq \,c\,\kappa \,m(m-1)\int_\Omega  \bfA_\zeta(\nabla \overline\varrho,\nabla \overline\varrho)(\xi+\overline\varrho)^{-m}\dx+
\,c\,\int_\Omega (\xi+\overline\varrho)^{-m}\dx
\end{align*}
with $c=c(\|\zeta\|_{C^{2,\alpha}_{t,x}},\|\bfg_\zeta\|_{L^\infty_{t,x}},\|\bff_\zeta\|_{L^\infty_{t,x}})$. 
If we absorb now the terms containing $\bfA_\zeta$ and apply Gronwall's lemma we obtain
\begin{align*}
\int_\Omega (\xi+\overline\varrho(t)))^{-m}\dx\leq e^{Cm}\int_\Omega (\xi+\overline\varrho_0)^{-m}\dx.
\end{align*}
The constant $C$ depends on $\|\zeta\|_{C^{2}_{t,x}},\|\partial_t\nabla^2\zeta\|_{C^{\alpha}_{t,x}},\|\nabla^3\zeta\|_{C^{\alpha}_{t,x}}, \sup J_\zeta^{-1}$, $\|\bfw\|_{C^{1,\alpha}_{t,x}}$, $\|\partial_t\nabla\bfw\|_{C^{\alpha}_{t,x}}$ and $\|\nabla^2\bfw\|_{C^{\alpha}_{t,x}},$ but is independent of $m$.
Taking the $m$-th root shows
\begin{align*}
\bigg(\int_\Omega \Big(\frac{1}{\xi+\overline\varrho(t)}\Big)^{-m}\dx\bigg)^{\frac{1}{m}}\leq e^{C}\bigg(\int_\Omega (\xi+\overline\varrho_0)^{-m}\dx\bigg)^{\frac{1}{m}}.
\end{align*}
Passing with $m\rightarrow \infty$ implies
\begin{align*}
\sup_\Omega\frac{1}{\xi+\overline\varrho(t)}\leq e^C\sup_\Omega\frac{1}{\xi+\overline\varrho_0}
\end{align*}
or, equivalently,
\begin{align*}
 e^C\inf_\Omega(\xi+\overline\varrho_0)\leq \inf_{\Omega}\big(\xi+\overline\varrho(t)\big).
\end{align*}
Consequently, passing with $\xi\rightarrow0$ we have
$ e^C\min_\Omega \overline\varrho_0\leq \overline\varrho(t,x)$
 for all $(t,x)\in \overline I\times\overline\Omega$. Thus, transforming back to $\varrho$, (c) is shown and the proof is complete.
\end{proof}

\subsection{The internal energy equation}
\label{sec:int}
%First of all we assume that the function $P$ in \eqref{m98} equals to $z^\gamma$ (rather than only the asymptotic behaviour from \eqref{md1}). Under this assumption we have
%\begin{align*}
%p(\varrho,\vartheta)=\varrho^\gamma+\frac{a}{3}\vartheta^4,\quad e(\varrho,\vartheta)=\frac{1}{\gamma-1}\varrho^{\gamma-1}+a\frac{\vartheta^4}{\varrho}
%\end{align*}
%and in accordance with Gibb's relation
%\begin{align*}
%s(\varrho,\vartheta)=\frac{4a}{3}\frac{\vartheta^3}{\varrho}.
%\end{align*}
%This gives the following for the regularized system:
%\begin{equation} \label{neu'}
%p_\delta(\varrho,\vartheta)=\varrho^\gamma+\delta\varrho^\beta+\frac{a}{3}\vartheta^4,\quad e_\delta(\varrho,\vartheta)=\frac{1}{\gamma-1}\varrho^{\gamma-1}+\frac{\delta}{\beta-1}\varrho^{\beta-1}+a\frac{\vartheta^4}{\varrho}.
%\end{equation}
%Note that this differs from \eqref{neu} (otherwise there are mixed terms again, we have to check that this is fine.)
The artificial viscosity of the mollified continuity equation produces some dissipative forces that will turn into heat. This is captured by the internal energy equation which we will solve next. For that we introduce the artificial dissipation as
\begin{align}
\label{eq:rdiss}
\rdiss(\vr)=(\gamma \vr^{\gamma-2}+\delta\beta \vr^{\beta - 2}) |\nabla \vr |^2
\end{align}
Indeed, we observe that by the renormalized continuity equation~\eqref{eq:renormz}, we find that
\begin{align}
\label{eq:pres}
\begin{aligned}
\partial_t \big(\tfrac{1}{\gamma-1}\varrho^\gamma\big)+\Div\big(\tfrac{1}{\gamma-1}\varrho^\gamma\bfw\big)&=\Delta\big(\tfrac{1}{\gamma-1}\varrho^\gamma\big)-\varrho^\gamma \Div\bfw - \ep
\gamma \vr^{\gamma-2} |\nabla \vr |^2
\end{aligned}
\end{align}
Hence the (regularized) internal energy part $\tilde e_R(\vr, \vt):=e_R(\vr,\vt)+c_v\vt=\frac{a\vt^4}{\vr}+c_v\vt$
with $\tilde p_R(\vr, \vt):=p_R(\vr,\vt)+\seb{\vr}\vt=\frac{a\vt^4}{3}+\vr\vt$ is required to satisfy
\begin{align}\label{m119a}
\begin{aligned}
\partial_t\big(\vr&\tilde e_R(\vr, \vt)\big)
+\Div\big(\vr\tilde e_R(\vr, \vt)\bfw\big)-
\Div\big(\varkappa_\delta (\vt) \nabla \vt\big) 
 \\
& = \bfS^{\ep}(\vartheta,\nabla\bfw):\nabla\bfw -
\tilde p_R(\vr,\vt)\Div\bfw 
 \\
 \quad &+  \ep
\rdiss(\vr)  +
\delta \frac{1}{{\vt}^2} - \ep \vt^5+\vr H\quad\text{in}\quad I\times\Omega_\zeta,\\
\partial_{\nu_\zeta}&\vt\big|_{\partial\Omega_\zeta}=0\quad\text{on}\quad I\times\partial\Omega_\zeta
\end{aligned}
\end{align}
%
%The regularized internal energy equation reads as
%\begin{align}\label{m119a}
%\begin{aligned}
%\partial_t&\big(\varrho e(\varrho, \vartheta) \big)
%+\Div\big(\varrho e(\varrho, \vartheta)\bfw\big)-
%\Div\big(\varkappa_\delta (\vt) \nabla \vt\big) 
% \\
%& = \bfS^{\ep}(\vartheta,\nabla\bfw):\nabla\bfw -
%p_\delta(\varrho, \vartheta) \Div\bfw 
% \\
% \quad &+  \seb{\ep
%\rdiss(\vr)}  +
%\delta \frac{1}{{\vt}^2} - \ep \vt^5+\vr H\quad\text{in}\quad I\times\Omega_\zeta,\\
%\partial_{\nu_\zeta}&\vt\big|_{\partial\Omega_\zeta}=0\quad\text{on}\quad I\times\partial\Omega_\zeta
%\end{aligned}
%\end{align}
and we have $\vt(0)=\vt_0$ (note that $\varkappa_\delta$ and $\bfS^\varepsilon$ are defined in \eqref{neu}).

The equation does indeed uniquely define $\vt$ provided $\vr$ exists and is satisfyingly smooth. Certainly $\vr$ constructed by Theorem~\ref{thm:regrho} does inherit the necessary smoothness.
%We assume that $\mathfrak{f}:I\times\partial\Omega_\zeta\rightarrow\R$ is a measurable function.
% such that
%\begin{align}\label{eq:2207}
%0\leq \mathfrak f(t,x)\leq \overline c_{\mathfrak f}
%\end{align}
%for a positive constant $c_{\mathfrak f}$. 
Accordingly and similar to Theorem \ref{thm:regrho} we have the following result concerning a classical solution to \eqref{m119a}. 
\begin{theorem}\label{thm:regtheta}
Let $\zeta\in C^{2,\alpha}(\overline{I}\times \omega,[-\frac{L}2,\frac{L}2])$ with $\alpha\in(0,1)$ be the function describing the boundary. Suppose additionally that $\partial_t\nabla^2\zeta$ as well as $\nabla^3\zeta$ belong to the class $C^{\alpha}(\overline I\times \omega)$ and suppose that $J_\zeta:=\mathrm{det}\nabla\bfPsi_\zeta$ is strictly positive. Assume that $\bfw\in C^{1,\alpha}(\overline I\times\overline\Omega_\zeta)$ such that $\partial_t\nabla\bfw$ and $\nabla^2\bfw$ belong to the class $C^{\alpha}(\overline I\times \overline\Omega_\zeta)$ and $\tr_{\zeta}\bfw=\partial_t\zeta\nu$. %\seb{Assume that $\vr$ is the unique, strictly positive and smooth solution with the above data constructed in Theorem~\ref{thm:regrho}, with $\vr_0\in C^{2,\alpha}(\overline{\Omega}_{\zeta(0)})$, strictly positive.}

 Assume further that
 % $\vt_0\in C^{2,\alpha}(\overline\Omega_{\zeta(0)})$, $\vt_0$ strictly positive, 
 $\vr,H\in C^{1,\alpha}(\overline I\times\overline\Omega_\zeta;[0,\infty))$ with $\nabla^2\varrho\in C(\overline I\times\overline\Omega_\zeta)$ that $\partial_{\nu_\zeta}\varrho\big|_{\partial\Omega_\zeta}=0$  and $\varrho$ strictly positive on $I\times\partial\Omega_\zeta $.
\begin{enumerate}
\item There is a unique classical solution $\vt$ to \eqref{m119a}
which belongs to the regularity class
\begin{align*}
\mathcal Z^{I}_\zeta:=\big\{z \in C^1(\overline I\times\overline\Omega_\zeta):\ \nabla^2z \in C(\overline I\times\overline\Omega_\zeta)\big\}.
\end{align*}
In particular, we have
\begin{align*}
\|\vt\|_{C^1_{t,x}}&+\|\nabla^2\vt\|_{C_{t,x}}\leq\,c\Big(\vt_0,\zeta,\sup J_\zeta^{-1},\bfw,\varrho,H\Big),
\end{align*}
with dependence via the (semi-)norms in the affirmative function spaces.
\item We have the estimate
\begin{align*}
&\min\Big\{C^{-1}\min_{\Omega_{\zeta(0)}}\vt_0,1\Big\}\leq\min_{I\times\Omega_\zeta} \vt\leq\max_{I\times\Omega_\zeta} \vt\leq \max\Big\{C\max_{\Omega_{\zeta(0)}}\vt_0,1\Big\},
\end{align*}
where $C=C(\zeta,\sup J_\zeta^{-1},\bfw,\varrho,H)$ with dependence via the (semi-)norms in the affirmative function spaces.
\end{enumerate}
\end{theorem}
\begin{proof}
Equation \eqref{m119a} contains several nonlinear terms which blow up for small or large values of $\vt$. Hence we replace them
with regularized versions.
Let $\chi_\ell\in C^\infty([0,\infty))$ with $\chi_\ell(Z)= Z$ for $Z\in[1/\ell,\ell]$ and
$c\ell^{-1}\leq \chi_\ell\leq C\ell$ for some positive constants $c,C$ and $\ell\gg1$. We also define the function
\begin{align*}
\mathfrak b_{(t,x)}(\vt):=\vr(t,x) \tilde e_R(\vr(t,x),\vt)=a\vartheta^4+c_v\varrho(t,x)\vartheta,\quad\vartheta\geq0.
\end{align*}
Since $\mathfrak{b}_{(t,x)}'(\vartheta)=4a\vartheta^3+c_v\varrho(t,x)$
is strictly positive by the assumptions on $\varrho$ the inverse
$\mathfrak{b}_{(t,x)}^{-1}$ satisfies
\begin{align}\label{eq:0411}
(\mathfrak{b}_{(t,x)}^{-1})'\leq\,c(\varrho).
\end{align}
We define
\begin{align*}
\bfS^{\ep,\ell}(t,x,\vt,\nabla\bfw)&=\bfS^\ep((\mathfrak{b}_{(t,x)}^{-1}(\chi_\ell(\mathfrak{b}_{(t,x)}(\vartheta)),\nabla\bfw),\,\, \varkappa_\delta^\ell(t,x,\vt)=
\frac{\varkappa_\delta(\mathfrak{b}_{(t,x)}^{-1}(\chi_\ell(\mathfrak{b}_{(t,x)}(\vartheta)))}{b'_{(t,x)}(\mathfrak{b}_{(t,x)}^{-1}(\chi_\ell(\mathfrak{b}_{(t,x)}(\vartheta)))}b'_{(t,x)}(\vartheta),
%\\ f^\ell(\vt)&=f^\ell(t,x,\vt)=(1+\sqrt[4]{\chi_\ell(\vt^4)})\mathfrak f(t,x),
\end{align*}
and consider the equation
\begin{align}\label{m119aL}
\begin{aligned}
\partial_t&(\mathfrak{b}_{(t,x)}(\vt))
+\Div\big(\mathfrak{b}_{(t,x)}(\vt)\bfw\big)-
\Div\big(\varkappa_\delta^\ell (\vt) \nabla \vt\big)  \\
& = \bfS^{\ep,\ell}(\vartheta,\nabla\bfw):\nabla\bfw -
\tilde p_R(\vr,\mathfrak{b}_{(t,x)}^{-1}(\chi_\ell(\mathfrak{b}_{(t,x)}(\vartheta))) \Div\bfw  + \rdiss(\vr)% \ep\delta \beta \varrho^{\beta - 2} |\nabla \varrho |^2 
\\& +
\delta \big({\mathfrak{b}_{(t,x)}^{-1}(\chi_\ell(\mathfrak{b}_{(t,x)}(\vartheta)))}\big)^{-2} - \ep (\mathfrak{b}_{(t,x)}^{-1}(\chi_\ell(\mathfrak{b}_{(t,x)}(\vartheta)))^{5}+\vr H\quad\text{in}\quad I\times\Omega_\zeta,\\
\partial_{\nu_\zeta}&\vt\big|_{\partial\Omega_\zeta}=0\quad\text{on}\quad I\times\partial\Omega_\zeta
\end{aligned}
\end{align}
and we have $\vt(0)=\vt_0$.
We will show that a solution $\vr$ to \eqref{m119aL} exists and that
\begin{align}\label{eq:maxl}
\max_{I\times\Omega_\zeta} \vt\leq \max\Big\{C\max_{\Omega_{\zeta(0)}}\vt_0,1\Big\}
\end{align}
as well as
\begin{align}\label{eq:minl}
\min_{I\times\Omega_\zeta} \vt\geq \min\Big\{C^{-1}\min_{\Omega_{\zeta(0)}}\vt_0,1\Big\}
\end{align}
with $C=C(\|\zeta\|_{C^{2,\alpha}_{t,x}},\|(\partial_t\nabla^2\zeta,\nabla^3\zeta)\|_{C^{\alpha}_{t,x}},\sup J_\zeta^{-1},\|\varrho\|_{C^1_{t,x}},\|\bfw\|_{C^{1,\alpha}_{t,x}},\|\nabla^2\bfw\|_{C^\alpha_{t,x}},\|H\|_{L^\infty_{t,x}})$ independent of $\ell$.
%\end{enumerate}
Consequently, the cut-offs in \eqref{m119aL} are not seen for $\ell$ are enough and we obtain the result for the original problem \eqref{m119a}.
Arguing as in the proof of Theorem \ref{thm:regrho} we can transform \eqref{m119aL} to the reference domain. For this purpose it is useful to work with the weak formulation
\begin{align*}
\nonumber\int_I\bigg(\frac{\dd}{\dt}\int_{ \Omega_\zeta}\mathfrak{b}_{(t,x)}(\vt) \psi\dx&-\int_{ \Omega_\zeta}\Big(\mathfrak{b}_{(t,x)}(\vt))\,\partial_t\psi
+\mathfrak{b}_{(t,x)}(\vt)\bfw\cdot\nabla\psi\Big)\dx\bigg)\dt\\&+
\int_I\int_{ \Omega_\zeta}\varkappa^\ell_\delta (\vt) \nabla \vt \cdot\nabla\psi\dxt \\
& = \int_I\int_{\Omega_{\zeta}}\Big[\bfS^{\ep,\ell}(\vt,\nabla\bfw):\nabla\bfw -
\tilde p_R(\vr,\mathfrak{b}_{(t,x)}^{-1}(\chi_\ell(\mathfrak{b}_{(t,x)}(\vartheta)))   \Div \bfw \Big]\,\psi \dxt
\\ &+ \int_I\int_{\Omega_{\zeta}}\Big[ \rdiss(\vr)  +
\delta (\mathfrak{b}_{(t,x)}^{-1}(\chi_\ell(\mathfrak{b}_{(t,x)}(\vartheta)))^{-2} - \ep (\mathfrak{b}_{(t,x)}^{-1}(\chi_\ell(\mathfrak{b}_{(t,x)}(\vartheta)))^4 +\vr H\Big]\,\psi\dxt
\end{align*}
for all $\psi\in C^\infty(\overline I\times\R^3)$. Actually we will solve the PDE for $Z:=(\mathfrak{b}_{(t,x)}(\vt))\circ\bfPsi_\zeta$, which is defined on the reference configuration and hence a PDE for a cylindrical time-space domain can be considered. Accordingly we are setting $\psi=\overline\psi\circ\bfPsi_\zeta^{-1}$ for some $\overline\psi\in C^\infty(\overline I\times\R^3)$,
$\overline\varrho=\varrho\circ\bfPsi_\zeta$, $\overline\bfw=\bfw\circ\bfPsi_\zeta$, $\overline H=H\circ\bfPsi_\zeta$ and $\overline\vt=\vt\circ\bfPsi_\zeta$
this is equivalent to
\begin{align*}
&\int_I\frac{\dd}{\dt}\int_{ \Omega_\zeta}Z\circ\bfPsi_\zeta^{-1} \overline\psi\circ\bfPsi_\zeta^{-1}\dxt\\
&-\int_I\int_{ \Omega_\zeta}Z\circ\bfPsi_\zeta^{-1}\,\Big(\partial_t\overline\psi\circ \bfPsi_\zeta^{-1}+\nabla\overline\psi\circ \bfPsi_\zeta^{-1}\cdot \partial_t\bfPsi_\zeta^{-1}\Big)\dxt\\
&+\int_I\int_{\Omega_\zeta}Z\circ\bfPsi_\zeta^{-1}\overline\bfw\circ \bfPsi_\zeta^{-1}\cdot\nabla \bfPsi_\zeta^{-1}\nabla \overline\psi\circ\bfPsi_\zeta^{-1}\dx\dt\\&+
\int_I\int_{ \Omega_\zeta}\varkappa_\delta^\ell (\overline\vt\circ\bfPsi_\zeta^{-1})\big(\nabla\bfPsi_\zeta^{-1}\big)^T \nabla\bfPsi_\zeta^{-1}\nabla \overline\vt\circ\bfPsi_\zeta^{-1} \cdot\nabla\overline\psi\circ\bfPsi_\zeta^{-1}\dxt \\
%&-\int_I\int_{\partial\Omega_\zeta}\mathfrak{f}\,\overline\psi\circ\bfPsi_\zeta^{-1}\dH\dt\\
 =  &\int_I\int_{\Omega_{\zeta}}(\nabla\bfPsi_\zeta^{-1})^T\bfS^{\ep,\ell}(\overline\vartheta\circ\bfPsi_\zeta^{-1},\nabla\bfPsi_\zeta^{-1}\nabla\overline\bfw\circ\bfPsi_\zeta^{-1}) \nabla\overline\bfw\circ\bfPsi_\zeta^{-1}\,\overline\psi\circ\bfPsi_\zeta^{-1} \dxt
 \\ 
& - \int_I\int_{\Omega_{\zeta}}
\tilde p_R(\overline\varrho\circ\bfPsi_\zeta^{-1}, \mathfrak{b}_{(t,x)}^{-1}(\chi_\ell(\mathfrak{b}_{(t,x)}(\overline\vartheta\circ\bfPsi_\zeta^{-1}))\nabla\overline\bfw\circ\bfPsi_\zeta^{-1}:\big(\nabla\bfPsi_\zeta^{-1}\big)^T \,\overline\psi\circ\bfPsi_\zeta^{-1} \dxt 
\\ 
&+ \int_I\int_{\Omega_{\zeta}}\rdiss(\overline{\vr}\circ\bfPsi_\zeta^{-1})\,\overline\psi\circ\bfPsi_\zeta^{-1}\dxt\\
&+ \int_I\int_{\Omega_\zeta}
\delta \big(\mathfrak{b}_{(t,x)}^{-1}(\chi_\ell(Z\circ\bfPsi_\zeta^{-1}\big)\big)\big)^{-2}\,\overline\psi\circ\bfPsi_\zeta^{-1}\dxt\\
&+\int_I\int_{\Omega_\zeta}\Big[ - \ep \big(\mathfrak{b}_{(t,x)}^{-1}\big(\chi_\ell\big(Z\circ\bfPsi_\zeta^{-1}\big)\big) \big)^{5}+\overline\vr\overline H\Big]\,\overline\psi\circ\bfPsi_\zeta^{-1}\dxt
\end{align*}
and, setting $J_\zeta=\mathrm{det}\nabla\bfPsi_\zeta$ and $\overline{\rdiss}=(\delta\beta \overline\vr^{\beta - 2}+\gamma \overline\vr^{\gamma-2}) |\nabla\bfPsi_\zeta^{-1}\nabla \overline\vr|^2$, we find
\begin{align*}
&\quad \int_I\frac{\dd}{\dt}\int_{ \Omega}J_\zeta Z\overline\psi\dxt-\int_I\int_{ \Omega}J_\zeta Z\,\Big(\partial_t\overline\psi+\nabla\overline\psi\cdot \partial_t\bfPsi_\zeta^{-1}\circ\bfPsi_\zeta\Big)\dxt\\
&\quad +\int_I\int_{\Omega}J_\zeta Z\overline\bfw\cdot\big(\nabla \bfPsi_\zeta\big)^{-1}\nabla \overline\psi\dx\dt
%\\
%-\int_I\int_{\partial\Omega}J_\zeta\overline{\mathfrak{f}}\,\overline\psi\dH\dt
+\int_I\int_{ \Omega}J_\zeta\varkappa^\ell_\delta (\overline\vt)\big(\nabla\bfPsi_\zeta\big)^{-T} \big(\nabla\bfPsi_\zeta\big)^{-1}\nabla \overline\vt \cdot\nabla\overline\psi\dxt \\
& = \int_I\int_{\Omega}J_\zeta(\nabla\bfPsi_\zeta)^{-T}\bfS^{\ep,\ell}(\overline\vartheta,(\nabla\bfPsi_\zeta)^{-1}\nabla\bfw):\nabla\overline\bfw\,\overline\psi \dxt
 \\& 
- \int_I\int_{\Omega}
J_\zeta \tilde p_R(\overline\varrho, \mathfrak{b}_{(t,x)}^{-1}(\chi_\ell(Z)))\nabla\overline\bfw:\big(\nabla\bfPsi_\zeta\big)^{-T} \,\overline\psi \dxt \\ 
&\quad + \varepsilon\int_I\int_{\Omega}J_\zeta \overline{\rdiss}\,\overline\psi\dxt
%\\&
+ \int_I\int_{\Omega}
\delta J_\zeta\big(\mathfrak{b}_{(t,x)}^{-1}(\chi_\ell(Z))\big)^{-2}\,\overline\psi\dxt\\
&\quad +\int_I\int_{\Omega}J_\zeta\Big[ - \ep \big(\mathfrak{b}_{(t,x)}^{-1}\big(\chi_\ell(Z)\big) \big)^{5}+\overline\vr\overline H\Big]\,\overline\psi\dxt
\end{align*}
for all $\overline\psi\in C^\infty(\overline I\times\overline\Omega)$.
Again we replace $\overline\psi$ by $\overline\psi/J_\zeta$ to obtain
%\begin{align*}
%&\int_I\frac{\dd}{\dt}\int_{ \Omega} Z\overline\psi\dxt-\int_I\int_{ \Omega} Z\,\Big(\partial_t\overline\psi+\nabla\overline\psi\cdot \partial_t\bfPsi_\zeta^{-1}\circ\bfPsi_\zeta\Big)\dxt\\
%&\quad +\int_I\int_{\Omega}J_\zeta Z\overline\bfw\cdot\big(\nabla \bfPsi_\zeta\big)^{-1}\nabla \overline\psi\dx\dt
%+\int_I\int_{ \Omega}J_\zeta\varkappa^\ell_\delta (\overline\vt)\big(\nabla\bfPsi_\zeta\big)^{-T} \big(\nabla\bfPsi_\zeta\big)^{-1}\nabla \overline\vt \cdot\nabla\overline\psi\dxt \\
%& = \int_I\int_{\Omega}J_\zeta(\nabla\bfPsi_\zeta)^{-T}\bfS^{\ep,\ell}(\overline\vartheta,(\nabla\bfPsi_\zeta)^{-1}\nabla\bfw):\nabla\overline\bfw\,\overline\psi \dxt
%- \int_I\int_{\Omega}
%J_\zeta p_M(\overline\varrho, \overline\vartheta)\nabla\overline\bfw:\big(\nabla\bfPsi_\zeta\big)^{-T} \,\overline\psi \dxt \\ 
%&\quad + \int_I\int_{\Omega}J_\zeta \overline{\rdiss}\,\overline\psi\dxt
%+ \int_I\int_{\Omega}
%\delta J_\zeta\big(\mathfrak{b}_{(t,x)}^{-1}(\chi_\ell(\mathfrak{b}_{(t,x)}\big(\overline\vartheta\big)\big)\big)\big)^{-2}\,\overline\psi\dxt\\
%&\quad +\int_I\int_{\Omega}J_\zeta\Big[ - \ep \big(\mathfrak{b}_{(t,x)}^{-1}\big(\chi_\ell\big(\mathfrak{b}_{(t,x)}\big(\overline\vartheta\big)\big)\big) \big)^{5}+\overline\vr\overline H\Big]\,\overline\psi\dxt
%\end{align*}
\begin{align*}
&\int_I\frac{\dd}{\dt}\int_{ \Omega}Z\overline\psi\dxt-\int_I\int_{ \Omega}Z\,\partial_t\overline\psi\dxt
\\&
=\int_I\int_{ \Omega}Z\,\nabla\overline\psi\cdot \partial_t\bfPsi_\zeta^{-1}\circ\bfPsi_\zeta\dxt
-\int_I\int_{\Omega}Z\overline\bfw\cdot\big(\nabla \bfPsi_\zeta\big)^{-1}\nabla \overline\psi\dx\dt
\\& \quad
-\int_I\int_{ \Omega}\varkappa_\delta^\ell (\overline\vt)\big(\nabla\bfPsi_\zeta\big)^{-T} \big(\nabla\bfPsi_\zeta\big)^{-1}\nabla \overline\vt \cdot\nabla\overline\psi\dxt
\\&
\quad -
\int_I\int_{ \Omega}J_\zeta\varkappa_\delta^\ell (\overline\vt)\big(\nabla\bfPsi_\zeta\big)^{-T} \big(\nabla\bfPsi_\zeta\big)^{-1}\nabla \overline\vt \cdot\nabla J_\zeta^{-1}\overline\psi\dxt
\\&\quad
 + \int_I\int_{\Omega}(\nabla\bfPsi_\zeta)^{-T}\bfS^{\ep,\ell}(\overline\vartheta,(\nabla\bfPsi_\zeta)^{-1}\nabla\overline\bfw):\nabla\overline\bfw \,\overline\psi \dxt  
 \\&
 \quad - \int_I\int_{\Omega}
 \tilde p_R(\overline\varrho, \mathfrak{b}_{(t,x)}^{-1}(\chi_\ell(Z)))\nabla\overline\bfw:\big(\nabla\bfPsi_\zeta\big)^{-T} \,\overline\psi \dxt 
 %\\ & \quad
 + \int_I\int_{\Omega}\Big[ \overline{\rdiss}+\overline\vr\overline H \Big]\,\overline\psi\dxt
 \\&
 \quad + \int_I\int_{\Omega_{\zeta}}
\big(\mathfrak{b}_{(t,x)}^{-1}(\chi_\ell(Z))\big)^{-2}\,\overline\psi\dxt
-\int_I\int_{\Omega} \ep \big(\mathfrak{b}_{(t,x)}^{-1}\big(\chi_\ell(Z)\big) ^{5}\,\overline\psi\dxt
\\
&\quad 
+\int_I\int_{ \Omega} J_\zeta  Z\,\Big(\partial_t J_\zeta^{-1}+\nabla J_\zeta^{-1}\cdot \partial_t\bfPsi_\zeta^{-1}\Big)\overline\psi\dxt
%\\
-\int_I\int_{\Omega}J_\zeta Z \overline\bfw\cdot \big(\nabla \bfPsi_\zeta\big)^{-1}\nabla J_\zeta^{-1}\overline\psi\dx\dt.
\end{align*}
Recalling \eqref{m98}, \eqref{mp8a} and that $Z=\mathfrak{b}_{(t,x)}(\overline\vt)$ we set
\begin{align*}
g_\zeta^\ell(Z)&=(\nabla\bfPsi_\zeta)^{-T}\bfS^{\ep,\ell}(\mathfrak{b}_{(t,x)}^{-1}(\chi_\ell(Z)),(\nabla\bfPsi_\zeta)^{-1}\nabla\overline\bfw):\nabla\overline\bfw - 
 \tilde p_R(\overline\varrho,\mathfrak{b}_{(t,x)}^{-1}(\chi_\ell(Z)))\nabla\overline\bfw:\big(\nabla\bfPsi_\zeta\big)^{-1}  \\ 
&+\ep
\delta \beta \overline\varrho^{\beta - 2} |\big(\nabla\bfPsi_\zeta\big)^{-1}\nabla \overline\varrho |^2+
\delta (\mathfrak{b}_{(t,x)}^{-1}(\chi_\ell(Z)))^{-2} - \ep (\mathfrak{b}_{(t,x)}^{-1}(\chi_\ell(Z)))^{5} +\overline\vr\overline H\\
&+Z\,\Big(\partial_t J_\zeta^{-1}+\nabla J_\zeta^{-1}\cdot \partial_t\bfPsi_\zeta^{-1}-\overline\bfw\cdot \big(\nabla \bfPsi_\zeta\big)^{-1}\nabla J_\zeta^{-1}\Big)J_\zeta,\\
\bfg_\zeta^\ell(Z)&=-J_\zeta\frac{\varkappa_\delta (\mathfrak{b}_{(t,x)}^{-1}(\chi_\ell(Z)))}{b'_{(t,x)}((\mathfrak{b}_{(t,x)}^{-1}(\chi_\ell(Z))))}\nabla J_\zeta^{-1}\big(\nabla\bfPsi_\zeta\big)^{-T} \big(\nabla\bfPsi_\zeta\big)^{-1}\\&+\frac{\varkappa_\delta ((\mathfrak{b}_{(t,x)}^{-1}(\chi_\ell(Z))))}{\mathfrak{b}_{(t,x)}'((\mathfrak{b}_{(t,x)}^{-1}(\chi_\ell(Z))))}\big(\nabla\bfPsi_\zeta\big)^{-T} \big(\nabla\bfPsi_\zeta\big)^{-1}\nabla\varrho(t,x) \mathfrak{b}_{(t,x)}^{-1}(Z),\\
\bff_\zeta^\ell(Z)&= Z\Big( \partial_t\bfPsi_\zeta^{-1}
-\big(\nabla \bfPsi_\zeta\big)^{-1}\overline\bfw\Big),\\
\bfA_\zeta^\ell(Z)&=\frac{\varkappa_\delta ((\mathfrak{b}_{(t,x)}^{-1}(\chi_\ell(Z))))}{\mathfrak{b}_{(t,x)}'((\mathfrak{b}_{(t,x)}^{-1}(\chi_\ell(Z))))}\big(\nabla\bfPsi_\zeta\big)^{-T} \big(\nabla\bfPsi_\zeta\big)^{-1},
%\quad f^\ell(Z)=(1+\chi_\ell((Z)^{1/4}))\overline{\mathfrak f}
\end{align*}
such that the equation becomes
%\begin{align}\label{eq:thetaref}
%\begin{aligned}
%\int_I\frac{\dd}{\dt}\int_{\Omega}aZ\,\overline\psi\dx\dt
%&=\int_I\int_{\Omega}g_\zeta^\ell(\overline\vt) \overline\psi\dxt+\int_I\int_{\Omega}\nabla\overline\vt\cdot\bfg^\ell_\zeta(\overline\vt) \overline\psi\dxt
%\\&+\int_I\int_{\Omega}\bff_\zeta^\ell(\overline\vt)\cdot\nabla\overline\psi\dxt-\int_I\int_{\Omega}\bfA_\zeta^\ell(\overline\vt)\nabla\overline\vt\cdot\nabla\overline\psi\dxt.
%\end{aligned}
%\end{align}
%Choosing $\overline\psi$ from $C^\infty_c(I\times\Omega)$ shows
\begin{align}\label{eq:thetapde}
\partial_tZ &= g_\zeta^\ell(Z)+\nabla Z\cdot\bfg_\zeta^\ell(Z)-\Div(\bff_\zeta^\ell(Z))+\Div\big(\bfA_\zeta^\ell(Z)\nabla Z\big)\quad\text{in}\quad I\times\Omega,
\\
\label{eq:thetaboundary}
\nu\cdot\bfA^\ell_\zeta(Z) \nabla Z&=0\quad \text{on}\quad I\times\partial\Omega, 
\\
\label{eq:Zinitial}
Z(0)&=(\vt_0^4+c_v\vt_0\vr(0))\circ\bfPsi_\zeta \quad\text{in}\quad \Omega.
\end{align}
Here the boundary conditions are a direct consequence of the weak formulation and the fact that $\bff_\zeta^\ell\cdot \nu=0$ on the boundary, by the assumed coupling of $\mathbf{w}$ and $\partial_t\zeta$ at the boundary, cf. \eqref{1611}.
 %transformation performed in \eqref{eq:LSU1}. In particular they read in terms of $\overline{\vt}$ that
%\begin{align*}
%\nu\tilde\bfA_\zeta^\ell(\overline\vt)\cdot\nabla\overline\vt=0.
%\end{align*}
%Since several terms are unbounded we replace $Z$ at several occasions with
%$\chi_L(Z)$, where $\chi_L\in C^\infty([0,\infty))$ with $\chi_L(Z)= Z$ for $Z\in[1/L,L]$ and
%$cL^{-1}\leq \chi_L\leq CL$ for some positive constants $c,C$. We set
%\begin{align*}
%g_\zeta^L(Z)=g_\zeta(\chi_L(Z)),\quad \bfg_\zeta^L(Z)=\bfg_\zeta(\chi_L(Z)),\quad \bff_\zeta^L(Z)=\bff_\zeta(\chi_L(Z)),\quad \bfA_\zeta^L(Z)=\bfA_\zeta(\chi_L(Z)),
%\end{align*}
%and consider the equation
%\begin{align}\label{eq:LSU1'0}
%\begin{aligned}
%\partial_taZ&= g_\zeta^L(Z)+\nabla Z\cdot\bfg^L_\zeta(Z)-\Div(\bff^L_\zeta(Z))+\Div\big(\bfA^L_\zeta(Z)\nabla Z\big),\\
%&\qquad\qquad-\nu\bfA_\zeta^L(Z)\cdot\nabla Z+\bff_\zeta^L(Z)\cdot\nu=0.
%\end{aligned}
%\end{align}
We can further rewrite \eqref{eq:thetapde} as
\begin{align}\label{eq:LSU1'}
\begin{aligned}
\partial_t Z
&+b_\zeta^\ell(t,x,Z,\nabla Z)=\bfA^\ell_\zeta(Z):\nabla^2Z\quad\text{in}\quad I\times\Omega,
%\\
%&\qquad\nu\bfA^\ell_\zeta(Z)\cdot\nabla Z=0\quad \text{on}\quad I\times\partial\Omega,
\end{aligned}
\end{align}
where
\begin{align*}
 b_\zeta^\ell(t,x,u,\bfU)&=-g_\zeta^\ell(u)+\Div_x\bff_\zeta^\ell(u)+\big(-\bfg^\ell_\zeta(u)+\partial_u\bff^\ell_\zeta(u)-\Div_x\bfA^\ell_\zeta(u)\big)\cdot\bfU\\&-\partial_u \bfA_\zeta^\ell(u)(\bfU,\bfU).
%\Psi_\zeta^\ell(t,x,u)&=-\overline{\mathfrak{f}}.
\end{align*}
Analogous to the proof of Theorem~\ref{thm:regrho} (a) we can use the theory from \cite[Thm. 7.2, 7.3 \& 7.4 Chapter V]{LSU} to infer the existence of a unique classical solution $Z$ to \eqref{eq:LSU1'} with the required regularity. We will below show that the same $Z$ is actually a solution for all $\ell$ that are sufficiently large. Observe already, that by defining $\vt$ as the respective reverse transformation of $Z$ implies the existence of a respective regular $\vt$ claimed in (a) (provided we can show that the cut-offs in \eqref{m119aL} are not seen). This follows by showing respective independent upper and lower bounds on $Z$ which then eventually imply (b), as well. 
%Note that different to setting of Theorem \ref{thm:regrho} we do not have homogeneous boundary conditions. Due to this the constant also depends on the $\alpha/2$-H\"older-semi-norm of $\partial_t\Psi_\zeta$ and the $\alpha$-H\"older-semi-norm of $\nabla_x\Psi_\zeta$ with respect to $x$; the linear-growth constants of $\Psi_\zeta(t,x,u)$ with respect to the third variable and
%the constants in
%\begin{align*}
%-u\Psi_\zeta(t,x,u)\leq\,c_4u^2+c_5\quad \forall(t,x,u)\in I\times\Omega\times\R.
%\end{align*}
%All of them can be controlled by $\|\mathfrak f\|_{C^{1,\alpha}}$.\\
%Our strategy is the prove estimates as those stated in (b) for $Z_L$ but uniformly in $L$. If we manage to do so, then the cut-offs in \eqref{eq:LSU1'} are not seen and we obtain a solution
%to \eqref{eq:thetapde}--\eqref{eq:thetaboundary}. Setting $\overline\vt:=\sqrt[4]{Z}$ and $\vt:=\overline\vt\circ \bfPsi_\zeta$ we obtain the required solution to \eqref{m119a}. Since these transformations are smooth (having strict positivity and boundedness of $Z$ at hand) $\vt$ inherits the regularity from $Z$.\\
%Consequently, all of the following effort is to show
%\begin{align}\label{eq:0312b}
%\min\Big\{C^{-1}\min_{\Omega}Z_0,1\Big\}\leq\min_{I\times\Omega_\zeta} Z_L\leq\max_{I\times\Omega_\zeta} Z_L\leq \max\Big\{C\max_{\Omega}Z_0,1\Big\},
%\end{align}
%with $C=C(\|\zeta\|_{C^{2}},\sup J_\zeta^{-1},\|\varrho\|_{C^1},\|\bfw\|_\infty,\|\nabla\bfw\|_\infty)$ but independent of $L$, where $Z_0:=\sqrt[4]{\vt_0\circ \bfPsi_\zeta^{-1}}$.\\
First we bound the maximum of $Z$ and respectively proof \eqref{eq:maxl}. Here we argue as in the proof of Theorem \ref{thm:regrho} and set
\begin{align*}
v(t,x):=\varphi(x)e^{-\lambda_1 t}Z
\end{align*}
where $\varphi\in C^\infty(\overline I\times \overline\Omega)$ is such that
\begin{align}
\label{eq:testLSU1'}\varphi(x)&\geq1\quad\text{in}\quad \overline I\times \overline\Omega,\\
\frac{\nabla\varphi\cdot\bfA^\ell_\zeta(Z)\nu}{\varphi}&<0 \quad\text{on}\quad\overline I\times\partial\Omega.\label{eq:testLSU2'}
\end{align}
Note that we have $\frac{\varkappa_\delta (\mathfrak b_{(t,x)}^{-1}(Z))}{b'_{(t,x)}(b_{(t,x)}^{-1}(Z))}\geq \frac{\underline\varkappa}{4}$
by \eqref{m108} such that the coercivity constant of $\bfA^\ell_\zeta(Z)$ can be bounded from below independently of $\ell$. Consequently, the function $\varphi$ can also be chosen independently of $\ell$.
% We have by \eqref{eq:LSU1} using linearity of b
%\begin{align*}
%\partial_t v&=\varphi e^{-\lambda_1 t}\partial\overline\varrho-\lambda v=
%-\varphi e^{-\lambda_1 t}b_\zeta(t,x,\overline\varrho,\nabla\overline\varrho)+\varphi e^{-\lambda_1 t}\bfA_\zeta:\nabla^2\overline\varrho-\lambda_1 v\\
%&=
%-b_\zeta(t,x,\varphi e^{-\lambda_1 t}\overline\varrho,\varphi e^{-\lambda_1 t}\nabla\overline\varrho)+\varphi e^{-\lambda_1 t}\bfA_\zeta:\nabla^2\overline\varrho-\lambda_1 v\\
%&=
%-b_\zeta(t,x,v,\nabla v+\overline\varrho e^{-\lambda_1 t}\nabla\varphi)+\varphi e^{-\lambda_1 t}\bfA_\zeta:\nabla^2\overline\varrho-\lambda_1 v.
%\end{align*}
In an interior maximum point $(t_0,x_0)\in I\times \Omega$ of $Z$
we have again
\begin{align}\nonumber
0
&=
-e^{-\lambda_1 t}\varphi b_\zeta^\ell\big(t,x,Z,\nabla Z\big)+\bfA^\ell_\zeta(Z_L):\nabla^2 v-\lambda_1 v\nonumber\\
&-\bfA^\ell_\zeta(Z):\big(2e^{-\lambda_1 t_0}\frac{Z}{\varphi}\nabla\varphi\otimes\nabla\varphi+Z e^{-\lambda_1 t_0}\nabla^2\varphi\big)\nonumber\\
&\leq -e^{-\lambda_1 t}\varphi b^\ell_\zeta\Big(t_0,x,Z,-\frac{Z}{\varphi} \nabla\varphi\Big)-\lambda_1 v+c(\varphi)e^{-\lambda_1 t_0}(1+Z)\nonumber\\
&\leq\,ce^{-\lambda_1 t_0}(1+Z)+\frac{\delta \varphi e^{-\lambda_1 t_0}}{(\mathfrak b_{(t,x)}^{-1}(\chi_\ell(Z)))^2}-\lambda_1e^{-\lambda_1 t_0}\varphi Z,\label{eq:0312}
\end{align}
%We estimate
%\begin{align*}
%g_\zeta^L(Z)&\leq c(1+Z)+\frac{\delta}{Z^{1/2}}\\
%-\Div_x\bff_\zeta(Z)&\leq c(1+Z)\\
%-\bfg^L_\zeta(Z)\cdot\frac{Z}{\varphi} \nabla\varphi&\leq c(1+Z)\\
%\nabla_u\bff^L_\zeta(Z)\cdot \frac{Z}{\varphi} \nabla\varphi&\leq\,c \\
%-\Div_x\bfA^L_\zeta(Z)\cdot\frac{Z}{\varphi} \nabla\varphi&\leq \\
%\nabla_u \bfA_\zeta(Z)\Big(\frac{Z}{\varphi} \nabla\varphi,\frac{Z}{\varphi} \nabla\varphi\Big)&\leq \\
%\end{align*}
%Here we have
where
\begin{align}\label{eq:0312c}
c=c(\varphi,\|\zeta\|_{C^{2,\alpha}_{t,x}},\|(\partial_t\nabla^2\zeta,\nabla^3\zeta)\|_{C^\alpha_{t,x}},\|\overline\varrho\|_{C^1_{t,x}},J_\zeta^{-1},\|\bfw\|_{C^{1,\alpha}_{t,x}},\|\nabla^2\bfw\|_{C^\alpha_{t,x}},\|H\|_{L^\infty_{t,x}})
\end{align}
is independent of $\ell$. Note that we used that the coefficients in the definition of $b^\ell_\zeta$
 have linear growth uniformly in $\ell$ except for $\delta (\mathfrak b_{(t,x)}^{-1}(\chi_\ell(u)))^{-2}$, $-\ep (\mathfrak b_{(t,x)}^{-1}(\chi_\ell(u)))^5$ and $\partial_u \bfA_\zeta^\ell(u)(\bfU,\bfU)$. Fortunately, the first two terms have the correct sign, whereas the second one is evaluated at $\bfU=-\frac{Z}{\varphi} \nabla\varphi$.
Now we distinguish two cases. If 
%$\max_{t,x}v(t,x)\leq 1$ 
$Z(t_0,x_0)\leq 1$ there is nothing to show. Otherwise, $\frac{\delta}{\mathfrak b_{(t,x)}^{-1}(\chi_\ell(Z(t_0,x_0))))^{2}}$ is bounded (independent of $\ell$) such that we obtain a contradiction in~\eqref{eq:0312} by choosing $\lambda_1$ large (depending on the quantities in \eqref{eq:0312c}). The case $x_0\in\partial\Omega$ and $t_0>0$ can be ruled out again as in \eqref{eq:1412}. Hence
\eqref{eq:maxl} follows with a constant independent of $\ell$.\\
%<  we may assume again that $Z(t_0,x_0)>1$ and find
%\begin{align*}
%0&\leq e^{-\lambda_1 t_0}\big(Z\nabla\varphi\cdot\bfA_\zeta^\ell(Z)(t_0,x_0)\nu(x_0)+\varphi\nabla Z\cdot\bfA^\ell_\zeta(Z)(t_0,x_0)\nu(x_0)\big)\\
%&=Z e^{-\lambda_1 t_0}\Big(\nabla\varphi\cdot\bfA^\ell_\zeta(t_0,x_0)\nu(x_0)\Big).
%\end{align*}
%If we choose $\lambda$ in \eqref{eq:testLSU2'} large enough we have a contradiction and \eqref{eq:maxl} follows with a constant independent of $\ell$.\\
In order to prove \eqref{eq:minl} we first establish a lower bound which depends on $\ell$.
%In case of a minimum in the interior we obtain similarly to \eqref{eq:0312}
%\begin{align}\label{eq:0313b}
%0&\geq\,-ce^{\lambda_1 t_0}(1+Z_L)+\frac{\delta e^{\lambda_1 t_0}}{\chi_L(Z_L)^{1/2}}-\ep\chi_L(Z_L)^{5/2}+\lambda_1e^{\lambda_1 t_0}\varphi Z_L.
%\end{align}
%We wish to argue that $Z_L(t_0,x_0)>0$. Assuming that $Z_L(t_0,x_0)=0$ and
%taking also into account boundedness of $Z_L$ we obtain $\frac{\delta}{(1/L)^{1/2}}\leq\,c$. This contradicts the fact that $c$ does not depend on $L$ and hence $Z_L(t_0,x_0)>0$. Consequently, an appropriate choice of $\lambda_1$ contradicts \eqref{eq:0313b}. Consequently, we have
%$Z_L>0$ in $I\times\Omega$. Unfortunately, the corresponding argument for boundary points fails as we cannot exclude that $Z_L$ becomes zero there. However, non-negativity is enough to apply Moser
Choosing first $\ell$ large enough and than $\underline Z\in(0,\inf Z_0)$ small enough (depending on $\ell$) we have
$g_\zeta^\ell(\underline Z)-\Div\big(\bff^\ell_\zeta(\underline Z)\big)\geq0$. This is thanks to the term $\delta \chi_\ell(Z)^{-1/2}$ in the definition of $g_\zeta^\ell$.
Consequently, we obtain from \eqref{eq:thetapde}
\begin{align*}
\partial_ta(Z-\underline Z)&\geq g_\zeta^\ell(Z)-g_\zeta^\ell(\underline Z)+\nabla Z\cdot\bfg_\zeta^\ell(Z)-\nabla \underline Z\cdot\bfg_\zeta^\ell(\underline Z)
-\Div\big(\bff^\ell_\zeta(Z)-\bff^\ell_\zeta(\underline Z)\big)\\&+\Div\big(\bfA_\zeta^\ell(Z)\nabla (Z-\underline Z)\big).
\end{align*} 
Multiplying by $(Z-\underline Z)^-$ and integrating over $\Omega$ implies
\begin{align*}
\frac{a}{2}\frac{\dd}{\dt}&\int_\Omega((Z-\underline Z)^-)^2\dx+\int_\Omega\bfA_\zeta^\ell(Z)\big(\nabla (Z-\underline Z)^-,\nabla (Z-\underline Z)^-\big)\dx\\
&\leq \int_\Omega\big(g_\zeta^\ell(\underline Z)-g_\zeta^\ell(Z)\big)(Z-\underline Z)^-\dx+\int_\Omega\big(\nabla \underline Z\cdot\bfg_\zeta^\ell(\underline Z)-\nabla Z\cdot\bfg_\zeta^\ell(Z)\big)(Z-\underline Z)^-\dx\\
&+\int_\Omega\big(\bff_\zeta^\ell(Z)-\bff_\zeta^\ell(\underline Z)\big)\nabla(Z-\underline Z)^-\dx
%&\leq \int_\Omega\big(g_\zeta^\ell(\underline Z)-g_\zeta^\ell(Z)\big)(Z-\underline Z)^-\dx+\int_\Omega\big(\nabla \underline Z\cdot\bfg_\zeta^\ell(\underline Z)-\nabla Z\cdot\bfg_\zeta^\ell(Z)\big)(Z-\underline Z)^-\dx\\
%&+\int_\Omega\big(\bff_\zeta^\ell(Z)-\bff_\zeta^\ell(\underline Z)\big)\nabla(Z-\underline Z)^-\dx
\end{align*} 
using also \eqref{eq:thetaboundary}. By the Lipschitz continuity of $g_\zeta^\ell$, $\bfg_\zeta^\ell$ and $\bff_\zeta^\ell$ (recall \eqref{eq:0411} and the assumptions on $\varrho$) in $Z$ and \eqref{eq:maxl}  we obtain
\begin{align*}
\frac{\dd}{\dt}&\int_\Omega \frac{a}{2}|(Z-\underline Z)^-|^2\dx+\int_\Omega \bfA_\zeta^\ell(Z)\big(\nabla Z^-,\nabla Z^{-}\big)\dx\\
&\leq \,\xi\int_{\Omega}|\nabla (Z-\underline Z)^-|^2\dx +c(\xi,\ell)\int_\Omega|(Z-\underline Z)^-|^2\dx
\end{align*}
for all $\xi>0$. Due to \eqref{m108} the first term can be absorbed for $\xi$ small enough, whereas the second one can be handled by Gronwall's lemma and $\vt_0>0$. We conclude that 
\begin{align}\label{eq:0101}
Z\geq\underline Z>0\quad\text{in}\quad \overline I\times\overline\Omega.
\end{align}
Recall that $\underline Z$ depends on $\ell$. We are now going to prove a uniform lower bound.
Similarly to \eqref{eq:testLSU1'} and \eqref{eq:testLSU2'} we consider a function $\varphi\in C^\infty(\overline I\times\overline\Omega)$ satisfying
\begin{align}
\label{eq:testLSU1''}\varphi(x)&\geq1\quad\text{in}\quad \overline I\times \overline\Omega,\\
\frac{\nabla\varphi\cdot\bfA^\ell_\zeta(Z)\nu}{\varphi}&\geq 1\quad\text{on}\quad\overline I\times\partial\Omega.\label{eq:testLSU2''}
\end{align}
Let us first assume that the minimum of $v=\varphi e^{\lambda_1 t} Z$ is attained in an interior point
$(t_0,x_0)\in I\times\Omega$.
 We obtain similarly to \eqref{eq:0312}
\begin{align}\label{eq:0313b}
0&\geq\,-ce^{\lambda_1 t_0}(1+Z)+\frac{\delta\varphi e^{\lambda_1 t_0}}{(\mathfrak b_{(t,x)}^{-1}(\chi_\ell(Z)))^{2}}-\ep(\mathfrak b_{(t,x)}^{-1}(\chi_\ell(Z)))^{5}+\lambda_1e^{\lambda_1 t_0}\varphi Z.
\end{align}
%We wish to argue that $Z_L(t_0,x_0)>0$. Assuming that $Z_L(t_0,x_0)=0$ and
%taking also into account boundedness of $Z_L$ we obtain $\frac{\delta}{(1/L)^{1/2}}\leq\,c$. This contradicts the fact that $c$ does not depend on $L$ and hence $Z_L(t_0,x_0)>0$. Consequently, 
An appropriate choice of $\lambda_1$ contradicts \eqref{eq:0313b}.
In the case of $x_0\in\partial\Omega$ and $t_0>0$ we have similarly to the proof of (b)
\begin{align*}
0&\geq Z e^{-\lambda_1 t_0}\nabla\varphi\cdot\bfA^\ell_\zeta(t_0,x_0)\nu(x_0).
\end{align*}
This gives a contradiction  by \eqref{eq:0101}, \eqref{eq:testLSU1''} and \eqref{eq:testLSU2''}. Consequently, the minimum of $Z$ is attained in a point
$(0,x_0)$ for some $x_0\in\overline\Omega$. This gives the claim of (b) since $\lambda_1$ is independent of $\ell$. Hence all properties of $Z$ are shown that imply (by transformation) the existence of a function $\vt$ with the required properties.
\end{proof}

\section{Construction of an approximate solution}
\label{sec:3}

In this section we construct an approximation of the system, where the continuity equation contains an artificial diffusion ($\ep$-layer) and the pressure is stabilised by a high power of the density ($\delta$-layer). Following \cite{FN} we add various regularizing terms depending on $\varepsilon$ and $\delta$ to the equations to preserve the energy balance. One of the regularizing terms can only be shown to belong to $L^1$, which is not enough to conclude uniform continuity in time needed for the application of Theorem~\ref{thm:auba}. To overcome this peculiarity we include a further diffusion term of the fluid velocity which is non-linear and of $p$-growth with $p>\beta>2$. It vanishes in the limit but improves the time integrability mentioned before.
 Additionally, we regularize the shell equation by
 replacing the operator $K$ with 
\begin{align*}
K_\ep(\eta)=K(\eta)+\ep\mathcal L(\eta),\quad \mathcal L(\eta)=\frac{1}{2}\int_\omega|\nabla^3\eta|^2\dH,
\end{align*}
defined for $\eta\in W^{3,2}(\omega)$. Thanks to this we can prove compactness of the shell energy in the Galerkin limit.

\begin{remark}
We observe that adding dissipative regularization terms to the shell equation is not possible. This is a special feature for energetically closed systems and in contrast to other fluid systems~\cite{GraHil}. Indeed, a dissipation term in the solid creates heat on the surface, which consequently effects the temperature. In the case of shells this yields a non-homogeneous Neumann boundary value for the temperature variable. This non-homogeneity naturally possesses the ''wrong sign'' in order to attain in the limit the boundary values for the temperature that are in accordance with the concept of weak solutions. In the case of visco-elastic solids, where dissipative terms  such as an additional heat source are included (they are physical and not only relaxation terms) our approximation would yield the correct non-homogeneous boundary values. However, we considered here perfectly elastic solids. Hence all energy is supposed to be stored in the elastic potential. 
\end{remark}

In contrast to \cite{BrSc} and \cite{LeRu} we construct the fixed point on the Galerkin level. This allows to remove one regularization level for the boundary and the convective term that was needed there. %$\kappa$-layer from  (which is reminiscent of , wherethe boundary and the convective terms are regularised, is not needed anymore.
%Our procedure to solve the regularised system differs from \cite{LeRu}
%and \cite{BrSc}: we approximate the latter one by a Galerkin-Ansatz which can be solved by means of a fixed point argument where the shell and the convective terms are decoupled (in \cite{LeRu}
%and \cite{BrSc} the fixed point is applied on the continuous level and the decoupled system is solved by a Galerkin approximation).
The formulation of the Galerkin approximation in our case is more involved since the basis functions are defined on the a priori unknown time dependent domain. The fixed point argument (which is now applied on the Galerkin level) is, however, much easier. After constructing a solution on the basic level, we prove in Subsection \ref{subsec:teb} the energy equality and derive further estimates through the Helmholz-function. In particular, we derive the approximate system and the a-priori estimates. They are essential for the remainder of the paper and are preserved in all limit procedures.\\
For the original system we seek a solution of the shell in the class
\begin{align*}
 Y^I:=W^{1,\infty}(I;L^2(\omega))\cap L^\infty(I;W^{2,2}(\omega)).
\end{align*}
However, in this section we are dealing with a regularised system
where instead solutions are located in

\begin{align*}
\tilde Y^I:= W^{1,\infty}(I;L^2(\omega))\cap L^{\infty}(I;W^{3,2}(\omega)).
\end{align*}
For $\zeta\in \tilde Y^I$ with $\|\zeta\|_{L^\infty_{t,x}}\leq \frac{L}{2}$ we consider

\begin{align*}
\tilde X_\zeta^I&:=L^p(I;W^{1,p}(\Omega_{\zeta   (t)})),
%\quad X_\zeta^I:=L^2(I;W^{1,2}(\Omega_{\zeta   (t)})),
\quad Z_\zeta^I:=L^2(I;W^{1,2}(\Omega_\zeta))\cap L^\infty(I;L^4(\Omega_\zeta)).
\end{align*}
The space $X^I_\eta$ is defined in Section \eqref{sec:weak}.
A solution to the regularized system, in the weak formulation, is a quadruplet $(\eta,\bfu,\varrho,\vt)\in  \tilde Y^I\times \tilde X_{\eta}^I\times X_{\eta}^I\times Z_{\eta}^I$ that satisfies the following.
\begin{enumerate}[label={(K\arabic{*})}]
 \item\label{K1} The regularized weak momentum equation
\begin{align}\label{eq:regu}
\begin{aligned}
&\int_I\frac{\dd}{\dt}\int_{\Omega_{\eta}}\varrho\bfu \cdot\bfphi\dxt-\int_I\int_{\Omega_{\eta}}\varrho\bfu \cdot\partial_t\bfphi\dx\dt\\
&-\int_I\int_{\Omega_{\eta}}\varrho\bfu\otimes \bfu:\nabla \bfphi\dx\dt+\int_I\int_{\Omega_{ \eta }}\bfS(\vt,\nabla\bfu):\nabla\bfphi \dxt
\\&-\int_I\int_{\Omega_{\eta }}
p_\delta(\vr,\vt)\,\Div\bfphi\dxt+\int_I\int_{\Omega_{\eta }}
\varepsilon\nabla\varrho\nabla\bfu\cdot\bfphi\dxt\\
&+\int_I\bigg(\frac{\dd}{\dt}\int_\omega \partial_t \eta\, b\dH-\int_\omega \partial_t\eta\,\partial_t b\dH + \int_\omega K_\ep'(\eta)\,b\dH\bigg)\dt
\\
&+\int_I\int_{\Omega_{\eta }}\ep(1+\vt)\overline{\bfP}:\nabla\bfphi\dxt\\
&=\int_I\int_{\Omega_{\eta}}\varrho\bff\cdot\bfphi\dxt+\int_I\int_\omega g\,b\,\dd x\dt
%+\int_{\Omega_{\eta_0}}\bfq_0\cdot\bfphi(0,\cdot)\dx+\int_\omega\eta_0^\kappa\,b\,\dd\mathcal H^2
\end{aligned}
\end{align} 
holds for all test-functions $(b,\bfphi)\in C^\infty(\omega)\times C^\infty(\overline{I}\times \R^3)$ with $\mathrm{tr}_{\eta}\bfphi=b\nu$
and for some $\overline{\bfP}\in L^{p'}(I\times\Omega_{\eta})$. %\seb{This is the weak limit of the additional diffusion term term $(1+\vt_N)(1+\abs{\nabla \bfu_N})^{p-2}\nabla \bfu_N$.}
Moreover, we have $(\varrho\bfu)(0)=\bfq_0$, $\eta(0)=\eta_0$ and $\partial_t\eta(0)=\eta_1$. The boundary condition $\mathrm{tr}_{\eta}\bfu=\partial_t\eta\nu$ holds in the sense of Lemma \ref{lem:2.28}.
\item\label{K2} The regularized continuity equation 
\begin{align}\label{eq:regvarrho}
\partial_t\vr+\Div\big(\varrho\bfu\big)=\ep\Delta\vr
\end{align}
holds in $I\times\Omega_{\eta}$ with $\partial_{\nu_{\eta}}\vr|_{\partial\Omega_{\eta}}=0$ as well as $\vr(0)=\vr_0$.
%\item\label{K3} The regularized internal energy equation
%\begin{align}\label{m119}
%\begin{aligned}
%\partial_t&\big(\varrho e_\delta(\varrho, \vartheta) \big)
%+\Div\big(\varrho e_\delta(\varrho, \vartheta)\mathscr R_\kappa\bfu\big)-
%\Div\big(\varkappa_\delta (\vt) \nabla \vt\big)  \\
%& = \bfS(\vartheta,\nabla\bfu): \nabla \bfu -
%p_\delta(\varrho, \vartheta) \Div\mathscr R_\kappa\bfu  +  \ep
%\delta \beta \varrho^{\beta - 2} |\nabla \varrho |^2  +
%\delta \frac{1}{{\vartheta}^2} - \ep \vartheta^5
%\end{aligned}
%\end{align}
%in $I\times\Omega_{\mathscr R_\kappa\eta}$ and we have $\partial_{\nu_{\mathscr R_\kappa\eta}}\vt|_{\partial\Omega_{\mathscr R_\kappa\eta}}=0$ as well as $\vt(0)=\vt_0$.
\item \label{K3} The entropy balance
\begin{align} \label{m217*finala}\begin{aligned}
 \int_I &\frac{\dd}{\dt}\int_{\Omega_{\eta}} \vr s(\varrho,\vartheta) \, \psi\dxt
 - \int_I \int_{\Omega_{\eta}} \big( \vr s(\varrho,\vartheta) \partial_t \psi + \varrho s (\varrho,\vartheta)\bfu \cdot \nabla\psi \big)\dxt
\\& \geq\int_I\int_{\Omega_{\eta}}
\frac{1}{\vartheta}\Big[\bfS(\vt, \nabla \vu) : \nabla \bfu +\ep(1+\vt)\max\big\{|\overline{\bfP}|^{p'},|\nabla\bfu|^p\big\}\Big]
\psi\,\dif x\,\dif t 
\\&  +\int_I\int_{\Omega_{\eta}}
\frac{1}{\vartheta}\Big[\frac{\delta}{2} ( \vartheta^{\beta -
1} + \frac{1}{\vartheta^2}) \Big) |\nabla \vartheta|^2 + \delta
\frac{1}{{\vartheta}^2} \Big] 
\psi\,\dif x\,\dif t \\
&- \int_I\int_{\Omega_{\eta}}
\Big( \frac{\varkappa(\vartheta)}{\vartheta} + \delta (
\vartheta^{\beta - 1} + \frac{1}{\vartheta^2}) \Big) \nabla
\vartheta  \cdot \nabla\psi \dxt+ \int_I \int_{\Omega_{\eta}} \frac{\vr}{\vt} H \psi  \dxt\\
&+ \int_I\int_{\Omega_{\eta}}\ep\left[\rdiss(\vr)-  \vartheta^4\right] \psi\dxt
\end{aligned}
\end{align}
holds for all
$\psi\in C^\infty(\overline I\times \R^3)$ with $\psi \geq 0$; in particular, all integrals on the right hand side are well defined. Moreover, we have
$\lim_{r\rightarrow0}\vr s(\vr,\vt)(t)\geq \vr_0 s(\vr_0,\vt_0)$ and $\partial_{\nu_{\eta}}\vt|_{\partial\Omega_\eta}\leq 0$.
\item \label{K4} The total energy balance
\begin{equation} \label{EI20finala}
\begin{split}
- \int_I \partial_t \psi \,
\mathcal E_{\ep,\delta} \dt &=
\psi(0) \mathcal E_{\ep,\delta}(0)+\int_I\psi \int_{\Omega_{\eta}}   \left(\frac{\delta}{\vartheta^2}-\ep \vt^5  \right) \dx \dt + \int_I \psi \int_{\Omega_{\eta}} \vr H  \dxt\\&+\int_I\int_{\Omega_{\eta}}\varrho\bff\cdot\bfu\dxt+\int_I\psi\int_M g\,\partial_t\eta\,\dd y\dt
\end{split}
\end{equation}
holds for any $\psi \in C^\infty_c([0, T))$.
Here, we abbreviated
$$\mathcal E_{\ep,\delta}(t)= \int_{\Omega_{\eta}(t)}\Big(\frac{1}{2} \varrho(t) | {\bf u}(t) |^2 + \varrho(t) e_\delta(\varrho(t),\vartheta(t))\Big)\dx+\int_M \frac{|\partial_t\eta(t)|^2}{2}\,\dd y+ K_\ep(\eta(t)).$$
\end{enumerate}
\begin{remark}\label{rem:2201A}
In order to deal with the term $\int_I\int_{\Omega_{\eta }}
\varepsilon\nabla\varrho\nabla\bfu\cdot\bfphi\dxt$ (appearing in \eqref{eq:regu} to balance the artificial viscosity term in \eqref{eq:regvarrho}) in the proof of \eqref{eq:conetatC}
we need higher integrability of $\nabla\bfu$ in time. This is achieved by introducing an artificial $p$-Laplacian term $\ep(1+\vt)(1+|\nabla\bfu|)^{p-2}\nabla\bfu$ for some $p>\beta>2$ on the Galerkin approximation in the next section. It gives the additional term $\ep(1+\vt)\max\big\{|\overline{\bfP}|^{p'},|\nabla\bfu|^p\big\}$ in \eqref{m217*finala}.
The term $\ep(1+\vt)\overline{\bfP}$ in \eqref{eq:regu} is the weak limit of the $p$-Laplacian term and can be seen as the defect in the strong convergence of $\nabla\bfu$. It disappears in the limit $\ep\rightarrow0$.
\end{remark}

The rest of this section is dedicated to the proof of the following existence theorem.
\begin{theorem}\label{thm:regu}
Assume that we have for some $\alpha\in(0,1)$
\begin{align}\label{inintial}
\begin{aligned}
\frac{|\bfq_0|^2}{\varrho_0}&\in L^1(\Omega_{\eta_0}),\ \varrho_0 ,\vt_0\in C^{2,\alpha}(\overline\Omega_{\eta_0}),\   \eta_0\in W^{3,2}(\omega;[-\tfrac{L}{4},\tfrac{L}{4}]),\ \eta_1\in L^{2}(\omega),\\
\bff&\in L^2(I;L^\infty(\R^3)),\ g\in L^2(I\times \omega),\ H\in C^{1,\alpha}(\overline I\times\R^3), \ H\geq0.
\end{aligned}
\end{align}
Furthermore suppose that $\varrho_0$ and $\vt_0$ are strictly positive and that \eqref{eq:compa} is satisfied. Then there exists a solution $(\eta,\bfu,\varrho,\vt)\in \tilde Y^I\times \tilde X_{\eta}^I\times X_{\eta}^{I}\times Z_{\eta}^{I}$ to \ref{K1}--\ref{K4}. Here, we have $I=(0,T_*)$, where
$T_*<T$ only if $\lim_{t\rightarrow T^\ast}\|\eta(t,\cdot)\|_{L^\infty_x}=\frac{L}{2}$ or the Koiter energy degenerates (namely, if $\lim_{s\to t}\overline{\gamma}(s,y)=0$ for some point $y\in \omega$).
% The solution satisfies the total energy balance
%\begin{equation} \label{eq:energy}
%\begin{split}
%- \int_I \partial_t \psi \,
%\mathcal E_{\ep,\delta} \dt &=
%\psi(0) \mathcal E_{\ep,\delta}(0)+\int_I\psi \int_{\Omega_{\eta}}   \left(\frac{\delta}{\vartheta^2}-\ep \vt^5  \right) \dx \dt \\&+ \int_I \psi \int_{\Omega_{\eta}} \vr H  \dxt+\int_I\int_{\Omega_{\eta}}\varrho\bff\cdot\bfu\dxt+\int_I\psi\int_\omega g\,\partial_t\eta\,\dd\mathcal H^2\dt
%\end{split}
%\end{equation}
%for any $\psi \in C^\infty_0[0, T)$.
%Here, we abbreviated
%$$\mathcal E_{\ep,\delta}(t)= \int_{\Omega_\eta(t)}\Big(\frac{1}{2} \varrho(t) | {\bf u}(t) |^2 + \varrho(t) e_\delta(\varrho(t),\vartheta(t))\Big)\dx+\int_M \frac{|\partial_t\eta(t)|^2}{2}\,\dd\mathcal H^2+ K_\ep(\eta(t)).$$
%$\Omega_{\regkap{\eta(s)}}$ is approaching a selfintersection with $s\to T_*$.
\end{theorem}
We prove Theorem \ref{thm:regu} in two steps. First we construct a finite dimensional Galerkin approximation to \ref{K1}--\ref{K3} in the next subsection.
Then we derive the energy balance, prove uniform a priori estimates and pass to the limit.
%\begin{remark}
%\label{rem:new}
%\seb{The restriction $\|\eta\|_\infty< \frac{L}{2}$ is needed for the construction
%of our extension operator, see Lemma \ref{lem:extension}. The latter one is used for the renormalized continuity equation, see Theorem \ref{lem:warme} b) and, in particular, Section \ref{sec:renep}. This is why we keep the assumption $\|\eta\|_\infty< \frac{L}{2}$ during the whole construction and only relax it at the very end
%in Section \ref{sec:maxexist}.}
%\end{remark}
 
\subsection{Galerkin approximation} 
By solving respective eigenvalue problems we construct a smooth orthogonal basis $(\tilde\bfX_k)_{k\in\N}$
of $W^{1,2}_0(\Omega)$ that is orthogonal in $L^2(\Omega)$ and a smooth orthonormal basis $(\tilde Y_k)_{k\in\N}$ of $W^{3,2}(\omega)$ which is orthogonal in $L^2(\omega)$. We define vector fields $\tilde\bfY_k$ by setting $\tilde\bfY_k=\mathscr F_\Omega((\tilde Y_k\nu)\circ\bfvarphi^{-1})$, where $\mathscr F_\Omega$ is the extension operator used in Section \ref{sec:ext}. We recall that $\mathscr F_\Omega:W^{k,2}(\omega)\rightarrow W^{k,2}(\R^n)$ for $k\in\N$ such that the $\tilde\bfY_k$'s are smooth. 
%Now define pointwise in $t$
%\begin{align*}
%\bfX_k:=\tilde\bfX_k\circ \bfPsi_{\mathscr R_\kappa\zeta}^{-1},\quad \bfY_k:=\tilde\bfY_k\circ\bfPsi_{\mathscr R_\kappa\zeta}^{-1}.
%\end{align*}
%By Lemma \ref{lem:diffeo} we still know that $\bfX_k$ and $\bfY_k$ belong to the class $C^3(\overline{\Omega}_{\mathscr R_\kappa\zeta}(t))$.
%Obviously, $(\bfX_k)_{k\in\N}$ forms a basis of $W^{1,2}_0(\Omega_{\regkap \zeta}(t)).$
Now we choose an enumeration $(\tilde\bfomega_k)_{k\in\N}$ of $(\tilde\bfX_k)_{k\in\N}\cup (\tilde\bfY_k)_{k\in\N}$. In return we associate
$w_k:=(\tilde\bfomega_k|_{\partial\Omega} \nu)\circ\bfvarphi$. 
Obviously, we obtain a basis $(\tilde\bfomega_k)_{k\in\N}$ of $W^{1,2}_0(\Omega)$
and a basis $(w_k)_{k\in\N}$ of $W^{3,2}(\omega)$. We
define for $\phi\in W^{3,2}(\omega)$ the orthogonal projection (in space) $\mathcal P_N$ as
\[
\mathcal P_N(\phi):=\sum_{k=1}^N P_N^k(\phi)w_k:=\sum_{k=1}^N\skp{\phi}{w_k}_{{W^{3,2}}(\omega)}w_k,
\]
 which satisfies the expected stability and convergence properties in all spaces relevant for the analysis. 
%Analogous to the arguments
%in \cite[p. 237]{LeRu} we find that
%\begin{align*}
%Z:=\mathrm{span}\Big\{(\varphi w_k,\varphi\bfomega_k)\,|\varphi\in C^1(I),\,k\in\N\Big\}
%\end{align*}
%is dense in the solution space
%\begin{align*}
%Z_{\regkap\zeta}:=\Big\{(\xi,\bfvarphi)\in Y^I\times X^I_{\regkap\zeta}:\,\,\partial_t \xi\nu_{\regkap\zeta}=\mathrm{tr}_{{\regkap\zeta}}\bfvarphi\Big\},
%\end{align*}
%and in the space of test-functions
%\begin{align*}
%Z_{\regkap\zeta}^*:=\Big\{(\xi,\bfvarphi)\in C(\overline I;W^{2,2}_0(M))\times L^2(I,W^{1,2}(\Omega_{\regkap\zeta}))\cap C(\overline I;L^2(\Omega_{\regkap\zeta})):\,\,\xi\nu=\mathrm{tr}_{{\regkap \zeta}}\bfvarphi\Big\}.
%\end{align*}
Next we seek for a couple of discrete solutions $(\eta_N,\bfu_N)$ of the form
\begin{align*}
\eta_N=\mathcal P_N\eta_0+\sum\nolimits_{k=1}^N\int_0^t\alpha_{kN} w_k\ds
,\quad \bfu_N=\sum\nolimits_{k=1}^N\alpha_{kN} \tilde\bfomega_k\circ \bfPsi_{\eta_N}^{-1},
\end{align*}
with time-dependent coefficients $\bfalpha_N=(\alpha_{kN})_{k=1}^N$, which solve the following discrete version of \eqref{eq:regu}: 
\begin{align}\label{eq:decuN'}
\begin{aligned}
&\int_{\Omega_{\eta_N}}\varrho_N(t) \bfu_N(t)\cdot \tilde\bfomega_k\circ \bfPsi_{\eta_N}^{-1}(t)\dx\\%-\frac{\kappa}{2}\int_0^t\int_{\partial\Omega_{\eta_N}}\partial_t\eta_N\bfu_N\cdot\tilde\bfomega_k\circ \bfPsi_{\eta_N}^{-1}\dxs
%+\int_0^t\int_{\Omega_{\eta_N }}\frac{1}{N}\bfu_N\cdot\tilde\bfomega_k\circ %\bfPsi_{\zeta_N}^{-1}\dx\ds\\
&-\int_0^t\int_{\Omega_{\eta_N}}\Big( \varrho_N\bfu_N\cdot \partial_t\Big(\tilde\bfomega_k\circ \bfPsi_{\eta_N}^{-1}\Big) +\varrho_N\bfu_N\otimes \bfu_N:\nabla \tilde\bfomega_k\circ \bfPsi_{\eta_N}^{-1}\Big)\dxt
\\
&+\int_0^t\int_{\Omega_{\eta_N }}\Big(\bfS^\ep(\vt_N,\nabla\bfu_N):\nabla\tilde\bfomega_k\circ \bfPsi_{\eta_N}^{-1}\Big)\dx\ds
\\
&-\int_0^t\int_{\Omega_{\eta_N }}\Big(
p_\delta(\varrho_N,\vt_N)\,\Div\tilde\bfomega_k\circ \bfPsi_{\eta_N}^{-1}+\varepsilon\nabla\varrho_N\nabla\bfu_N\tilde\bfomega_k\circ \bfPsi_{\eta_N}^{-1}\Big)\dx\ds
\\
&+\int_0^t\int_\omega \Big(K_\ep'(\eta_N)w_k-\partial_t\eta_N\,\partial_t w_k\Big)\dH\ds+\int_\omega \partial_t \eta_N(t) w_k\dH\ds
\\
&=\int_0^t\int_{\Omega_{\eta_N}}\varrho_N\bff\cdot\tilde\bfomega_k\circ \bfPsi_{\eta_N}^{-1}\dxs+\int_0^t\int_\omega g\,w_k\,\dd y\ds\\
&+\int_{\Omega_{\eta_N(0)}}\bfq_0\cdot\tilde\bfomega_k\circ \bfPsi_{\eta_N}^{-1}(0,\cdot)\dx+\int_\omega\eta_1\,w_k\,\dd y.
\end{aligned}
\end{align} 
Here $\varrho_N=\varrho(\eta_N,\bfu_N)$ and $\vt_N=\vt(\eta_N,\bfu_N,\vr_N)$ are the unique solutions from Theorems~\ref{thm:regrho} and \ref{thm:regtheta} subject to the initial data $\varrho_0$ and $\vt_0$, where $\zeta\equiv \eta_N$ and $\bfw\equiv \bfu_N$.
Note that by construction we have $\tr_{\eta_N}\bfu_N=\partial_t\eta_N\nu$
and that we can choose $\alpha_{kN}(0)$ in a way that
$\bfu_N(0)$ converges to $\bfq_0/\varrho_0$.
In order to solve \eqref{eq:decuN'} we decouple the nonlinearities. Consider a given couple of discrete functions $(\zeta_N,\bfv_N)$ of the form
\begin{align}
\label{eq:zeta}
\zeta_N=\mathcal P_N\eta_0+\sum\nolimits_{k=1}^N\int_0^t\beta_{kN} w_k\ds
,\quad \bfv_N=\sum\nolimits_{k=1}^N\beta_{kN} \tilde\bfomega_k\circ \bfPsi_{\zeta_N}^{-1},
\end{align}
with time-dependent coefficients $\bfbeta_N=(\beta_{kN})_{k=1}^N$.
By construction they satisfy $\tr_{\zeta_N}\bfv_N=\partial_t\zeta_N\nu$.
We aim to solve
\begin{align}\label{eq:decuN}
\begin{aligned}
&\int_{\Omega_{\zeta_N}}\varrho_N(t)  \bfu_N(t)\cdot \tilde\bfomega_k\circ \bfPsi_{\zeta_N}^{-1}(t)\dx\\%-\frac{\kappa}{2}\int_0^t\int_{\partial\Omega_{\zeta_N}}\partial_t\zeta_N\bfu_N\cdot\tilde\bfomega_k\circ \bfPsi_{\zeta_N}^{-1}\dxs
%+\int_0^t\int_{\Omega_{\zeta_N }}\frac{1}{N}\bfu_N\cdot\tilde\bfomega_k\circ %\bfPsi_{\zeta_N}^{-1}\dx\ds\\
&-\int_0^t\int_{\Omega_{\zeta_N}}\Big( \varrho_N\bfu_N\cdot \partial_t\Big(\tilde\bfomega_k\circ \bfPsi_{\zeta_N}^{-1}\Big) +\varrho_N\bfv_N\otimes \bfu_N:\nabla \tilde\bfomega_k\circ \bfPsi_{\zeta_N}^{-1}\Big)\dxt
\\
&+\int_0^t\int_{\Omega_{\zeta_N }}\Big(\bfS^\ep(\vt_N,\nabla\bfu_N):\nabla\tilde\bfomega_k\circ \bfPsi_{\zeta_N}^{-1}\Big)\dx\ds
\\
&-\int_0^t\int_{\Omega_{\zeta_N }}\Big(
p_\delta(\varrho_N,\vt_N)\,\Div\tilde\bfomega_k\circ \bfPsi_{\zeta_N}^{-1}+\varepsilon\nabla\varrho_N\nabla\bfu_N\tilde\bfomega_k\circ \bfPsi_{\zeta_N}^{-1}\Big)\dx\ds
\\
&+\int_0^t\int_\omega \Big(K_\ep'(\eta_N)w_k-\partial_t\eta_N\,\partial_t w_k\Big)\dH\ds+\int_\omega \partial_t \eta_N(t) w_k\dH
\\
&=\int_0^t\int_{\Omega_{\zeta_N}}\varrho_N\bff\cdot\tilde\bfomega_k\circ \bfPsi_{\zeta_N}^{-1}\dxs+\int_0^t\int_\omega g\,w_k\,\dd y\ds\\
&+\int_{\Omega_{\zeta_N(0)}}\bfq_0\cdot\tilde\bfomega_k\circ \bfPsi_{\zeta_N}^{-1}(0,\cdot)\dx+\int_\omega\eta_1\,w_k\,\dd y.
\end{aligned}
\end{align} 
Here  $\varrho_N=\varrho(\zeta_N,\bfv_N)$ and $\vt_N=\vt(\zeta_N,\bfv_N,\vr_N)$ are the unique solutions from Theorems~\ref{thm:regrho} and \ref{thm:regtheta} subject to the initial data $\varrho_0$ and $\vt_0$, where $\zeta\equiv \zeta_N$ and $\bfw\equiv \bfv_N$.
Note that this is possible since $\norm{P_N\eta_0}_{L^\infty_x}\leq \frac{L}{3}$ for $N$ large enough, which implies $\norm{\zeta_N}_{L^\infty_{t,x}}\leq \frac{L}{2}$ for $T_*$ small enough.
 The system \eqref{eq:decuN} is equivalent to
a system of integro-differential equations for the vector $\bfalpha_N=(\alpha_{kN})_{k=1}^N$. 
 It reads as
\begin{align}\label{eq:ide1}
\mathcal A(t)\bfalpha_N(t)&=\int_0^t\mathcal B(\sigma)\bfalpha_N(\sigma)\ds+\int_0^t \tilde{\mathcal B}\bigg(\sigma,\bfalpha_N(\sigma),\int_0^\sigma\bfalpha_N(s)\dd s\bigg)\ds+\int_0^t\bfc(\sigma)\ds+\tilde{\bfc},
\end{align}
with
\begin{align*}%\label{eq:ide}
%\begin{aligned}
\mathcal A_{ij}&=\int_{\Omega_{\zeta_N}}\varrho_N(t) \tilde\bfomega_i\circ \bfPsi_{\zeta_N}^{-1}(t) \cdot\tilde\bfomega_j\circ \bfPsi_{\zeta_N}^{-1}(t)\dx
+\int_\omega w_i w_j\dH
\\
\mathcal{B}_{ij}
&=\int_{\Omega_{\zeta_N}}\Big( \varrho_N\tilde\bfomega_i\circ \bfPsi_{\zeta_N}^{-1}\cdot \partial_t\Big(\tilde\bfomega_j\circ \bfPsi_{\zeta_N}^{-1}\Big) +\varrho_N\bfv_N\otimes \tilde\bfomega_i\circ \bfPsi_{\zeta_N}^{-1}:\nabla\tilde\bfomega_j\circ \bfPsi_{\zeta_N}^{-1}\Big)\dx%-\frac{\kappa}{2}\int_{\partial\Omega_{\zeta}}\partial_t\zeta\bfu\cdot\bfomega_k\dx
\\
&-\int_{\Omega_{\zeta_N }}
\varepsilon\nabla\varrho_N\nabla\tilde\bfomega_i\circ \bfPsi_{\zeta_N}^{-1}\cdot\tilde\bfomega_j\circ \bfvarphi_{\zeta_N}^{-1}\dx\ds
-\int_{\omega}w_i\,\partial_t w_j\dH
\\
\tilde{\mathcal B}_{j}&=\int_\omega K_\ep'\bigg(\mathcal P_N\eta_0+\sum\nolimits_{k=1}^N\int_0^\sigma\alpha_{kN}(s) w_k\dd s\bigg)w_j\dH\\
&+\int_{\Omega_{\zeta_N }}\bfS^\ep\Big(\vt_N,\sum_{k=1}^N\nabla\big(\alpha_{kN}\tilde\bfomega_k\circ \bfPsi_{\zeta_N}^{-1}\big)\Big):\nabla\tilde\bfomega_j\circ \bfPsi_{\zeta_N}^{-1}\dx
\\
c_i&=\int_{\Omega_{\zeta_N}}p_\delta(\vr_N,\vt_N)\,\Div\tilde\bfomega_i\circ \bfPsi_{\zeta_N}^{-1}\dx+\int_{\Omega_{\zeta_N}}\varrho_N\bff\cdot\tilde\bfomega_i\circ \bfPsi_{\zeta_N}^{-1}\dxt+\int_\omega g\,w_i\,\dd y
\\
\tilde{c}_i&=\int_{\Omega_{\zeta_N(0)}}\bfq_0\cdot\tilde\bfomega_i\circ \bfPsi_{\zeta_N}^{-1}(0,\cdot)\dx+\int_\omega\eta_1\,w_i\,\dd y.
%\end{aligned}
\end{align*}
The matrix $\mathcal A_{ij}$ is invertible and all non-linear quantities are locally Lipschitz continuous in $\bfalpha_N$ (compare also with \cite[Thm. 4.4]{BrSc}). Also our analysis from Section \ref{sec:3a} shows that $\vt_N$ and $\varrho_N$ depend in a smooth way on $\bfv_N$ and $\zeta_N$. By the Picard-Lindel\"of theorem there is a unique solution in short time.\footnote{
Eventually, the solution can be extended for arbitrary times due to the a priori estimates which we derive below in \eqref{wWS27}.}
%For details we refer to \cite[Thm. 4.4]{BrSc}.
Consequently, we obtain a solution $(\eta_N,\bfu_N)$ to \eqref{eq:decuN} which satisfies the following energy balance (testing \eqref{eq:decuN} by $(\bfu_N,\partial_t\eta_N)$ and \eqref{eq:warme} by $\frac{1}{2}|\bfu_N|^2$)
\begin{align*}
\begin{aligned}
-&\int_I\bigg(\int_{\Omega_{\zeta_N}}\varrho_N\frac{|\bfu_N|^2}{2}\dx+\int_\omega\frac{|\partial_t\eta_N|^2}2\,\dd y+K_\ep(\eta_N)\bigg)\partial_t\psi\dt
\\ 
& +\int_I\psi\int_{\Omega_{\zeta}}\bfS^\ep(\vt_N,\nabla\bfu_N):\nabla\bfu_N\dxt\\
&=\psi(0)\bigg(\int_{\Omega_{\zeta_N(0)}}\frac{|\bfq_0|^2}{2\varrho_0}\dx%+\int_\omega\frac{|\eta_0|^2}2\,\dd y
+\int_\omega\frac{|\eta_1|^2}2\,\dd y+K_\ep(\eta_0)\bigg)\\
&+\int_I\psi\int_{\Omega_{\zeta_N}}\varrho_N\bff\cdot\bfu_N\dxt+\int_I\psi\int_\omega g\partial_t \eta_N\,\dd y\dt\\
& +\int_I\psi\int_{\Omega_{\zeta_N}}p_\delta(\vr_N,\vt_N)\,\Div\bfu_N\dxt
\end{aligned}
\end{align*}
for all $\psi\in C^\infty_c([0,T))$. Testing further the continuity equation by
$\delta\frac{\beta\varrho_N^{\beta-1}}{(\beta-1)}+\frac{\gamma\varrho_N^{\gamma-1}}{(\gamma-1)}$ yields %\eqref{eq:press} \seb{(using the notation defined in from \eqref{neu} and \eqref{eq:rdiss})}
\begin{align}\label{eq:1712}
\begin{aligned}
%&\psi(t)\bigg(\int_{\Omega_{\zeta_N}}\varrho_N(t)\frac{|\bfu_N(t)|^2}{2}\dx+\int_\omega\frac{|\partial_t\eta_N(t)|^2}2\,\dd\mathcal H^2+\frac{{K(\eta_N(t))}}{2}\db{\kappa\mathcal L(\eta_N)}\bigg)\\
-&\int_I\bigg(\int_{\Omega_{\zeta_N}}\varrho_N\frac{|\bfu_N|^2}{2}\dx
+\int_{\Omega_{\zeta_N}}\Big(\tfrac{1}{\gamma-1}\vr_N^\gamma+\tfrac{\delta}{\beta-1}\vr_N^\beta\Big)\dx+\int_\omega\frac{|\partial_t\eta_N|^2}2\,\dd y+K_\ep(\eta_N)\bigg)\partial_t\psi\dt
\\ 
& +\int_I\psi\int_{\Omega_{\zeta}}\bfS^\ep(\vt_N,\nabla\bfu_N):\nabla\bfu_N\dxt
+\int_I\psi\int_{\Omega_{\zeta}}\ep
%\delta \beta \varrho_N^{\beta - 2} |\nabla \varrho_N |^2
\rdiss(\vr_N)\dxt
\\
&=\psi(0)\bigg(\int_{\Omega_{\zeta_N(0)}}\frac{|\bfq_0|^2}{2\varrho_0}\dx
%+\int_\omega\frac{|\eta_0|^2}2\,\dd y
+\int_\omega\frac{|\eta_1|^2}2\,\dd y+K_\ep(\eta_0)+\int_{\Omega_{\zeta_N}(0)}\Big(\tfrac{1}{\gamma-1}\vr_0^\gamma+\tfrac{\delta}{\beta-1}\vr_0^\beta\Big)\dx\bigg)\\
&+\int_I\psi\int_{\Omega_{\zeta_N}}\varrho_N\bff\cdot\bfu_N\dxt+\int_I\psi\int_\omega g\partial_t \eta_N\,\dd y\dt\\
& +\int_I\psi\int_{\Omega_{\zeta_N}}\Big(\big(\vr_N\vt_N+\tfrac{a}{3}\vt_N^4\big)\,\Div\bfu_N+\big(\vr_N^\gamma+\delta\vr_N^\beta\big)\,\Div(\bfu_N-\bfv_N)\Big)\dxt
\end{aligned}
\end{align}
for all $\psi\in C^\infty_c([0,T))$.
We consider the mapping %\db{why semin-norm?}
\begin{align*}
F&:D\rightarrow F(D),\quad \bfbeta\mapsto\bfalpha,\quad
D=\Big\{\bfbeta\in C^{1,\alpha}(\overline I_\ast,\R^N):\,\sup_{I_\ast}\norm{\bfbeta'}_\alpha\leq K^*\Big\}%,\, \sup_{I_\ast}\abs{\bfalpha}\leq \frac{L}{2}} \Big\}.
\end{align*}
where $I_*=(0,T_*)$ and $\alpha\in(0,1)$.
We will choose $K^*$ sufficiently large. In dependence of $K^*$ we find $T_*$ (sufficiently small) but uniform to solve the above ODE uniquely on $I_*$. Note that we may take $T^*$  small enough (in dependence of $K^*$) such that $\zeta_N$ (defined via $\bfbeta$ by~\eqref{eq:zeta}) satisfies $\norm{\zeta_N}_{L^\infty_{t,x}}\leq \frac{L}{2}$  for any $\bfbeta\in D$.
%\sebcom{Show the continuity of $F$ (w.r.t the geometry). This goes by transferring the coefficients $\mathcal A_{ij}$ ect. to the reference configuration. The continuity of $\rho_N$ and $\theta_N$ should follow by the  equations in the reference geometry above as well.} \db{Upper semicontinuous reicht (wie bei set-valued mappings)}
We are going to prove that $F$ has a fixed point. Let us first note that $F$ is upper-semicontinuous. Indeed, if we have a sequence $(\bfbeta^j)$ which converges in 
$C^{1,\alpha}(\overline I)$ to some $\bfbeta$ such that $\bfalpha^j=F(\bfbeta^j)$ converges in $C^{1,\alpha}(\overline I)$ to some $\bfalpha$, we have $\bfalpha=F(\bfbeta)$. This is due to the unique solvability of \eqref{eq:decuN} and the continuity of the coefficients $\mathcal A$, $\mathcal B$, $\tilde{\mathcal{B}}$ and $\bfc$. In fact, the continuity of $\mathcal A$, $\mathcal B$, $\tilde{\mathcal{B}}$ and $\bfc$ (with respect to $\bfbeta$) can be shown by transforming the integrals to the reference domain and using \eqref{map} similarly to the proofs of Theorems \ref{thm:regrho} and \ref{thm:regtheta}. The regularity and continuity of $\vr_N$ and $\vt_N$ then implies the continuity of the coefficients.

Next we aim to show that $F(D)\subset D$.
The internal energy equation \eqref{m119aL} for $\vt_N$ yields
\begin{align*}
&-\int_I\int_{ \Omega_{\zeta_N}}(a\varrho_N^4+ c_v\vr_N\vartheta_N)\,\partial_t\psi\dx\dt-\psi(0)\int_{ \Omega_{\zeta_N(0)}}(a\varrho_0^4+ c_v\vr_N\vartheta_0)\dx\\
& = \int_I\int_{\Omega_{\zeta_N}}\Big[\bfS^{\ep}(\vartheta_N,\nabla\bfv_N): \nabla\bfv_N -
(\tfrac{a}{3}\vt_N^4+\vr_N\vt_N) \Div \bfv_N \Big]\,\psi \dxt\\ 
&+ \int_I\int_{\Omega_{\zeta_N}}\Big[ \ep\rdiss(\vr_N)  +
\frac{\delta}{\vartheta_N^{2}} - \ep \vartheta_N^5 \Big]\,\psi\dxt
\end{align*}
for all $\psi\in C^\infty_c([0,T))$.
 Combining this with \eqref{eq:1712} implies
\begin{align} \label{eq:0201}
\begin{aligned}
- \int_I \partial_t \psi \,
\mathcal E^N_{\ep,\delta} \dt &=
\psi(0) \mathcal E^N_{\ep,\delta}(0)+\int_I\psi \int_{\Omega_{\zeta_N}}   \left(\frac{\delta}{\vartheta_N^2} - \ep\vartheta_N^5 \right) \dx \dt \\&+ 
\int_I\psi \int_{\Omega_{\zeta}}\Big(\bfS^\ep(\vt_N,\nabla\bfv_N):\nabla\bfv_N-\bfS^\ep(\vt_N,\nabla\bfu_N):\nabla\bfu_N\Big)\dxt\\
&+\int_I\psi \int_{\Omega_{\zeta}}p_\delta(\varrho_N, \vartheta_N) \big(\Div \bfu_N-\Div \bfv_N \big) \dxt\\
&+
\int_I \psi \int_{\Omega_{\eta}}\big( \vr_N H +\varrho_N\bff\cdot\bfu_N\big)\dxt+\int_I\psi\int_\omega g\,\partial_t\eta_N\,\dd y\dt
\end{aligned}
\end{align}
with
\begin{align*}
\mathcal E^N_{\ep,\delta}(t)&= \int_{\Omega_{\zeta_N(t)}}\Big(\frac{1}{2}\varrho_N(t)| {\bf u}_N(t) |^2 + \varrho_N(t) e_\delta(\varrho_N(t),\vartheta_N(t))\Big)\dx\\&+\int_\omega \frac{|\partial_t\eta_N(t)|^2}{2}\,\dd y+ K_\ep(\eta_N(t)).
\end{align*}
By choosing $\psi=\mathbb I_{(0,t)}$, we find that \eqref{eq:0201} implies uniform a-priori estimates. Note that we can apply Young's inequality to the forcing terms in \eqref{eq:0201} and absorb terms containing the unknowns in the left-hand side. %The pressure terms are bounded  by Theorem~\ref{thm:regrho}and Theorem~\ref{thm:regtheta}.
 Moreover, by Theorem~\ref{thm:regtheta}
 we obtain bounds for $\theta_N$ (in dependence of $\varepsilon,\delta, N,K^*$) from below
 such that
\begin{align*}
\int_{I_*}\int_{\Omega_{\zeta_N}}   \frac{\delta}{\vartheta_N^2}\dxt\leq \,c(\varepsilon,\delta,N,K^*)T^{*}\leq 1
\end{align*}
for $T^{*}$ small enough. 
So, in order to apply the Gronwall lemma it is enough to control {the error term
\begin{align*}
\int_{I_*} &\int_{\Omega_{\zeta_N}}\Big(\bfS^\ep(\vt_N,\nabla\bfv_N):\nabla\bfv_N-\bfS^\ep(\vt_N,\nabla\bfu_N):\nabla\bfu_N\Big)\dxt\\
&+\int_{I_*}\int_{\Omega_{\zeta}}p_\delta(\varrho_N, \vartheta_N) \big(\Div \bfu_N-\Div \bfv_N \big) \dxt\\
&\leq\int_{I_*} \int_{\Omega_{\zeta_N}}\Big(\bfS^\ep(\vt_N,\nabla\bfv_N):\nabla\bfv_N+p_\delta(\vr_N,\vt_N)(|\nabla\bfu_N|+|\nabla\bfv_N|)\Big)\dxt.
\end{align*}
Using Theorem~\ref{thm:regrho} and \ref{thm:regtheta} we can bound $\vr_N$ and $\vt_N$ in terms of $K$ such that the above is bounded by
\begin{align*}
&\leq\,c(K)\,\int_{I_*} \int_{\Omega_{\zeta_N}}\big(1+|\nabla\bfv_N|^p\big)\dxt+c(K)\int_{I_*} \int_{\Omega_{\zeta_N}}|\nabla\bfu_N|^2\dxt\\
&\leq \,c(K,N)\,T^\ast\Big(1+\sup_{I^\ast}|\bfbeta_N|^p\Big)+c(K,N)T^\ast\sup_{I_\ast}\int_{\Omega_{\zeta_N}}|\bfu_N|^2\dx\\
&\leq \,c(K,N)\,T^\ast+c(K,N)T^\ast\sup_{I_\ast}\int_{\Omega_{\zeta_N}}\vr_N|\bfu_N|^2\dx.
%&\leq\,\frac{\ep}{2}\int_{I_*}\int_{\Omega_{\zeta_N}}\vt_N^5\dxt+c(\ep)\,\int_I \int_{\Omega_{\zeta_N}}\big(1+|\nabla\bfv_N|^{\frac{5p}{5}}+|(\partial_t\nabla^3\zeta_N)\circ\bfvarphi_{\zeta_N}^{-1}|^2\big)\dxt\\
%&\leq\,\frac{\ep}{2}\int_I\int_{\Omega_{\zeta_N}}\vt_N^5\dxt+\,c(\ep,N)\,T^\ast\Big(1+\sup_{I^\ast}|\bfbeta_N|^{\frac{5p}{4}}\Big)\\
%&\leq\,\frac{\ep}{2}\int_I\int_{\Omega_{\zeta_N}}\vt_N^5\dxt+\,c(\ep,N)\,T^\ast\big(1+ (K^*)^{\frac{5p}{4}}\big)
\end{align*}
%using \eqref{m105}, \eqref{m106} and the definition of $\bfS^{\ep}$.
We choose $T^\ast=T^\ast(\ep,N,K^*)$ small enough such that $c(K,N)\,T^\ast\leq \frac{1}{2}$} and 
and obtain
\begin{align}
\sup_{I_*}E_{\ep,\delta}^N
\leq &c(\bff,H,g,\bfq_0,\eta_0,\eta_1,\varrho_0). 
\end{align}
In particular, we have 
\begin{align}
\label{wWS27}\sup_{I_*}\int_{\Omega_{\zeta_N}}|\bfu_N|^2\dx+
\sup_{I_*}\int_\omega \frac{|\partial_t\eta_N|^2}{2}\,\dd y+ \sup_{I_*}K_\ep(\eta_N)\leq\,c(\bff,H,g,\bfq_0,\eta_0,\eta_1,\varrho_0).
\end{align}
recalling the lower bound for $\varrho_N$ from Theorem \ref{thm:regrho} (b) (which depends on $N$ here).
%\seb{ Hence, we find by interpolation that $\eta_N\in C^\alpha((0,T^*)\times \omega)$. But this implies by \eqref{initial} that
%\[
%\abs{\eta^N(t,x)}\leq \abs{\eta^N(t,x)-\mathcal P\eta_0(x)}+\norm{\eta_0}_\infty+ (T^*)^\alpha c(\bff,H,g,\bfq_0,\eta_0,\eta_1,\varrho_0)\leq \frac{L}{2},
%\]
%for $T^*$ small enough.
%}
% choosing $K^*$ large enough (depending on the data and the parameters including $N$)
Consequently, we see that the mapping $\bfbeta\mapsto\bfalpha$ satisfies $F(D)\subset D$, for $K^{*}$ large enough.\\
Now, we need to prove compactness of $F$ with respect to the $C^{1,\alpha}(\overline I)$ topology.
First we find by Leibnitz rule that
 \[
\partial_t \bfalpha_N=\mathcal A^{-1}\Big(\partial_t(\mathcal A \bfalpha_N)-\partial_t\mathcal A \bfalpha_N\Big).
 \]
Due to \eqref{eq:ide1} and the regularity of $\varrho_N$ and $\vt_N$ from Theorems~\ref{thm:regrho}
and \ref{thm:regtheta}
we have $\partial_t(\mathcal A \bfalpha_N)\in C^1(\overline I_*)$. This can be easily
seen by transforming the integrals in the definitions of the coefficients $\mathcal A$, $\mathcal B$, $\tilde{\mathcal{B}}$ and $\bfc$ to the reference domain and recalling from \eqref{map} that $\bfPsi_{\zeta_N}$ and $\bfPsi^{-1}_{\zeta_N}$ have the same regularity as $\zeta_N$.
Also note that $\bfbeta_N\in C^{1,\alpha}(\overline I_*)$ implies $\zeta_N\in C^{2,\alpha}(\overline I_*)$ by construction. Similarly, we are going to prove that $\partial_t\mathcal A_{i,j}\in C^1(\overline I_*)$.
By taking the test function $\tilde\bfomega_i\circ \bfvarphi_{\zeta_N}^{-1}\cdot \tilde\bfomega_j\circ \bfvarphi_{\zeta_N}^{-1}$ in the continuity equation we find that
\begin{align*}
\partial_t \mathcal A_{i,j}&=\frac{\dd}{\dt}\int_{\Omega_{\zeta_N}}\varrho_N\,\tilde\bfomega_i\circ \bfPsi_{\zeta_N}^{-1}\cdot \tilde\bfomega_j\circ \bfPsi_{\zeta_N}^{-1}\dx
\\
%&=\int_{\Omega_{\zeta}}\varrho_N\,\partial_t(\tilde\bfomega_i\circ \bfPsi_{%\zeta_N}^{-1}\cdot \tilde\bfomega_j\circ \bfPsi_{\zeta_N}^{-1})\dx\\
&=\int_{\partial\Omega_{\zeta}}\partial_t\zeta_N\nu\circ \bfvarphi_{\zeta_N}^{-1}\varrho_N\tilde\bfomega_i\circ \bfPsi_{\zeta_N}^{-1}\cdot \tilde\bfomega_j\circ \bfPsi_{\zeta_N}^{-1}\nu_{\Omega_{\zeta}}\dH\\
&+\int_{\Omega_{\zeta}}\varrho_N\bfv_N\cdot\nabla(\tilde\bfomega_i\circ \bfPsi_{\zeta_N}^{-1}\cdot \tilde\bfomega_i\circ \bfPsi_{\zeta_N}^{-1})\dx\\
&+\varepsilon\int_{\Omega_{\zeta}}\nabla\varrho_N\cdot\nabla(\tilde\bfomega_i\circ \bfPsi_{\zeta_N}^{-1}\cdot \tilde\bfomega_j\circ \bfPsi_{\zeta_N}^{-1})\dx \\
&+\int_{\Omega_{\zeta}}\varrho_N\partial_t\big(\tilde\bfomega_i\circ \bfPsi_{\zeta_N}^{-1}\big)\cdot \tilde\bfomega_j\circ \bfPsi_{\zeta_N}^{-1}\dx\\
&+\int_{\Omega_{\zeta}}\varrho_N\tilde\bfomega_i\circ \bfPsi_{\zeta_N}^{-1}\cdot \partial_t\big(\tilde\bfomega_j\circ \bfPsi_{\zeta_N}^{-1}\big)\dx.
\end{align*}
The last two terms containing the time-derivative behave as $\bfbeta_N$ which is bounded in $C^{1,\alpha}(\overline I_*)$.
Consequently, we find that $\partial_t \bfalpha_N\in C^{1}(\overline I_*)$ with bound depending only on $K$ (and $N$). So, the mapping $F$ is compact by Arcel\'a-Ascoli's theorem. Consequently, there is a fixed point $\bfalpha^\ast$ which gives rise to the solution to \eqref{eq:decuN'} if $T^\ast$ is sufficiently small (depending on $\delta,\varepsilon$, $K^*$ and $N$).  The intervall of existence can be extended by iterating the procedure and gluing the solutions together.
%Note that on setting $\zeta_N=\eta_N$ and $\bfv_N=\bfu_N$ in \eqref{eq:0201} we obtain
%\begin{align} \label{eq:0201B}
%\begin{aligned}
%- \int_I \partial_t \psi \,
%\mathcal E^N_{\delta} \dt &=
%\psi(0) \mathcal E^N_{\delta}(0)+\int_I\psi \int_{\Omega_{\eta_N}}   \left(\frac{\delta}{\vartheta_N^2} - \ep\vartheta_N^5 \right) \dx \dt \\
%&+
%\int_I \psi \int_{\Omega_{\eta_N}}\big( \vr_N H +\varrho_N\bff\cdot\bfu_N\big)\dxt+\int_I\psi\int_\omega g\,\partial_t\eta_N\,\dd\mathcal H^2\dt
%\end{aligned}
%\end{align}
%for all $\psi\in C^\infty_0[0,T)$.

\subsection{Uniform estimates--total energy balance}
\label{subsec:teb}
At this stage $\vt_N$ is still strictly positive by Theorem \ref{thm:regtheta} (with a bound depending on $N$) so we can divide the internal energy defined in \eqref{m119a} by $\vt_N$ to obtain the entropy balance
\begin{equation} \label{apeneq}
\begin{split}
\partial_t (\varrho_N s(\varrho_N, \vartheta_N))
& +
\Div (\varrho_N s (\varrho_N, \vartheta_N) \bfu_N) - \Div \Big[
\Big( \frac{\varkappa(\vartheta_N)}{\vartheta_N} + \delta (
\vartheta_N^{\beta - 1} + \frac{1}{\vartheta^2_N}) \Big) \Grad
\vartheta_N \Big] 
\\
&= \frac{1}{\vartheta_N} \Big[\Big(
\frac{\varkappa_\delta(\vartheta_N)}{\vartheta_N} + \delta ( \vartheta_N^{\beta -
1} + \frac{1}{\vartheta_N^2}) \Big) |\Grad \vartheta_N|^2 + \delta
\frac{1}{{\vartheta_N}^2} \Big] + \frac{\vr_N}{\vt_N} H
\\
&+\frac{1}{\vartheta_N} \Big(\bfS^\ep(\vt_N, \Grad \bfu_N) : \Grad \bfu_N+ \ep\rdiss(\vr_N) - \ep \vartheta_N^4\Big)
\end{split}
\end{equation}
satisfied in $I\times \Omega_{\eta_N}$, together with the boundary condition $\Grad \vt_N \cdot \nu_{\eta_N}|_{\partial\Omega_{\eta_N}} =0$.
%$$\Grad \vt_N \cdot \nu_{\eta_N}|_{\partial\Omega_{\eta_N}} = \frac{(1+\vt_N)}{\varkappa(\vt_N)}|(\partial_t\nabla^3\eta_N)\circ\bfvarphi_{\eta_N}^{-1}|^2.$$ 
 In the weak form it reads as
\begin{align*} %\label{m217*finala}\begin{aligned}
 \int_I \frac{\dd}{\dt}&\int_{\Omega_{\eta_N}} \vr_N s(\varrho_N,\vartheta_N) \, \psi\dxt
 - \int_I \int_{\Omega_{\eta_N}} \big( \vr s(\varrho_N,\vartheta_N) \partial_t \psi + \varrho_N s (\varrho_N,\vartheta_N)\bfu_N \cdot \nabla\psi \big)\dxt
\\& \geq\int_I\int_{\Omega_{\eta_N}}
\frac{1}{\vartheta_N}\bfS^\ep(\vt_N, \nabla \vu_N) : \nabla \bfu_N 
\psi\,\dif x\,\dif t \\
%&+\ep\int_I\int_\omega\,\frac{1}{\vt_N\circ\bfvarphi_{\eta_N}}|\partial_t\nabla^3\eta_N|^2\psi\dH\dt\
& +\int_I\int_{\Omega_{\eta_N}}
\frac{1}{\vartheta_N}\Big[\Big(
\frac{\varkappa(\vartheta_N)}{\vartheta_N} + \frac{\delta}{2} ( \vartheta_N^{\beta -
1} + \frac{1}{\vartheta_N^2}) \Big) |\nabla \vartheta_N|^2 + \delta
\frac{1}{\vartheta_N^2} \Big] 
\psi\,\dif x\,\dif t \\
&+ \int_I\int_{\Omega_{\eta_N}}
\Big( \frac{\varkappa(\vartheta_N)}{\vartheta_N} + \delta (
\vr_N^{\beta - 1} + \frac{1}{\vartheta_N^2}) \Big) \nabla
\vartheta_N \cdot \nabla\psi \dxt+ \int_I \int_{\Omega_{\eta_N}} \frac{\vr_N}{\vt_N} H \psi  \dxt\\
&+ \int_I\int_{\Omega_{\eta_N}}\ep\left[\rdiss(\vr_N)
 -  \vartheta^4_N\right] \frac{\psi}{\vartheta_N} \dxt
\end{align*}
for all
$\psi\in C^\infty(\overline I\times \R^3)$ with $\psi \geq 0$. We combine this with the energy balance proved in \eqref{eq:0201} which reads as (note that in the fixed point we have $\zeta_N=\eta_N$ and $\bfv_N=\bfu_N$)
\begin{align}
\begin{aligned}
- \int_I \partial_t \psi \,
\mathcal E^N_{\delta} \dt &=
\psi(0) \mathcal E^N_{\delta}(0)+\int_I\psi \int_{\Omega_{\eta_N}}   \left(\frac{\delta}{\vartheta_N^2}-\ep \vt_N^5  \right) \dx \dt +
\int_I \psi \int_{\Omega_{\eta_N}} \vr_N H  \dxt\\&+\int_I\int_{\Omega_{\eta_N}}\varrho_N\bff\cdot\bfu_N\dxt+\int_I\psi\int_\omega g\,\partial_t\eta_N\,\dd\mathcal H^2\dt
\end{aligned}
 \label{eq:2807}
\end{align}
with
\begin{align*}
\mathcal E^N_{\delta}(t)&= \int_{\Omega_{\eta_N(t)}}\Big(\frac{1}{2} \varrho_N(t)| {\bf u}_N(t) |^2 + \varrho_N(t) e_\delta(\varrho_N(t),\vartheta_N(t))\Big)\dx\\&+\int_\omega \frac{|\partial_t\eta_N(t)|^2}{2}\,\dd y+ K(\eta_N(t)).
\end{align*}
%We introduce the \emph{ballistic free energy} for some parameter value $\Theta>0$
%\[
%H_\Theta(\vr, \vt) = \vr \left( e(\vr, \vt) - \Theta s (\vr, \vt) \right),\
%H_{\delta, \Theta}(\vr, \vt) = \vr \left( e_\delta(\vr, \vt) - \Theta s(\vr, \vt) \right),
%\]
We follow \cite[Chapter 2, Section 2.2.3]{F}, and obtain by substracting from \eqref{eq:2807} $\Theta$-times the integral of \eqref{apeneq} (or $\Theta$-times the weak formulation tested with $\psi\equiv 1$) to obtain
\begin{align}\label{wWS272}
\begin{aligned}
- \int_I \partial_t \psi& \,
\big(\mathcal E^N_{\delta,\ep}- \Theta\varrho s(\varrho,\vartheta)\big) \dt +\Theta \int_{\Omega_{\eta_N}}\sigma_{\varepsilon,\delta}^{N}\dxt
 +\int_I\psi\int_{\Omega_{\eta_N}}\Big(\ep\vartheta^5-\frac{\delta}{\vt^2}\Big)\dxt\\
 &=
\psi(0) \big(\mathcal E^N_{\delta,\ep}-\Theta\varrho s(\varrho,\vartheta)\big)(0)+\Theta \int_I\psi\int_{\Omega_{\eta_N}}\ep\vartheta^4\dxt\\&+ 
\int_I \psi \int_{\Omega_{\eta_N}} \vr_N H  \dxt+\int_I\int_{\Omega_{\eta_N}}\varrho_N\bff\cdot\bfu_N\dxt+\int_I\psi\int_\omega g\,\partial_t\eta_N\,\dd y\dt,
\end{aligned}
\end{align}
where
\begin{align*}
\sigma_{\varepsilon,\delta}^N&=\frac{1}{\vartheta_N}\Big[\bfS(\vartheta_N,\nabla\bfu_N):\nabla\bfu_N+\ep(1+\vt_N)|\nabla\bfu_N|^p\Big]\\
&\frac{1}{\vartheta_N}\Big[\frac{\varkappa(\vartheta_N)}{\vartheta_N}|\nabla\vartheta_N|^2+\frac{\delta}{2}\Big(\varrho^{\beta-1}+\frac{1}{\vartheta^2_N}\Big)|\nabla\vartheta_N|^2+\delta\frac{1}{\vartheta_N^2}\Big]+\frac{\varepsilon}{\vartheta_N} \rdiss(\vr_N) .
\end{align*}
Consequently, we obtain the estimates
\begin{align*}
&\sup_I\int_{\Omega_{\eta_N}}\vr_N|\bfu_N|^2\dx+\sup_I\int_{\Omega_{\eta_N}}\vr_N^\beta\dx
%+\int_I\int_{\Omega_{\eta_N}}|\nabla\bfu_N|^p\dxt
+\int_I\int_{\Omega_{\eta_N}}|\nabla\bfu_N|^p\dxt\leq\,c,\\
&\ep\sup_I\int_\omega|\nabla^3\eta_N|^2\dH+\sup_{I}\int_\omega \frac{|\partial_t\eta_N|^2}{2}\,\dd y+ \sup_{I}K(\eta_N)\leq\,c,\\
&\sup_I\int_{\Omega_{\eta_N}}\vt_N^4\dx+\int_I\int_{\Omega_{\eta_N}}\frac{1}{\vartheta_N} \Big(
\frac{\varkappa_\delta(\vartheta_N)}{\vartheta_N} + \delta ( \vartheta_N^{\beta -
1} + \frac{1}{\vartheta_N^2}) \Big) |\Grad \vartheta_N|^2+\frac{\vr_N}{\vt_N}H\dxt\leq\,c,
\end{align*}
where $c=(\bff,H,g,\bfq_0,\eta_0,\eta_1,\varrho_0)$ is independent of $N$.
The first estimate together with Poincar\'e's inequality, the boundary condition
$\tr_{\eta_N}\bfu_N$ and bound for $\partial_t\eta_N$ from the second estimate implies% together with Poincar\'e's inequality
that $\bfu_N$ is bounded in $L^p(I;L^p(\Omega_{\eta_N}))$. 
%Consequently,
%we can infer that $\regkap\bfu_N\in L^\infty(I;W^{1,\infty}(\Omega_{\regkap\eta_N}))$ uniformly in $N$. Hence we obtain from Theorem \ref{thm:regrho} (c) that $\varrho_N$ is uniformly bounded from below. This implies boundedness of $\bfu_N$ in $L^\infty(I_*;L^2(\Omega_{\regkap\eta_N}))$.
So, we may choose a subsequence such that
\begin{align}\label{conv1}
\eta_N&\rightharpoonup^\ast\eta\quad\text{in}\quad L^\infty(I,W^{3,2}(\omega)),\\
\label{conv2}\partial_t\eta_N&\rightharpoonup^\ast\partial_t\eta\quad\text{in}\quad L^\infty(I,L^{2}(\omega)),\\
\label{conv3orig}\bfu_N&\rightharpoonup^{\eta}\bfu\quad\text{in}\quad L^p(I;L^p(\Omega_{\eta_N})),\\
\label{conv4orig}\nabla\bfu_N&\rightharpoonup^{\eta}\nabla\bfu\quad\text{in}\quad L^p(I;L^p(\Omega_{\eta_N}))),\\
\label{conv4origp}|\nabla\bfu_N|^{p-2}\nabla\bfu_N&\rightharpoonup^{\eta}\overline{\bfP}\quad\text{in}\quad L^{p'}(I;L^{p'}(\Omega_{\eta_N}))),\\
\label{conv:vrN1}\vr_N&\rightharpoonup^{\eta,*}\vr\quad\text{in}\quad L^\infty(I;L^\beta(\Omega_{\eta_N})),\\
\label{conv:vtN1}\vt_N&\rightharpoonup^{\eta,*}\vt\quad\text{in}\quad L^\infty(I;L^4(\Omega_{\eta_N})),\\
\label{conv:vtN1b}\vt_N&\rightharpoonup^{\eta}\vt\quad\text{in}\quad L^\beta(I;L^{3\beta}(\Omega_{\eta_N})),\\
\label{conv:vtN2}\nabla\vt_N&\rightharpoonup^{\eta}\nabla\vt\quad\text{in}\quad L^2(I;L^2(\Omega_{\eta_N}))),
\end{align}
for some $\overline{\bfP}\in L^{p'}(I\times\Omega_{\eta})$.
This implies
\begin{align}
\eta_N&\rightarrow\eta\quad\text{in}\quad C(\overline I\times \omega).\label{etaN1}
%\regkap\eta_N&\to \regkap\eta \quad\text{in}\quad C^2(I\times\partial\Omega).\label{etaN2}
\end{align}
Compactness of $\vt_N$ can be shown as in \cite[Chapter 3, Section 3.5.3.]{F} using \eqref{apeneq}. It is based on local arguments, which are not effected by the moving shell. Consequently we have
\begin{align}\label{conv:vtNstrong}
\vt_N\rightarrow^\eta\vt\quad\text{in}\quad L^4(I;L^4(\Omega_{\eta_N})).
\end{align}
In order to pass to the limit in various terms
in the equations
we are concerned with the compactness of $\vr_N$. 
Applying Corollary \ref{rem:strong} yields
\begin{align}\label{0103a}
\varrho_N\rightarrow^{\eta}\varrho\quad\text{in}\quad L^2(I;L^2(\Omega_{\eta_N})).
\end{align}
We aim at improving the exponent from 2 to $\beta$ in order to pass to the limit in the pressure. Testing the continuity equation with $\varrho_N^{\beta-1}$
yields
\begin{align}
\label{eq:renormzN}
\begin{aligned}
\int_{\Omega_{\eta_N}}\varrho_N^\beta\dx+\int_0^t&\int_{\Omega_{\eta_N}}\frac{4(\beta-1)}{\beta}\varepsilon |\nabla\varrho_N^\frac{\beta}{2}|^2\dx\ds\\
&=\int_{\Omega_{\eta_N(0)}}\varrho_0^\beta\dx-\int_0^t\int_{\Omega_{\eta_N}}\vr_N^{\beta-1}\Div\bfu_N\dx\ds.
\end{aligned}
\end{align}
Since $p>\beta$, we find that the right hand side is uniformly bounded recalling \eqref{conv4orig} and \eqref{conv:vrN1}. We conclude
 (for a non-relabelled subsequence)
%\begin{align}
%%\label{conv:vrN1}\vr_N^\frac{\beta}{2}&\rightharpoonup^{\eta,*}\vr^\frac{\beta}{2}\quad\text{in}\quad L^\infty(I_*;L^2(\Omega_{\eta_N})),\\
%\label{conv:vrN2}\nabla\vr_N^\frac{\beta}{2}&\rightharpoonup^{\eta}\nabla\vr^\frac{\beta}{2}\quad\text{in}\quad L^2(I;L^2(\Omega_{\eta_N}))).
%\end{align}
\begin{align}\label{0103}
\varrho_N\rightarrow^{\eta}\varrho\quad\text{in}\quad L^\beta(I;L^\beta(\Omega_{\eta_N})).
\end{align}
%To improve the integrability we use
%\eqref{eq:1712} with $\bfv_N=\bfu_N$ and $\xi_N=\eta_N$ and $\psi=\mathbb{I}_{(0,t)}$. The forcing terms can be controlled by a Gronwall-argument such that we obtain
%\begin{align}\label{eq:0610}
%\begin{aligned}
%\int_I\int_{\Omega_{\eta_N}}&\varrho_N^{\beta-2}|\nabla\varrho_N|^2\dxt\\
%&\leq\,c(\bff,g,\bfq_0,\varrho_0,\eta_0,\eta_1)\bigg(1+\int_{I}\int_{\Omega_{\eta_N}}p(\varrho_N,\vartheta_N)\Div\bfu_N\dxt\bigg)
%\end{aligned}
%\end{align}
%neglecting various non-negative terms on the left-hand side. The constant depends on
%$\varepsilon$ and $\delta$ but is independent of $N$. Using the inform bounds
%from \eqref{conv4orig}, \eqref{conv:vrN1} and \eqref{conv:vtN1} we obtain
%\begin{align*}
%\int_{I}\int_{\Omega_{\eta_N}}p(\varrho_N,\vartheta_N)\Div\bfu_N\dxt\leq \,c\,
%\int_{I}\int_{\Omega_{\eta_N}}\Big(\varrho_N^{2\gamma}+\vartheta_N^8+|\nabla\bfu_N|^2\Big)\dxt\leq\,c
%\end{align*}
%choosing $\beta$ large enough. In combination with \eqref{0103} we have (after passing to a subsequence)
%\begin{align}\label{0103a}
%\varrho_N\rightarrow^{\eta}\varrho\quad\text{in}\quad L^q(I_*;L^q(\Omega_{\eta_N})),
%\end{align}
%for some $q>\beta$.
%This is enough to pass to the limit in the nonlinear pressure.
We are, however, still concerned with the term
\begin{align*}
\varepsilon\int_{I}\int_{\Omega_{\eta_N}}\nabla\varrho_N\nabla\bfu_N\cdot\bfphi\dxt,
\end{align*}
which requires compactness of $\nabla\vr_N$.
As for \eqref{eq:renormzN} we have
\begin{align*}
\int_{\Omega_{\eta_N}}\varrho^2\dx+\int_0^t&\int_{\Omega_{\eta}}2\varepsilon |\nabla\varrho|^2\dx\ds\\
&=\int_{\Omega_{\eta_N(0)}}\varrho_0^2\dx-\int_0^t\int_{\Omega_{\eta_N}}2\vr_N\Div\bfu_N\dx\ds.
\end{align*}
and applying Theorem \ref{lem:warme} (b) to the limit version. Due to \eqref{etaN1}, \eqref{0103a}
and the strong convergence of $\vr_N$ we can pass to the limit in all terms in
\eqref{eq:renormzN} expect for the one containing $\nabla\vr_N$. Consequently,
\begin{align*}
\lim_{N\to \infty}\int_0^t\int_{\Omega_{\eta_N}}|\nabla\varrho_N|^2\dx\ds=\int_0^t\int_{\Omega_{\eta}}|\nabla\varrho|^2\dx\ds
\end{align*}
for all $t\in I$, which implies strong convergence of $\nabla \vr_N$ and hence by \eqref{conv4orig}
\begin{align*}
\lim_{N\rightarrow\infty}\int_{I}\int_{\Omega_{\eta_N}}\nabla\varrho_N\nabla\bfu_N\cdot\bfphi\dxt=\int_{I}\int_{\Omega_{\eta}}\nabla\varrho\nabla\bfu\cdot\bfphi\dxt.
\end{align*}

%under the additional assumption that the time interval is sufficiently small. 

% The interval of existence can then be prolongated, \seb{in case the fix point solution $\eta$ satisfies $\norm{\eta}_{L^\infty([0,T^*]\times M)}<\frac{L}2$
%with the same argument applied to the following redefinitions. First we redefine $L =L-\norm{\eta}_{L^\infty([0,T^*]\times M)}$. The parametrisation stays untouched and will be described via the same reference coordinates along the initial outer normal $\nu$ defined via $\partial\Omega$. Otherwise we redefine $\Omega$ as $\Omega_{\regkap\eta(T^*)}$, {$\eta_1$ as $\partial_t\eta(T^*)$}, $\varrho_0$ as $\varrho(T^*)$
%and $\bfq_0$ as $\varrho(T^*)\bfu(T^*)$.
% Hence the interval of existence is only restricted if \seb{$\norm{\eta(T_*)}_\infty=\frac{L}2$}}
\subsection{Compactness of $\partial_t\eta_N$}
\label{ssec:comp}
The effort of this subsection is to prove that
\begin{align}
\partial_t\eta_N&\rightarrow\partial_t\eta\quad\text{in}\quad L^2(I;L^{2}(\omega)).\label{eq:conetatC}
\end{align}
We will show this convergence in the generality we will need also in the subsequent limit procedures in the next section. In particular, we will not make use of any higher regularity beyond $L^\infty_t(L^\gamma_x)$ with $\gamma>\frac{12}{7}$ for the density.\\
The following aim is showning
\begin{align}\label{eq:convrhouN}\begin{aligned}
\int_{I}\int_{\Omega_{\eta_N}}&|\sqrt{\varrho_N}\bfu_N|^2\dxt+\int_{I}\int_\omega|\partial_t\eta_N|^2\,\dd y\dt\\
&\longrightarrow \int_{I}\int_{\Omega_{\eta}}|\sqrt{\varrho}\bfu|^2\dxt+\int_{I}\int_\omega|\partial_t\eta|^2\,\dd y\dt,
\end{aligned}
\end{align}
which implies the strong convergence \eqref{eq:conetatC} by the strict convexity of the $L^2$-norm. Relation \eqref{eq:convrhouN} will be a consequence of 
\begin{align}\label{eq:312N}
\begin{aligned}
\int_{I}\int_{\Omega_{\eta_N}}&\varrho_N\bfu_N\cdot\mathscr F_{\eta_N}\partial_t\mathcal \eta_N\dxt+\int_{I}\int_\omega|\partial_t\eta_N|^2\,\dd y\dt\\
&\longrightarrow \int_{I}\int_{\Omega_{\eta}}\varrho\bfu\cdot\mathscr F_{\eta}\partial_t\eta\dxt+\int_{I}\int_\omega|\partial_t\eta|^2\,\dd y\dt
\end{aligned}
\end{align}
and 
\begin{align}\label{eq:313N}
\int_{I}\int_{\Omega_{\eta_N}}&\varrho_N\bfu_N\cdot(\bfu_N-\mathscr F_{\eta_N}\partial_t\mathcal \eta_N)\dxt
\longrightarrow \int_{I}\int_{\Omega_{\eta}}\varrho\bfu\cdot(\bfu-\mathscr F_{\eta}\partial_t\eta)\dxt.
\end{align}
First observe that (due to the trace theorem Lemma~\ref{lem:2.28}) we find that $\partial_t\eta_N$ possesses some compactness in space. To be precise, we have
\begin{align}\label{eq:1812}
 \norm{\partial_t\eta_N}_{L^2(I,W^{1-\frac{1}{r},r}(\omega))}+ \norm{\partial_t\eta_N}_{L^2(I, L^\ell(\omega))}\leq\,c
 \end{align}
 for all $r<2$ and $\ell<4$. The bounds only depend on the $L^2_t(W^{1,2}_x)$ bounds of $\bfu_N$ and hence are uniform by estimates \eqref{conv3orig} and \eqref{conv4orig}.
 We define the projection
\begin{align*}
\mathcal P_N w=\sum_{k=1}^N\alpha_k(w)w_k,\quad
\mathcal P_N^\zeta w=\sum_{k=1}^N\alpha_k(w)\tilde\bfomega_k\circ\bfPsi_\zeta^{-1},
\end{align*}
where {$\alpha_k(w)=\langle w,w_k\rangle_{W^{3,2}(\omega)}$} if $w_k=\tilde{Y}_\ell$ for some $\ell\in\N$ and
$\alpha_k(w)=0$ otherwise. Obviously, we have $\tr_\zeta \mathcal P_N^\zeta w=\mathcal P_N w$ for any $w\in W^{3,2}(\omega)$. We have by definition,
\begin{align}\label{eq:PN}
\|\mathcal P_N w\|_{W^{3,2}(\omega)}^2\leq\,\|w\|_{W^{3,2}(\omega)}^2\quad \forall w\in W^{3,2}(\omega).
\end{align}
The eigenvalue equation for the basis vectors implies additionally that
\begin{align}\label{eq:PN2}
\|\mathcal P_N w\|_{L^{2}(\omega)}^2\leq\,c\|w\|_{L^2(\omega)}^2\quad \forall w\in L^2(\omega).
\end{align}
Moreover, by definition of $\tilde\bfY_k$ and $\mathscr F_\zeta$ (see Section \ref{sec:ext}) we have
\begin{align}\label{eq:PNzeta}
\mathcal P_N^\zeta w=\mathscr F_{\zeta}(\mathcal P_N w)
\end{align}
for all $w\in W^{3,2}(\omega)$. Finally, we note that
$\mathcal P_N \eta_N=\eta_N$ such that $(\eta_N,\mathscr F_{\eta_N}\eta_N)$
%and $(\partial_t\eta_N,\mathscr F_{\Omega_{\eta}}(\partial_t\eta_N))$ 
is admissible in \eqref{eq:decuN'}.
Due to  the uniform a priori bounds from the last subsection and the respective embeddings, we find that the convergence in \eqref{eq:313N} follows directly from Lemma \ref{thm:weakstrong}
with the choices $v_N=\bfu_N-\mathscr F_{\eta_N}\partial_t\mathcal \eta_N$, $r_N=\mathcal P_N^{\eta_N}(\varrho_N\bfu_N)$ (which solves the projected equation \eqref{eq:decuN'} in the domain $\Omega_{\eta_N}$) and the continuity of the projection operator $\mathcal P^{\eta_N}_N$ defined above (recall also \eqref{eq:PNzeta}). The corresponding uniform estimates are given in the previous subsection and the weak convergence of $\mathscr F_{\eta_N}\partial_t\mathcal \eta_N$ follows from \eqref{conv1}, \eqref{conv2}, Lemma~\ref{lem:3.8'} and Corollary \ref{cor:2807'}.\\
In order to prove \eqref{eq:312N} we need to make use of the coupled momentum equation using Theorem~\ref{thm:auba}. We define $g_N=(\partial_t\eta_N,\vr_N\bfu_N\mathbb I_{\Omega_{\eta_N}})$ and $f_N=(\partial_t\eta_N,\mathscr F_{\eta_N}\partial_t\mathcal \eta_N)$ noticing that (by construction) $\Omega_{\eta_N}\subset \Omega\cup S_{L/2}$ as well as for all $s<\frac{1}{2}$ and $q<3$ 
$$\mathscr F_{\eta_N}\partial_t\mathcal \eta_N\in L^2(I;W^{s,q}(\Omega\cup S_{L/2}))$$
uniformly in $N$. The last observation is a consequence of \eqref{eq:1812} and Lemma~\ref{lem:3.8'} (a). In particular, we have
\begin{align}\label{eq:fgN1}
f_N\rightharpoonup f\quad\text{in}\quad L^2(I;X),
\end{align}
where $f=(\partial_t\eta,\mathscr F_{\eta}\partial_t\mathcal \eta)$ %and \db{$Y=W^{s_y,2}(\omega)\times W^{s_y,q}(\Omega\cup S_{L/2})$} for some $s_y<\frac{1}{2}$ (close to $\frac{1}{2}$).
and
\begin{align}\label{eq:fgN2}
g_N\rightharpoonup g\quad\text{in}\quad L^2(I;X'),
\end{align}
where $X=L^2(\omega)\times W^{s_x,q}(\Omega\cup S_{L/2})$ with $s_x<s_y<\frac12$ (such that $X'=L^2(\omega)\times W^{-s_x,q'}(\Omega\cup S_{L/2})$), since\footnote{Here, this follows easily from \eqref{0103a}, but it will be critical in the final limit $\delta\rightarrow0$.}
\begin{align}\label{eq:wwwwww}
\vr_N\bfu_N\rightharpoonup^{\eta}\vr\bfu\quad\text{in}\quad L^2(I;L^{\frac{6\gamma}{\gamma+6}}(\Omega_{\eta_N}))
\end{align}
and $L^{\frac{6\gamma}{\gamma+6}}_x\hookrightarrow W^{-s_x,q'}_x$ due to $\gamma>\frac{12}{7}$ (choosing $s_x$ sufficiently close to $1/2$ and $q$ close to 3).
%Moreover, by \eqref{eq:1812} and Lemmas \ref{lem:2.28} and \ref{prop:2.28} (together with \eqref{conv4orig}) we have that $$g_N\in L^2(I,W^{s,2}(\omega))\times L^2(I,W^{-s,2}(\Omega\cup S_{L/2}))$$
%uniformly in $N$ for all $s\leq\frac{1}{4}$. We define $X=L^2(\omega)\times W^{-s,2}(\Omega\cup S_{L/2})$, $X'=L^2(\omega)\times W^{s,2}(\Omega\cup S_{L/2})$ and $Y=W^{s,2}(\omega)\times L^{2}(\Omega\cup S_{L/2})$ for some $0<s\leq\frac14$. 
Further we define
\[
Z=W^{1,2}(\omega)\times W^{
1,q}(\Omega\cup S_{L/2})%{\rm span}\{(\omega_k,\bfomega_k\circ \bfPsi_{\eta_N(t)}^{-1}),k=1,\dots N\}.
\]
Boundedness of $g_N$ in $L^\infty(I;Z')$ follows now from, \eqref{conv2}, $\vr_N\bfu_N\in L^2_t(L^{\frac{2\beta}{\beta+1}}_x)$ uniformly
and the embedding $L^{\frac{2\beta}{\beta+1}}_x\hookrightarrow W^{-1,2}_x\hookrightarrow W^{-1,q}_x$ for $\beta>\frac{3}{2}$ and $q\geq 2.$
%which we equip with the $W^{s_y,2}(\omega)\times W^{1,q}(\Omega_{\eta^N(t)})$-norm for some $s_y\in(\frac{1}{q},s_x)$ . 
%Using the canonical extension by zero we can consider (for any fixed $t$ and $N$) $Z^N(t)$ as a subspace of $X$.
% In the subsequent limit procedures we will instead work with spaces of the form
%\begin{align*}
%Z^N(t)=\big\{(b,\bfvarphi)\in W^{s_y,2}(\omega)\times W^{1,q}(\Omega_{\eta_N(t)}):\,\tr_{\eta_N}\bfvarphi=b\big\}.
%\end{align*}
%Note that Lemma \ref{lem:2.28} guarantees the existence of a trace. Moreover, in both cases we can combine the trace theorem (Lemma \ref{lem:2.28}) and the extension operator from Lemma \ref{lem:3.8'} to extend functions from $\Omega_{\eta_N(t)}$ to $\Omega\cup S_{L/2}$ provided we allow some loss in integrability. In particular, we can assume that for all $t\in I$
%\begin{align}\label{ext:ZY}
%Z^N(t)\hookrightarrow Y,
%\end{align}
%where the embedding constant only depends on the $L^\infty_tW^{2,2}_x$-norm of $\eta_N$, which is uniformly bounded.\\ 
The conditions $(a)$ in Theorem~\ref{thm:auba} follow now from \eqref{eq:fgN1} and \eqref{eq:fgN2} by weak compactness. For $(b)$  we observe that we may assume that a regularizer $b\mapsto (b)_\kappa$ exists such that for any $s,a\in \R$ and $p\in[1,\infty)$
\begin{align}
\label{eq:molly}\norm{b-(b)_\kappa}_{W^{a,p}(\omega)}\leq \,c\kappa^{s-a}\norm{b}_{W^{s,p}(\omega)},\quad b\in W^{s,p}(\omega).
\end{align}
The estimate is well-known for $a,s\in\N_0$, while the general case follows by interpolation and duality. Moreover, since we use standard Fourier bases in $W^{3,2}(\omega)$ for the discretisation of $\eta_N$, we find by interpolation that the projection error satisfies the following stability estimates for all $s\in [0,3]$
\begin{align}
\label{eq:projection}
\norm{\mathcal P^N b}_{W^{s,2}(\omega)}\leq c\norm{b}_{W^{s,2}(\Omega)}.
\end{align}
Next we introduce the mollification operator on $\partial_t\eta_N$ by considering for $\kappa>0$ and $N\in\N$
$
\mathcal P^N((\partial_t\eta_N)_\kappa)
$
and set
\[
f_{N,\kappa}(t):=(\mathcal P^N((\partial_t\eta_N(t))_\kappa),\mathscr F_{\eta_N(t)}(\mathcal P^N((\partial_t\eta_N(t))_\kappa))).
\]
We find by the continuity of the mollification operator from \eqref{eq:molly}, the continuity of the projection
operator from \eqref{eq:projection} and the estimate for the extension operator (due to \eqref{conv1} and Lemma~\ref{lem:3.8'}) that for a.e.\ $t\in (0,T)$
 \begin{align}\label{eq:fNkappa}
 \norm{f_{N,\kappa}-f_N}_{L^2(\omega)\times W^{s_x,q}(\Omega\cup S_{L/2})}\leq \,c\kappa^{s_y-s}\norm{\partial_t\eta_N}_{W^{s_y,2}(\omega)},
 \end{align}
 which can be made arbitrarily small in $L^2$ choosing $\kappa$ appropriately, cf. \eqref{eq:1812}.
% This allows to deduce the $L^1(I;Z)$ bound (for $s=s_y$) as well as the convergence property (choosing $s=s_x$).
Similarly, we have
 \begin{align*}\label{eq:fNkappa}
 \norm{f_{N,\kappa}}_{W^{1,2}(\omega)\times W^{1,q}(\Omega\cup S_{L/2})}\leq \,c\kappa^{-1}\norm{\partial_t\eta_N}_{L^{2}(\omega)}.
 \end{align*}
Moreover, by \eqref{eq:fgN1} we clearly can deduce a converging subsequence such that $f_{N,\kappa}\weakto f_{\kappa}$ (for some $ f_{\kappa}$) in $L^2(I;X)$ for any $\kappa>0$,
 which implies (b).\\ 
 For (c) have to control
$\skp{g_N(t)-g_N(s)}{f_{N,\kappa}(t)}$
and hence decompose
%$f_{M,\kappa,t}(t)\mathbb I_{(s,t)}$ as test-function in \eqref{eq:decuN'}. Due to the a priori estimates from the previous subsection all terms in \eqref{eq:decuN'} have a time-integrability
%strictly larger than 1. On the other, $f_{M,\kappa,t}(t)$ is bounded in time and smooth in space with bounds depending on $\kappa$ but independent of $M$. We conclude
%\begin{align*}
%|\skp{g_N(t)-g_N(s)}{f_{M,\delta,t}(t)}|\leq\,c_\kappa|t-s|^\alpha
%\end{align*}
%for all $N>M\gg1$ for some $\alpha>0$, which is (d).
%\db{kann ich $N>M$ annehmen?}\\
%Further we deviate
 \begin{align*}
&\skp{g_N(t)-g_N(s)}{f_{N,\kappa}(t)}
\\
&=\big(\skp{g_N(t)}{(\mathcal P^N((\partial_t\eta_N(t))_\kappa),\mathscr F_{\eta_N(t)}(\mathcal P^N((\partial_t\eta_N(t))_\kappa)))}
\\
&\quad-\skp{g_N(s)}{(\mathcal P^N((\partial_t\eta_N(t))_\kappa),\mathscr F_{\eta_N(s)}(\mathcal P^N((\partial_t\eta_N(t))_\kappa)))}\big)
\\
&\quad +\skp{g_N(s)}{(0,\mathscr F_{\eta_N(t)}(\mathcal P^N((\partial_t\eta_N(t))_\kappa))-\mathscr F_{\eta_N(s)}(\mathcal P^N((\partial_t\eta_N(t))_\kappa))}=:(I)+(II).
 \end{align*}
 %\db{how to control $(III)$ and $(IV)$??}
We begin estimating $(II)$ using Corollary~\ref{cor:2807'} to find that
 \begin{align*}
(II)& = \int_s^t \int_{\Omega_{\eta_N(s)}} \vr_N(s)\bfu_N(s)\cdot \partial_\theta \mathscr F_{\eta_N(\theta)}(\mathcal P^N((\partial_t\eta_N)_\kappa)(t)\dx\,\dd\theta
\\
%&\leq\norm{\vr_N\bfu_N(s)\mathbb I_{\Omega_{\eta_N(s)}}}_{W^{-s_x,q'}_x}\int_s^t \| \partial_\theta \mathscr F_{\eta_N(\theta)}(\mathcal P^N((\partial_t\eta_N)_\kappa)(t)\|_{W^{s_x,q}}\,\dd\theta\\ 
&
\leq c\norm{\vr_N\bfu_N(s)}_{L^{\frac{6\gamma}{\gamma+6}}(\Omega_{\eta_N(s)})}\abs{s-t}^\frac{1}{2}\bigg(\int_I\norm{\partial_t\eta_N(\theta)}_{L^\ell(\omega)}^2\bigg)^{\frac{1}{2}}\norm{\mathcal P^N((\partial_t\eta_N))_\delta(t)}_{L^\infty(\omega)}
 \end{align*}
 for some $\ell<4$ (recall that $\gamma>\frac{12}{7}$).
By Sobolev's embedding's, \eqref{eq:molly} and \eqref{eq:projection} the last term can be estimated by
\begin{align*}
\norm{\mathcal P^N((\partial_t\eta_N)_\kappa)(t)}_{L^\infty(\omega)}&\leq c\norm{\mathcal P^N((\partial_t\eta_N)_\kappa(t)}_{W^{3,2}(\omega)}\leq c\norm{(\partial_t\eta_N)_\kappa(t)}_{W^{3,2}(\omega)}\\&\leq c\kappa^{-3}\norm{\partial_t\eta_N(t)}_{L^2(\omega)},
\end{align*}
which is bounded to to \eqref{conv2}. Using also
\eqref{eq:1812} we conclude
$$(II)\leq\,c(\kappa)\norm{\vr_N\bfu_N(s)}_{L^{\frac{6\gamma}{\gamma+6}}(\Omega_{\eta_N(s)})}\abs{s-t}^\frac{1}{2}$$
  The term $(I)$ is estimated using the test-function $\mathbb I_{(s,t)}f_{N,\kappa}$ in \eqref{eq:decuN'}. One obtains the uniform H\"older estimate in a similar sense as for $(II)$ using the various estimates on the extension, projections, embeddings and H\"older's inequality. We explain here in detail only the two most complicated terms stemming from the time derivative and the pressure. All other terms can be estimated analogously by simpler means. First, we consider the term acting on the time derivative. Observe that this term only appears due to the time-dependent extension. We choose $a$ such that $\frac{1}{a}+\frac{1}{\gamma}+\frac{1}{6}=1$. Then by the assumption $\gamma>\frac{12}{7}$, we find that $a<4$. Hence we can choose $a_0\in(a,4)$ and $\chi\in (0,1)$ such that $\frac{1}{a}=\frac{\chi}{2}+\frac{1-\chi}{a_0}$ and
  
  \[
  \norm{\partial_t\eta_N}_{L^a(\omega)}
  \leq  \norm{\partial_t\eta_N}_{L^2(\omega)}^\chi \norm{\partial_t\eta_N}_{L^{a_0}(\omega)}^{1-\chi}.
  \]
Using Corollary~\ref{cor:2807}, \eqref{eq:molly} and \eqref{eq:projection} we obtain
  \begin{align*}
&\Bigabs{\int_s^t\int_{\Omega_{\eta_N(\theta)}}\vr_N\bfu_N\cdot \partial_\theta \mathscr F_{\eta_N(\theta)}(\mathcal P^N((\partial_t\eta_N)_\kappa)(t)\dx\,\dd\theta}
\\
&\quad \leq c\int_s^t\norm{\vr_N}_{L^\gamma(\Omega_{\eta_N})}\norm{\bfu_N}_{L^6(\Omega_{\eta_N})}\norm{\mathscr F_{\eta_N(\theta)}(\mathcal P^N((\partial_t\eta_N)_\kappa)}_{L^{a}(\Omega_{\eta_N})}\,\dd\theta
\\
&\quad \leq c\int_s^t\norm{\vr_N}_{L^\gamma(\Omega_{\eta_N})}\norm{\bfu_N}_{L^6(\Omega_{\eta_N})}\norm{\partial_t\eta_N(\theta)}_{L^{a}(\omega)}\norm{\mathcal P^N((\partial_t\eta_N)_\kappa)}_{L^{\infty}(\omega)}\,\dd\theta
\\
&\quad 
\leq c\norm{\vr_N}_{L^\infty(I;L^\gamma)}
\norm{\partial_t\eta_N}_{L^\infty(I;L^{2}(\omega)}^{1-\chi}\norm{\mathcal P^N((\partial_t\eta_N))_\kappa(t)}_{W^{3,2}(\omega)}\int_s^t
\norm{\partial_t\eta_N}_{L^{a_0}(\omega)}^\chi\norm{\bfu_N}_{L^6(\Omega_{\eta_N})}\,\dd\theta
\\
&\quad 
\leq c\kappa^{-3}\abs{s-t}^\frac{1-\chi}{2}\norm{\partial_t\eta_N}_{L^2(I;L^{a_0}(\omega)}^{\chi}\norm{\bfu_N}_{L^2(I;L^6(\Omega_{\eta_N})}\leq c\kappa^{-3}\abs{s-t}^\frac{1-\chi}{2},
%\\
%&\leq c\abs{s-t}^\frac{1}{2}\norm{\mathcal P^N((\partial_t\eta_N))_\delta(t)}_{L^\infty(\omega)}
  \end{align*}
where the constant depends on the a priori estimates only.
%  Similarly, for the convective term we choose $\tilde{a}<3$ and $\ell\in (1,\infty)$ such that $\frac{1}{\tilde{a}}+\frac{1}{\gamma}+\frac{1}{\ell}=1$. In this case we define $\chi\in (0,1)$ such that $\frac{1}{2\tilde{a}}=\frac{\chi}{2}+\frac{1-\chi}{6}$ and
%  \[
%  \norm{\bfu_N}_{L^{\tilde{2a}}(\Omega_{\eta_N})}
%  \leq  \norm{\bfu_N}_{L^2(\Omega_{\eta_N})}^\chi \norm{\bfu_N}_{L^{6}(\Omega_{\eta_N})}^{1-\chi}.
%  \]
As far as the pressure is concerned, H\"older's inequality and Lemma \ref{lem:3.8} (b) imply
     \begin{align*}
&\Bigabs{\int_s^t\int_{\Omega_{\eta_N(\theta)}}p_\delta(\vr_N,\vt_N)\, \Div \mathscr F_{\eta_N(\theta)}(\mathcal P^N((\partial_t\eta_N)_\kappa)(t)\dx\,\dd\theta}
\\
&\quad \leq c\|p_\delta(\vr_N,\vt_N)\|_{L^\infty(I;L^1(\Omega_{\eta_N}))}\int_s^t\norm{\nabla \mathscr F_{\eta_N(\theta)}(\mathcal P^N((\partial_t\eta_N)_\kappa)}_{L^{\infty}(\Omega_{\eta_N})}\,\dd\theta
\\
&\quad \leq c\int_s^t(1+\|\nabla\eta_N\|_{L^\infty(\omega)})\|\mathcal P^N((\partial_t\eta_N)_\kappa)\|_{W^{1,\infty}(\omega)}\,\dd\theta\\
&\quad \leq c\int_s^t(1+\|\nabla\eta_N\|_{L^\infty(\omega)})\|(\partial_t\eta_N)_\kappa)\|_{W^{3,2}(\omega)}\,\dd\theta\\
&\quad 
\leq c\kappa^{-3}\|\partial_t\eta_N\|_{L^\infty(I;L^2(\omega))}\int_s^t(1+\|\nabla\eta_N\|_{L^\infty(\omega)})\,\dd\theta\\
&\quad 
\leq c\kappa^{-3}|t-s|^{\frac{1}{2}}\bigg(\int_I(1+\|\nabla\eta_N\|^2_{L^\infty(\omega)})\,\dd\theta\bigg)^{\frac{1}{2}}\leq c\kappa^{-3}|t-s|^{\frac{1}{2}}
  \end{align*}  
  provided that we have
  \begin{align}\label{3112a}
p_\delta(\vr_N,\vt_N)&\in L^\infty(I;L^1(\Omega_{\eta_N})),\quad
\partial_t\eta_N\in L^\infty(I;L^2(\omega)),\\
\nabla\eta_N&\in L^2(I;L^\infty(\omega)),\label{3112b}
  \end{align}
 uniformly in $N$. While \eqref{3112a} follows here and on the subsequent directly form the energy estimates, we need some further regularity for \eqref{3112b}. On this level it follows from the regularisation of the shell equation, cf. \eqref{conv1}.\\
%   \begin{align*}
%&\Bigabs{\int_s^t\int_{\Omega_{\eta_N(\theta)}}\vr_N\bfu_N\otimes \bfu_N: \nabla \mathscr F_{\eta_N(\theta)}(\mathcal P^N((\partial_t\eta_N)_\kappa)(t)\dx\,\dd\theta}
%\\
%&\quad \leq c\int_s^t\norm{\vr_N}_{L^\gamma(\Omega_{\eta_N})}\norm{\bfu_N}^2_{L^{2\tilde{a}}(\Omega_{\eta_N})}\norm{\nabla \mathscr F_{\eta_N(\theta)}(\mathcal P^N((\partial_t\eta_N)_\kappa)}_{L^{\ell}(\Omega_{\eta_N})}\dd\theta
%\\
%&\quad
% \leq c\norm{\vr_N}_{L^\infty(s,t;L^\gamma(\Omega_{\eta_N}))}\norm{\mathcal P^N((\partial_t\eta_N)_\kappa)}_{W^{1,\ell}(\omega)}\int_s^t\norm{\bfu_N}_{L^6(\Omega_{\eta_N})}^{2(1-\chi)}\norm{\bfu_N}_{L^2(\Omega_{\eta_N})}^{2\chi}\dd\theta
%\\
%&\quad 
%\leq c\kappa^{-3}\norm{\partial_t\eta_N}_{L^2(\omega)}\abs{s-t}^{1-\sigma}.
%  \end{align*} 
%Using \eqref{eq:fNkappa}, we have
%\begin{align*}
%(III)&\leq\big(\|g_N(t)\|_{X'}+\|g_N(s)\|_{X'}\big)\|f_{N,\kappa,t}-f_{M,\kappa,t}\|_{X}\\
%&\leq\,c\kappa^{s_y-s_x}\big(\|g_N(t)\|_{X'}+\|g_N(s)\|_{X}\big)\big(\|f_{N}(t)\|_{Y}+\|f_{M}(t)\|_{Y}\big)\\
%&\leq\,c\kappa^{s_y-s_x}\big(\|g_N(t)\|_{X'}^2+\|g_N(s)\|_{X'}^2+\|f_{N}(t)\|_{Y}^2+\|f_{M}(t)\|^2_{Y}\big).
%\end{align*}
In conclusion, %for $\xi>0$ given
 we can now choose  $\alpha \in (0,1)$ close enough to one and conclude that for $\tau>0$ and $t\in [0,T-\tau]$
 \begin{align*}
&\Big|\dashint_0^\tau\skp{g_N(t)-g_N(t+s)}{f_{N,\kappa}(t)}\,\dd s\Big|\leq 
c\kappa^{-3}\tau^{1/2}\big(A_N(t)+1\big),
\end{align*}
where 
\begin{align*}
A_N(t)&=\|g_N(t)\|_{X'}^2+\|f_{N}(t)\|_{X}^2+\norm{\vr_N\bfu_N(t)}_{L^{\frac{6\gamma}{\gamma+6}}(\Omega_{\eta_N})}
\\
&\quad +\dashint_0^\tau\Big(\|g_N(s)\|_{X'}^2+\|f_{N}(s)\|_{X}^2+\norm{\vr_N\bfu_N(s)}_{L^{\frac{6\gamma}{\gamma+6}}(\Omega_{\eta_N})}\Big)\,\dd s
\end{align*} 
 uniformly bounded in $L^1(I)$ due to  \eqref{eq:fgN1} and \eqref{eq:fgN2} and \eqref{eq:wwwwww}.\\
Finally, the condition on $(4)$ follows by the usual compactness in (negative) Sobolev spaces.
%is easily checked. Indeed take a sequence $a_n$ that is uniformly bounded in $L^2(\omega)\times W^{-s_x,q'}(\Omega\cup S_{L/2})$ then there exists a subsequence that is compact in $W^{-s_y,2}(\omega)\times W^{-s_y,q'}(\Omega\cup S_{L/2})$. \db{Recalling \eqref{ext:ZY} we finally have
%\begin{align*}
%&\sup_{\{b\in Z^n(t) : \norm{b}_{W^{s_y,2}(\omega)\times W^{1,q}(\Omega_{\eta_N(t)})}\leq 1\}}\skp{a_n-a_m}{b}_{X,X'}\\&\qquad\qquad\leq\,c \sup_{\{b\in Y:\norm{b}_{Y}\leq 1\}} \skp{a_n-a_m}{b}_{X,X'}\to 0 
%\end{align*}
%as $n,m\to\infty$}, which implies the assertion in (e).

\subsection{Compactness of the shell energy}
\label{subsec:shell}
%Due to
% we can use the boundary condition $\partial_t\eta_N=\tr_{\regkap \eta_N}\bfu_N\cdot\nu$ to conclude
%with the help of \eqref{conv4orig} and Lemmas \ref{lem:sobdif} and \ref{lem:2.28}
%that 
%\begin{align}
%\partial_t\eta_N&\rightharpoonup\partial_t\eta\quad\text{in}\quad L^2(I;W^{3-1/r,r}(M)),\label{eq:conetatN}
%\end{align}
%for all $r<2$. Note that $\nabla^3\regkap \eta_N$ is uniformly bounded with a bounded independent of $N$ (but dependent on $\kappa$) due to \eqref{conv1}. This and 
%\seb{Moreover, we find (for a subsequence)
%the convergence of $\varrho_N\bfu_N\cdot \bfu_N \psi\to \varrho\bfu\cdot \bfu \psi$, for all smooth $\psi$ with compact support in $[0,T]\times \Omega_{\eta}$. Indeed, since by the uniform convergence there exists a $N_0$ such that $\supp(\psi)\subset [0,T]\times \Omega_{\eta_N}$ for all $N>N_0$ Lemma~\ref{thm:weakstrong}, due to \eqref{conv3}. Since $\varrho_N$ is uniformly strictly positive, we find that $\bfu_N\to \bfu$ strongly. The compactness of the trace then also implies that $\partial_t \eta_N\to \partial_t \eta$ in $L^2([0,T]\times\omega)$.
%}
In order to complete the proof of \ref{K4} it remains to justify the limit in the shell energy.
Since we have a regularized system \eqref{conv1} yields for any $q<\infty$
\begin{align}\label{eq:conetastrong}
\eta_N&\rightarrow\eta\quad\text{in}\quad L^q(I;W^{2,q}(\omega)),
\end{align}
which is enough to conclude
\begin{align}
\label{eq:1401bN}\lim_{N\rightarrow \infty}\int_I\psi K(\eta_N)\dt&=\int_I \psi K(\eta)dt
\end{align}
for all $\psi\in C^\infty_c(I)$ (this step will be much harder on the subsequent levels, see Section \ref{sec:shelleps}). It remains to show the convergence of the regularizer
\begin{align}
\label{eq:1401bNN}\lim_{N\rightarrow \infty}\int_I\psi \mathcal L(\eta_N)\dt&=\int_I \psi \mathcal L(\eta)\dt.
\end{align}
Next, we can assume that
\begin{align}
\partial_t\eta_N&\rightharpoonup\partial_t\eta\quad\text{in}\quad L^2(I;W^{1-1/r,r}(\omega)),\label{eq:conetatA}
\end{align}
for all $r<2$ due to \eqref{eq:1812}. 
We infer from \eqref{eq:decuN'} using $(\psi\eta_N,\psi\mathscr F_{\eta_N}(\eta_N))$ as a test-function
\begin{align}\label{eq:1911}
\begin{aligned}
\int_I\psi\int_\omega K'_\ep(\eta_N)&\eta_N\dH\dt
=\int_I\int_{\Omega_{\eta_N}}\varrho_N\bfu_N\cdot \partial_t\Big(\psi\mathscr F_{\eta_N}(\eta_N)\Big)\dxt\\&+\int_I\psi\int_{\Omega_{\eta_N}}\varrho_N\bfu_N\otimes \bfu_N:\nabla \mathscr F_{\eta_N}(\eta_N)\dxt
\\
&+\int_I\psi\int_{\Omega_{\eta_N }}\bfS^\ep(\vt_N,\nabla\bfu_N):\nabla\mathscr F_{\eta_N}(\eta_N)\dx\dt
\\
&+\int_I\psi\int_{\Omega_{\eta_N }}\Big(
p_\delta(\varrho_N,\vt_N)\,\Div\mathscr F_{\eta_N}(\eta_N)+\varepsilon\nabla\varrho_N\nabla\bfu_N\mathscr F_{\eta_N}(\eta_N)\Big)\dx\dt
\\
&+\int_I\psi\int_\omega \partial_t\eta_N\,\partial_t(\psi \eta_N)\dH\ds
\\&+\int_I\psi\int_{\Omega_{\eta_N}}\varrho_N\bff\cdot\mathscr{F}_{\eta_N}(\eta_N)\dxt+\int_I\psi\int_\omega g\,\eta_N\,\dH\dt\\
&+\psi(0)\int_{\Omega_{\eta_N(0)}}\bfq_0\cdot\mathscr{F}_{\eta_N}(\eta_N)(0,\cdot)\dx+\psi(0)\int_\omega\eta_1\,\eta_N\,\dd y.
\end{aligned}
\end{align}
The terms on the right-hand side related to the shell clearly converge to their expected limits because of \eqref{conv1} and \eqref{eq:conetatC}.
On account of Lemma \ref{lem:3.8'} and Corollary \ref{cor:2807'} we have
%\db{could be improved using $\ep \mathcal L$, but not needed}
\begin{align*}
\|\partial_t(\mathscr F_{\eta_N}(\eta_N))\|_{L^2_tL^{q_1}_x}+\|\mathscr F_{\eta_N}(\eta_N))\|_{L^\infty_tW^{1,q_2}_x}+\|\mathscr F_{\eta_N}(\eta_N))\|_{L^\infty_tW^{2,q_3}_x}\leq\,c
\end{align*}
uniformly in $N$ for all $q_1<4$, $q_2<\infty$ and $q_3<2$, cf. \eqref{conv1} and \eqref{eq:conetatC}. In particular, applying standard compact embeddings we can choose a subsequence  (not relabelled) such that
\begin{align*}
\partial_t(\mathscr F_{\eta_N}(\eta_N))&\rightharpoonup\partial_t(\mathscr F_{\eta}(\eta))\quad\text{in}\quad L^2(I;L^{q_1}(\Omega\cup S_{L/2})),\\
\mathscr F_{\eta_N}(\eta_N)&\rightarrow\mathscr F_{\eta}(\eta)\quad\text{in}\quad L^{q_2}(I;W^{1,q_2}(\Omega\cup S_{L/2})),\\
\mathscr F_{\eta_N}(\eta_N)&\rightarrow\mathscr F_{\eta}(\eta)\quad\text{in}\quad L^{\infty}(I;L^{\infty}(\Omega\cup S_{L/2})),
%\mathscr F_{\eta_N}(\eta_N)&\rightharpoonup\mathscr F_{\eta}(\eta)\quad\text{in}\quad %L^2(I;W^{2,q_3}(\Omega\cup S_{L/2})),
\end{align*}
for all $q_1<4$ and $q_2<\infty$.
Combining these convergences with the convergences form the last subsection we can pass to the limit in the
terms on the right-hand side of \eqref{eq:1911} related to the fluid system as well. On the other hand,
the resulting expression coincides with $\int_I\psi K(\eta)\dt$ as can be seen
from testing the limit system with $(\psi\eta,\psi\mathscr F_{\eta}(\eta))$. We conclude that
\begin{align*}
\ep\int_I\psi \mathcal L(\eta_N)\dt&=\frac{\ep}{2}\int_I\psi \mathcal L'(\eta_N)\eta_N\dt=\frac{1}{2}\int_I\psi K_\ep'(\eta_N)\eta_N\dt-\frac{1}{2}\int_I\psi K'(\eta_N)\eta_N\dt\\&\longrightarrow 
\frac{1}{2}\int_I\psi K_\ep'(\eta)\eta\dt-\frac{1}{2}\int_I\psi K'(\eta)\eta\dt=\ep\int_I\psi \mathcal L(\eta)\dt
\end{align*}
as $N\rightarrow\infty$ due to \eqref{eq:1401bN}. Combing this with \eqref{eq:conetastrong} shows that
\eqref{eq:1401bNN} must be true.
Combining all the convergences proven above
allows us to pass to the limit in the energy balance \eqref{eq:2807} and to conclude that
\begin{align*} 
- \int_I \partial_t \psi \,
\mathcal E_{\ep,\delta} \dt &=
\psi(0) \mathcal E_{\ep,\delta}(0)+\int_I\psi \int_{\Omega_{\eta}}   \left(\frac{\delta}{\vartheta^2}-\ep \vt^5  \right) \dx \dt \\&+
\int_I \psi \int_{\Omega_{\eta}} \vr H  \dxt+\int_I\psi\int_{\Omega_{\eta}}\varrho\bff\cdot\bfu\dxt+\int_I\psi\int_\omega g\,\partial_t\eta\,\dd y\dt
\end{align*}
with
\begin{align*}
\mathcal E_{\ep,\delta}(t)&= \int_{\Omega_{\eta(t)}}\Big(\frac{1}{2} \varrho(t)| {\bf u}(t) |^2 + \varrho(t) e_\delta(\varrho(t),\vartheta(t))\Big)\dx+\int_\omega \frac{|\partial_t\eta(t)|^2}{2}\,\dd y+ K_\ep(\eta(t)).
\end{align*}

\subsection{End of the proof of Theorem \ref{thm:regu}}
We have shown that a subsequence can be chosen that inhibits the necessary compactness properties to satisfy \ref{K2} and \ref{K4}. The entropy inequality \ref{K3} follows by further convergence properties of $\vt$ and weak sequential lower semi-continuity of the various convex terms. Due to its local character, the limit passage can be obtained without further difficulty by applying the methodology developed in \cite[Chapter 3]{F}.
Finally, we may also pass to the limit with the momentum equation \eqref{eq:decuN'} and establish \ref{K1}. First observe that the necessary convergence of the approximate solutions has been shown in Subsection~\ref{subsec:teb}. Hence we are left to show the convergence of the test-functions. For that please observe that $\tilde\bfomega_k\circ \bfPsi_{\eta_N}^{-1}\to^{\eta} \tilde\bfomega_k\circ \bfPsi_{\eta}^{-1}$ with $N\to \infty$ (and $k\in \{1,...,N\}$) in various spaces including $L^q(I,W^{1,q}(\Omega_{\eta_N}))\cap W^{1,2}(I,L^a(\Omega_{\eta_N}))$ for $q<\infty$ and $a<4$. Indeed, the a-priori estimates on $\eta_N, \partial_t \eta_N, \nabla \eta_N$ transfer to $\bfPsi_{\eta_N}^{-1}$ by the respective calculations and estimates in Section~\ref{sec:2}. Finally, we observe that the linear hull of $\{(\omega_k,\tilde\bfomega_k\circ \bfPsi_{\eta}^{-1}\}_{k\in \N}$ exhibits the full space of test functions.
Hence we conclude that \ref{K1} is satisfied.  

\section{Construction of a solution.}
\label{sec:5}
In this section we pass to the limit in the approximate equations. For technical reasons
the limits $\ep\rightarrow0$ and $\delta\rightarrow0$ have to be performed independently from each other.  
For the greater part of this Section we study the limit $\varepsilon\rightarrow0$ in the approximate system \ref{K1}--\ref{K4} and only highlight the difference in the $\delta$-limit.

\subsection{The limit system for $\ep\to0$}
\label{subsec:del}
We wish to establish the existence of a weak solution $(\eta,\bfu,\vr,\vt)$ to the system with artificial pressure in the following sense: We define
$$
\widetilde W_\eta^I= C_w(\overline{I};L^\beta(\Omega_\eta))
$$
as the function space for the density, whereas the other function spaces are defined in Section \ref{sec:weak}.
A weak solution is a quadruplet $(\eta,\bfu,\varrho,\vt)\in Y^I\times X_{\eta}^I\times\widetilde W_\eta^I\times Z_\eta^I$ that satisfies the following.
\begin{enumerate}[label={(D\arabic{*})}]
\item\label{D1} The momentum equation holds in the sense that\begin{align}\label{eq:apu}
\begin{aligned}
&\int_I\frac{\dd}{\dt}\int_{\Omega_{ \eta}}\varrho\bfu \cdot\bfphi\dx-\int_{\Omega_{\eta}} \Big(\varrho\bfu\cdot \partial_t\bfphi +\varrho\bfu\otimes \bfu:\nabla \bfphi\Big)\dxt
\\
&+\int_I\int_{\Omega_\eta}\bfS(\vt,\nabla\bfu):\nabla\bfphi \dxt-\int_I\int_{\Omega_{ \eta }}
p_\delta(\vr,\vt)\,\Div\bfphi\dxt\\
&+\int_I\bigg(\frac{\dd}{\dt}\int_\omega \partial_t \eta b\dH-\int_\omega \partial_t\eta\,\partial_t b\dH + \int_\omega K'(\eta)\,b\dH\bigg)\dt
\\&=\int_I\int_{\Omega_{\eta}}\varrho\bff\cdot\bfphi\dxt+\int_I\int_\omega g\,b\,\dd x\dt
\end{aligned}
\end{align} 
for all $(b,\bfphi)\in C^\infty(\omega)\times C^\infty(\overline{I}\times \R^3)$ with $\mathrm{tr}_{\eta}\bfphi=b\nu$. Moreover, we have $(\varrho\bfu)(0)=\bfq_0$, $\eta(0)=\eta_0$ and $\partial_t\eta(0)=\eta_1$. The boundary condition $\mathrm{tr}_\eta\bfu=\partial_t\eta\nu$ holds in the sense of Lemma \ref{lem:2.28}.
\item\label{D2}  The continuity equation holds in the sense that
\begin{align}\label{eq:apvarrho0}
\begin{aligned}
&\int_I\frac{\dd}{\dt}\int_{\Omega_{\eta}}\varrho \psi\dxt-\int_I\int_{\Omega_{\eta}}\Big(\varrho\partial_t\psi
+\varrho\bfu\cdot\nabla\psi\Big)\dxt=0
\end{aligned}
\end{align}
for all $\psi\in C^\infty(\overline{I}\times\R^3)$ and we have $\varrho(0)=\varrho_0$. 
\item \label{D2'} The entropy balance
\begin{align} \label{m217*final}\begin{aligned}
 \int_I \frac{\dd}{\dt}&\int_{\Omega_{\eta}} \vr s(\varrho,\vartheta) \, \psi\dxt
 - \int_I \int_{\Omega_{\eta}} \big( \vr s(\varrho,\vartheta) \partial_t \psi + \varrho s (\varrho,\vartheta)\bfu \cdot \nabla\psi \big)\dxt
\\& \geq\int_I\int_{\Omega_{\eta}}
\frac{1}{\vartheta} \Big[ \bfS(\vt, \nabla \vu) : \nabla \bfu + \Big(
\frac{\varkappa(\vartheta)}{\vartheta} + \frac{\delta}{2} ( \vartheta^{\beta -
1} + \frac{1}{\vartheta^2}) \Big) |\nabla \vartheta|^2 + \delta
\frac{1}{{\vartheta}^2} \Big] 
\psi\,\dif x\,\dif t \\
&- \int_I\int_{\Omega_{\eta}} 
\Big( \frac{\varkappa(\vartheta)}{\vartheta} + \delta (
\vartheta^{\beta - 1} + \frac{1}{\vartheta^2}) \Big) \nabla
\vartheta  \cdot \nabla\psi \dxt+ \int_I \int_{\Omega_{\eta}} \frac{\vr}{\vt} H \psi  \dxt
\end{aligned}
\end{align}
holds for all
$\psi\in C^\infty(\overline I\times \R^3)$ with $\psi \geq 0$. Moreover, we have
$\lim_{r\rightarrow0}\vr s(\vr,\vt)(t)\geq \vr_0 s(\vr_0,\vt_0)$ and $\partial_{\nu_{\eta}}\vt|_{\partial\Omega_\eta}\leq 0$.
\item \label{D3} The total energy balance
\begin{equation} \label{EI20final1}
\begin{split}
- \int_I \partial_t \psi \,
\mathcal E_\delta \dt &=
\psi(0) \mathcal E_\delta(0) + \int_I \psi \int_{\mt} \frac{\delta}{\vartheta^2}  \dx  \dt+ \int_I \psi \int_{\Omega} \vr H  \dxt+\int_I\int_{\Omega_{\eta}}\varrho\bff\cdot\bfu\dxt\\&+\int_I\psi\int_\omega g\,\partial_t\eta\,\dd y\dt
\end{split}
\end{equation}
holds for any $\psi \in C^\infty_c([0, T))$.
Here, we abbreviated
$$\mathcal E_\delta(t)= \int_{\Omega_\eta(t)}\Big(\frac{1}{2} \varrho(t) | {\bf u}(t) |^2 + \varrho(t) e_\delta(\varrho(t),\vartheta(t))\Big)\dx+\int_\omega \frac{|\partial_t\eta(t)|^2}{2}\,\dd y+ K(\eta(t)).$$
\end{enumerate}
\begin{theorem}\label{thm:ap}
Assume that we have for some $\alpha\in(0,1)$ and $s>0$
\begin{align*}
\frac{|\bfq_0|^2}{\varrho_0}&\in L^1(\Omega_{\eta_0}),\ \varrho_0,\vt_0\in C^{2,\alpha}(\overline\Omega_{\eta_0}), \ \eta_0\in W^{3+s,2}(\omega),\ \eta_1\in L^2(\omega),\\
\bff&\in L^2(I;L^\infty(\R^3)),\ g\in L^2(I\times \omega),\ H\in C^{1,\alpha}(\overline I\times \R^3),\  H\geq0.
\end{align*}
Furthermore suppose that $\varrho_0$ and $\vt_0$ are strictly positive and that \eqref{eq:compa} is satisfied. There is a solution $(\eta,\bfu,\varrho,\vt)\in Y^I\times X_\eta^I\times \widetilde W_\eta^I\times Z_\eta^I$ to \ref{D1}--\ref{D3}. 
Here, we have $I=(0,T_*)$, where $T_*<T$ only if $\lim_{t\rightarrow T^\ast}\|\eta(t,\cdot)\|_{L^\infty_x}=\frac{L}{2}$ or the Koiter energy degenerates (namely, if $\lim_{s\to t}\overline{\gamma}(s,y)=0$ for some point $y\in \omega$). 
\end{theorem}

\begin{lemma}
\label{cor:ap1}
Under the assumptions of Theorem \ref{thm:ap} the continuity equation holds in the renormalized sense as specified in Definition~\ref{def:ren}.
%that is
%\begin{align}\label{eq:final3}
%\begin{aligned}
%\int_I\frac{\dd}{\dt}\int_{\Omega_\eta}\theta(\varrho)\psi\dxt&-\int_I\int_{\Omega_\eta}\Big(\theta(\varrho)\partial_t\psi +\theta(\varrho) \bfu\cdot \nabla \psi\Big)\dxt
%\\
%& =- \int_I\int_{\Omega_\eta}(\varrho\theta'(\varrho)-\theta(\varrho)) \diver\bfu \,\psi\dxt
%\end{aligned}
%\end{align}
%for all $\psi\in C^\infty(\overline I\times\R^3)$ and all $\theta\in C^1(\R)$ with 
%$\theta(0)=0$ and
%$\theta'(z)=0$ for $z\geq M_\theta$.
\end{lemma}

The proof of the above theorem and lemma will be split in several parts.
For a given $\varepsilon$ we obtain a solution $(\eta_\varepsilon,\bfu_\varepsilon,\varrho_\varepsilon)$ to \ref{K1}--\ref{K4} by Theorem \ref{thm:regu}.
As in the preceding Section we can combine the total energy balance
\eqref{EI20finala} with the entropy balance \eqref{m217*finala} to obtain the total dissipation balance
\begin{align} \label{NTDB}
\begin{aligned}
&
 \int_{\Omega_{\eta_\ep}} \Big[ \frac{1}{2} \varrho_\varepsilon | {\bf u_\varepsilon} |^2 + H_{\delta,\Theta}(\varrho_\varepsilon,\vartheta_\varepsilon)
  \Big] \dx+\int_\omega \frac{|\partial_t\eta_N|^2}{2}\,\dd y+ K(\eta_N)\\
&+\Theta \int_0^\tau\int_{\Omega_{\eta_\ep}}\sigma_{\varepsilon,\delta}\dx
+ \int_0^\tau\int_{\Omega_{\eta_\ep}} \ep \vartheta_\varepsilon^5   \dt\\
&\leq \,\int_{\Omega_{\eta_\ep}}\Big[ \frac{1}{2} \vr_0 |\vu_0|^2  + H_{\delta,\Theta}(\varrho_0,\vartheta_0) \Big] \dx+\int_\omega \frac{|\partial_t\eta_1|^2}{2}\,\dd y+ K(\eta_0)\\
&+\int_0^\tau\int_{\Omega_{\eta_\ep}}\bigg(\frac{\delta}{\vartheta_\varepsilon^2}+\varepsilon\Theta\vartheta^4_\varepsilon\bigg)\dx\dt
\end{aligned}
\end{align}
for any $0 \leq \tau \leq T$. Here $H_{\delta, \Theta}(\vr, \vt)= \vr \left( e_\delta(\vr, \vt) - \Theta s(\vr, \vt) \right)$ for some $\Theta>0$
and
\begin{align*}
\sigma_{\varepsilon,\delta}&=\frac{1}{\vartheta_\varepsilon}\Big[\bfS(\vartheta_\varepsilon,\nabla\bfu_\varepsilon):\nabla\bfu_\varepsilon+\ep(1+\vt_\ep)\max\{|\overline{\bfP}_\ep|^{p'},|\nabla\bfu_\ep|^p\}\Big]+\frac{\varepsilon\delta}{2\vartheta_\varepsilon}{\beta\varrho_\varepsilon^{\beta-2}} |\nabla\varrho_\varepsilon|^2\\
&+\frac{1}{\vartheta_\varepsilon}\Big[\frac{\varkappa(\vartheta_\varepsilon)}{\vartheta_\varepsilon}|\nabla\vartheta_\varepsilon|^2+\frac{\delta}{2}\Big(\vt_\varepsilon^{\beta-1}+\frac{1}{\vartheta_\varepsilon^2}\Big)|\nabla\vartheta_\varepsilon|^2+\delta\frac{1}{\vartheta_\varepsilon^2}\Big].
\end{align*}
Absorbing the final term on the left-hand side of \eqref{NTDB} into the left-hand side we deduce the bounds
\begin{align} \label{Nbv1}
   &\sup_{t \in I} \int_{\Omega_{\eta_\ep}} \Big[ \frac{1}{2} \varrho_\varepsilon | {\bf u_\varepsilon} |^2 + H_{\delta,\Theta}(\varrho_\varepsilon,\vartheta_\varepsilon) \Big] \dx   \leq c\\
 \label{Nbv1B}   &\sup_{I}\int_\omega \frac{|\partial_t\eta_\ep|^2}{2}\,\dd y+ \sup_{I}K(\eta_\ep)+\ep\sup_I\mathcal L(\eta_\ep)\leq\,c.
\end{align}
In particular, we have
\begin{equation} \label{Nbv2}
\sup_{t \in I} \| \vr_\varepsilon \|^\beta_{L^\beta(\Omega_{\eta_\ep})}  +
 \sup_{t \in I} \| \vr_\varepsilon \vu_\varepsilon \|_{L^{\frac{2 \beta}{\beta + 1}}(\Omega_{\eta_\ep})}^{\frac{2 \beta}{\beta + 1}}+
 \sup_{t \in I} \| \vt_\varepsilon \|^4_{L^4(\Omega_{\eta_\ep})}  \leq c.
\end{equation}
Moreover, boundedness of the entropy production rate
\begin{equation} \label{Nbv3}
 \| \sigma_{\ep, \delta} \|_{L^1(I \times \Omega_{\eta_\ep})}  \leq c
\end{equation}
gives rise to
\begin{align} 
\label{Nbv4A}
 \ep\| \nabla \vu_\varepsilon \|^p_{L^p(I\times \Omega_{\eta_\ep}) }  + \ep\| \overline{\bfP}_\ep\|^{p'}_{L^{p'}(I\times\Omega_{\eta_\ep}) }  \leq c,\\
 \label{Nbv4}
 \| \mathcal D(\vu_\varepsilon) \|^2_{L^2(I\times \Omega_{\eta_\ep}) } + \| \nabla \vt_\varepsilon^{\beta/2} \|^2_{L^2(I\times\Omega_{\eta_\ep}) }  + \| \nabla \vt_\varepsilon \|^2_{L^2(I\times\Omega_{\eta_\ep}) }  \leq c;
\end{align}
whence, by Poincare's inequality and (\ref{Nbv2}),
\begin{equation} \label{Nbv5}
\| \vu_\varepsilon \|^2_{L^2(I; W^{1,2}(\Omega_{\eta_\ep}) ) }  +  \| \vt_\varepsilon \|^2_{L^2(I; W^{1,2}(\Omega_{\eta_\ep})) } \leq c.
\end{equation}
Finally, we deduce from the equation of continuity (\ref{eq:regvarrho}) (using the renormalized formulation from Theorem~\ref{lem:warme} (b) with $\theta(z)=z^2$ and testing with $\psi\equiv 1$)) that
\begin{equation} \label{Nbv6}
\int_{\Omega_{\eta_\ep(t)}} \vr_\varepsilon(t, \cdot) \dx = \int_{\Omega_{\eta_\ep(0)}} \vr_0 \dx,\
 \| \sqrt{\ep} \nabla \vr_\varepsilon \|_{L^2(I\times\Omega_{\eta_\ep})}  \leq c.
\end{equation}
Note that all estimates are independent of $\ep$.
%In particular, we have
%the following uniform bounds
%\begin{align}
% \bfu_\varepsilon&\in L^2(I;W^{1,2}(\Omega_{\eta_\epsilon})),\label{apv}\\
%\sqrt{ \varrho_\varepsilon} \bfu_\varepsilon&\in L^\infty(I;L^2( \Omega_{\eta_\epsilon})),\label{aprhov}\\
% \varrho_\varepsilon&\in L^\infty(I;L^\beta( \Omega_{\eta_\epsilon})),\label{aprho}\\
%  \varrho_\varepsilon\bfu_\varepsilon&\in L^\infty(I;L^\frac{2\beta}{\beta+1}( \Omega_{\eta_\epsilon})),\label{estrhou2}\\
%\varrho_\varepsilon\bfu_\varepsilon\otimes\bfu_\varepsilon&\in L^2(I;L^\frac{6\beta}{4\beta+3}(\Omega_{\eta_\epsilon})),\label{est:rhobfu22}\\
%\eta_\varepsilon&\in  L^\infty(I,W^{2,2}_0(M)),\label{con:epseta}\\
%\partial_t\eta_\varepsilon&\in  L^\infty(I,L^{2}(M)).\label{con:epsetat}
%\end{align}
Hence, we may take a subsequence such that
%by Lemma~\ref{thm:weakstrong}, we find for $s\in (1,2)$ and
 for some $\alpha\in (0,1)$ we have
\begin{align}
\eta_\epsilon&\rightharpoonup^\ast\eta\quad\text{in}\quad L^\infty(I;W^{2,2}(\omega))\label{eq:conveta1},\\
\ep\eta_\epsilon&\rightarrow0\quad\text{in}\quad L^\infty(I;W^{3,2}(\omega)),\label{eq:conveta1ep}\\
\eta_\epsilon&\rightharpoonup^\ast\eta\quad\text{in}\quad W^{1,\infty}(I;L^2(\omega)),
\label{eq:conetat1}\\
%\varepsilon\eta_\epsilon&\rightharpoonup^\ast\eta\quad\text{in}\quad W^{1,2}(I;W^{3,2}(\omega)),
%\label{eq:conetat1}\\
\eta_\epsilon&\to\eta\quad\text{in}\quad C^\alpha(\overline I\times \omega)),
\label{eq:conetatal}\\
\mathcal D(\bfu_\epsilon)&\rightharpoonup^\eta\mathcal D(\bfu)\quad\text{in}\quad L^2(I;L^{2}(\Omega_{\eta_\epsilon})),\label{eq:convu1}\\
\bfu_\epsilon&\rightharpoonup^\eta\bfu\quad\text{in}\quad L^2(I;L^{2}(\Omega_{\eta_\epsilon})),\label{eq:convu1b}\\
\ep\bfu_\epsilon&\rightarrow^\eta0\quad\text{in}\quad L^p(I;W^{1,p}(\Omega_{\eta_\epsilon})),\label{eq:convu1p}\\
\ep\overline\bfP_\epsilon&\rightarrow^\eta0\quad\text{in}\quad L^{p'}(I;L^{p'}(\Omega_{\eta_\epsilon})),\label{eq:convP}\\
\varrho_\epsilon&\rightharpoonup^{\ast,\eta}\varrho\quad\text{in}\quad L^\infty(I;L^\beta(\Omega_{\eta_\epsilon})),\label{eq:convrho1}\\
\varepsilon\nabla\varrho_\varepsilon&\rightarrow^\eta 0\quad\text{in}\quad L^2(I\times\Omega_{\eta_\epsilon}),\label{eq:van1}\\
\vt_\ep&\rightharpoonup^{\ast,\eta}\vt\quad\text{in}\quad L^\infty(I;L^4(\Omega_{\eta_\ep})),\label{eq:convteps}\\
\vt_\ep&\weakto^{\eta}\vt\quad\text{in}\quad L^\beta(I;L^{3\beta}(\Omega_{\eta_\ep})),\label{eq:convvt'1eps}\\
\vt_\ep&\rightharpoonup^\eta\vt\quad\text{in}\quad L^2(I;W^{1,2}(\Omega_{\eta_\ep})).\label{eq:convt2eps}
\end{align}
%Now, using the a-priori estimates \eqref{Nbv2}--\eqref{Nbv4} and the bounds that one gains (using the renormalized continuity equation from Lemma~\ref{cor:ap} with $\theta(z)=z^2$ and testing with $\psi\equiv 1$) we find due to $\beta>4$ that
%\begin{align}\label{est:nablarho}
%\int_I\int_{\Omega_{\eta_\epsilon}}\epsilon\abs{\nabla\varrho_\varepsilon}^2\, dx\, dt\leq C,
%\end{align}
%with $C$ independent of $\epsilon$. This
%implies
%\begin{align}
%\varepsilon\nabla\varrho_\varepsilon\rightarrow^\eta 0\quad\text{in}\quad L^2(I\times\Omega_{\eta_\epsilon}).\label{eq:van1}
%%\varepsilon\nabla\bfu_\varepsilon\nabla\varrho_\varepsilon\rightarrow^\eta 0\quad\text{in}\quad L^1(I\times\Omega_{\eta_\epsilon}).\label{eq:van2}
%\end{align}
%Note that due to \eqref{eq:convvt'1eps} also the term $\ep\vt_\ep^5$ vanishes as $\ep\rightarrow0$. Finally, we have
%\begin{align}
%\ep^{1/q'}\overline\bfh_\ep\rightarrow^\eta 0\quad\text{in}\quad L^{q}(I\times\Omega_{\eta_\epsilon})\label{eq:van2}
%\end{align}
%by Theorem \ref{thm:visu}.\\
%\seb{REMOVE: Please observe that this is the only essential point in the proof where we make use of the fact that $\varrho_\epsilon$ is uniformly bounded in $L^\beta$. }
%Note, that \eqref{con:epseta} and \eqref{con:epsetat} imply
%\begin{align}\label{eq:convetah}
%\eta_\varepsilon\rightarrow\eta\quad\text{in}\quad C(\overline{I}\times \partial\Omega).
%\end{align}
We observe that the a-priori estimates \eqref{Nbv2} imply uniform bounds of
$\varrho_\epsilon\bfu_\epsilon$ in $ L^\infty(I,L^\frac{2\beta}{\beta+1})$. Therefore, we may apply Lemma~\ref{thm:weakstrong} with the choice $v_i\equiv \bfu_\epsilon$, $r_i=\varrho_\epsilon$, $p=s=2$, $b=\beta$ and $m$ sufficiently large to obtain
\begin{align}
\varrho_\varepsilon\bu_\varepsilon&\rightharpoonup^\eta  {\varrho}  {\bfu}\qquad\text{in}\qquad L^q(I, L^a(\Omega_{\eta_\epsilon})),\label{conv:rhov2}
\end{align}
where $a\in (1,\frac{2\beta}{\beta+1})$ and $q\in (1,2)$.
We apply Lemma~\ref{thm:weakstrong} once more with the choice $v_i\equiv \bfu_\epsilon$, $r_i=\varrho_\epsilon\bfu_\epsilon$,  $p=s=2$, $b=\frac{2\beta}{\beta+1}$ and $m$ sufficiently large to find that
\begin{align}
{\varrho}_\varepsilon  {\bfu}_\varepsilon\otimes  {\bfu}_\varepsilon&\rightharpoonup^\eta  {\varrho}  {\bfu}\otimes  {\bfu}\qquad\text{in}\qquad L^1(I\times\Omega_{\eta_\epsilon}).\label{conv:rhovv2}
\end{align}
We also obtain
\begin{align}
\varrho_\varepsilon\bu_\varepsilon&\rightharpoonup^\eta  {\varrho}  {\bfu}\qquad\text{in}\qquad L^q(I, L^q(\Omega_{\eta_\epsilon}))\label{conv:rhov2B},\\
\varrho_\varepsilon\bu_\varepsilon&\rightharpoonup^{\eta,*}  {\varrho}  {\bfu}\qquad\text{in}\qquad L^\infty(I, L^{\frac{2\beta}{\beta+1}}(\Omega_{\eta_\epsilon}))\label{conv:rhov2BB},
\end{align}
for all $q<\frac{6\beta}{\beta+6}$. Moreover, we have as a consequence of \eqref{eq:convu1} and \eqref{eq:convteps}
\begin{align}
\bfS(\vt_\ep,\nabla\bfu_\ep)&\rightharpoonup^\eta  \overline{\bfS}\qquad\text{in}\qquad L^{4/3}(I, L^{4/3}(\Omega_{\eta_\ep}))\label{conv:Sep}
\end{align}
for some limit function $\overline{\bfS}$. The convergence
\eqref{eq:conveta1} and the assumption on $K$ (see Section \ref{subsec:model}) yields
\begin{align}\label{conv:K'ep}
K'(\eta_\ep)\rightharpoonup^* \overline K'\quad \text{in}\quad L^\infty(I;W^{-2,r}(\omega))
\end{align}
for any $r<2$ with some limit quantity $\overline K$.\\
%Finally we combine \eqref{Nbv1B} and \eqref{Nbv4} (together with the Lemma \ref{lem:2.28}) to conclude
%\begin{align}\label{eq:dtd3teta}
%\varepsilon\partial_t\nabla^3\eta\rightarrow0\quad\text{in}\quad L^q(I\times\omega)
%\end{align}
%for some $q>1$.\\
%\subsection{Equi-integrability of the pressure}
At this stage of the proof the pressure is only bounded in $L^1$, so we have to exclude its concentrations. The nowadays common approach from \cite[Chapter 3, Section 3.6.3]{F} only works locally where the moving shell is not seen (see Lemma \ref{prop:higher} below).
The problem can be circumvented by excluding concentrations at the boundary (see Lemma \ref{prop:higherb} which is inspired by \cite{Kuk}). The proof is exactly as in \cite[Lemma 6.4]{BrSc}.
\begin{lemma}\label{prop:higher}
Let $Q=J\times B\Subset I\times\Omega_{\eta}$ be a parabolic cube. The following holds for any $\epsilon\leq\epsilon_0(Q)$
\begin{equation}\label{eq:gamma+1}
\int_{Q}p_\delta(\vr_\ep,\vt_\ep)\vr_\ep\,\dif x\,\dif t\leq C(Q)
\end{equation}
with a constant independent of $\epsilon$.
\end{lemma}
\begin{lemma}\label{prop:higherb}
Let $\kappa>0$ be arbitrary. There is a measurable set $A_\kappa\Subset I\times\Omega_\eta$ such that we have for all $\varepsilon\leq \varepsilon_0(\kappa)$
\begin{equation}\label{eq:gamma+1b}
\int_{I\times\R^3\setminus A_\kappa}p_\delta(\vr_\ep,\vt_\ep)\vr_\ep\chi_{\Omega_{\eta_\epsilon}}\,\dif x\,\dif t\leq \kappa.
\end{equation}
\end{lemma}

We connect Lemma~\ref{prop:higher} and Lemma~\ref{prop:higherb} to obtain the following corollary.
\begin{corollary}\label{prop:higherb0}Under the assumptions of Theorem~\ref{thm:ap} there exists a function $\overline p$ such that
\[
p_\delta(\vr_\ep,\vt_\ep)\weakto^\eta \overline p\text{ in }L^1(I;L^1(\Omega_{\eta_\epsilon})),
\]
at least for a subsequence.
Additionally, for $\kappa>0$  arbitrary, there is a measurable set $A_\kappa\Subset I\times\Omega_\eta$ such that $\overline{p}\varrho\in L^1(A_\kappa)$ and
\begin{equation}\label{eq:gamma+1b0}
\int_{(I\times\Omega_{\eta})\setminus A_\kappa}\overline p\dxt\leq \kappa.
\end{equation}
\end{corollary}
Combining Corollary \ref{prop:higherb0} with the convergences \eqref{eq:conveta1}--\eqref{conv:K'ep} we can pass to the limit in \eqref{eq:regu} and \eqref{eq:regvarrho} and obtain the following. There is
$$(\eta,\bfu,\varrho,\vt,\overline p)\in Y^I\times X_{\eta}^I\times\widetilde W_\eta^I \times Z^I_\eta \times L^1(I\times \Omega_\eta)$$ that satisfies
\[
\bfu(\cdot,\cdot+\eta\nu)=\partial_t\eta\nu_{\eta}
 \quad\text{in}\quad I\times \omega,
\]
 the continuity equation 
\begin{align}\label{eq:apvarrho}
\begin{aligned}
&\int_I\frac{\dd}{\dt}\int_{\Omega_{\eta}}\varrho \psi\dxt-\int_I\int_{\Omega_{\eta}}\Big(\varrho\partial_t\psi
+\varrho\bfu\cdot\nabla\psi\Big)\dxt=0
\end{aligned}
\end{align}
for all $\psi\in C^\infty(\overline I \times \R^3)$ and the coupled weak momentum equation
\begin{align}\label{eq:apulim}
\begin{aligned}
&\int_I\frac{\dd}{\dt}\int_{\Omega_{ \eta}}\varrho\bfu\cdot \bfphi\dxt-\int_I\int_{\Omega_{\eta}} \Big(\varrho\bfu\cdot \partial_t\bfphi +\varrho\bfu\otimes \bfu:\nabla \bfphi\Big)\dxt\\
&+\int_{\Omega_{\eta}}\overline\bfS:\nabla\bfphi\dxt-\int_I\int_{\Omega_{ \eta }}
\overline{p}\,\Div\bfphi\dxt
\\
&
+\int_I\frac{\dd}{\dt}\int_\omega \partial_t \eta b\dH-\int_\omega \partial_t\eta\,\partial_t b\dH + \int_\omega \overline K'\,b\dH\dt
\\&=\int_I\int_{\Omega_{ \eta}}\varrho\bff\cdot\bfphi\dxt+\int_I\int_\omega g\,b\,\dd x\dt
\end{aligned}
\end{align} 
 for all $(b,\bfphi)\in C^\infty(\omega)\times C^\infty(\overline{I}\times \setR^3)$ with $\mathrm{tr}_{\eta}\bfphi=b\nu$.
It remains to show strong convergence of $\vt_\ep$, $\vr_\ep$ and $\nabla^2\eta_\ep$. The convergence proof for $\vt_\ep$ is entirely based on local arguments. Consequently the shell is not seen and we can follow the arguments in 
\cite[Chapter 3, Section 3.7.3]{F} to conclude
\begin{align}\label{m257a}
\vt_\ep \to^\eta \vt \ \mbox{in} \ L^4(I\times\Omega_{\eta_\ep}).
\end{align}
This yields $\overline\bfS=\bfS(\vartheta,\nabla \vu)$ in \eqref{eq:apulim}. Additionally we can pass to the limit in the entropy balance \eqref{m217*finala} using lower semi-continuity. The remainder of this subsection is dedicated to the proof of $\overline{p}=p(\vr,\vt)$. Eventually, we will pass to the limit in the shell energy in Section \ref{sec:shelleps} which will finish the proof of Theorem \ref{thm:ap}. \\
The proof of strong convergence of the density is based on the effective viscous flux identity introduced in \cite{Li2} and the concept of renormalized solutions from \cite{DL}. Arguing locally, there is no difference to the known setting and we can follow the arguments in \cite[Chapter 3, Section 3.6.5]{F}. We consider a parabolic cube $\tilde Q=\tilde J\times\tilde B$ with
$Q\Subset \tilde Q\Subset  I\times\Omega_\eta$. Due to \eqref{eq:conetatal} we can assume
that $\tilde Q\Subset  I\times\Omega_{\eta_\varepsilon}^I$ (by taking $\varepsilon$ small enough). 
For $\psi\in C^\infty_c(\tilde Q)$ we obtain
\begin{align}\label{eq:fluxpsi}
\begin{aligned}
\int_{I\times\mt} &\psi^2\big(p_\delta(\vr_\ep,\vt_\ep)-(\lambda(\vt_\ep)+2\mu(\vt_\ep))\Div \bfu_\varepsilon\big)\,\varrho_\varepsilon\dxt\\&\longrightarrow\int_{I\times\mt}\psi^2 \big( \overline{p}-(\lambda(\vt)+2\mu(\vt))\Div \bfu\big)\,\varrho\dxt
\end{aligned}
\end{align}
as $\varepsilon\rightarrow0$ (note that the term related to $\overline{\bfP}_\ep$ disappears due to \eqref{eq:convu1p} provided we choose $\beta$ large enough).
The proof of Lemma \ref{cor:ap1} follows exactly as in \cite[Lemma 6.2]{BrSc}. 
So, for $\psi\in C^\infty(\overline{I}\times \R^3)$ we have
\begin{align}\label{8.4}
\begin{aligned}
\int_{I}\frac{\dd}{\dt}&\int_{\R^3} \theta(\varrho)\,\psi\dxt-\int_{I\times\R^3}\theta(\varrho)\,\partial_t\psi\dxt
+\int_{I\times\R^3}\big(\varrho\theta'(\varrho)-\theta(\varrho)\big)\Div\mathscr E_\eta \bfu\,\psi\dxt\\
&=\int_{I\times\R^3}\theta(\varrho) \mathscr E_\eta \bfu\cdot\nabla\psi .
\end{aligned}
\end{align}
Here $\mathscr E_{\eta}:W^{1,2}(\Omega_\eta)\rightarrow W^{1,p}(\R^3)$ is the extension operator from \cite[Lemma 2.5]{BrSc}
where $1<p<2$ (but may be chosen close to 2).
In order to deal with the local nature of \eqref{eq:fluxpsi} we use ideas from
\cite{Fe}. First of all, by the monotonicity of the mapping $\vr\mapsto p(\vr,\vt)$, we find for arbitrary non-negative $\psi\in C^\infty_c(\tilde Q)$
\begin{align*}%\label{eq:fluxpsib}
&\liminf_{\varepsilon\rightarrow0}\int_{I\times\mt}\psi (\lambda(\vt)+2\mu(\vt))\big(\Div \bfu_\epsilon\,\varrho_\epsilon -\Div \bfu\,\varrho\big)\dxt\\
=&\liminf_{\varepsilon\rightarrow0}\int_{I\times\mt}\psi\Big( (\lambda(\vt_\ep)+2\mu(\vt_\ep))\Div \bfu_\epsilon\,\varrho_\epsilon -(\lambda(\vt)+2\mu(\vt))\Div \bfu\,\varrho\Big)\dxt\\
=&\liminf_{\varepsilon\rightarrow0}\int_{I\times\Omega_{\eta_\epsilon}}\psi\Big(\big( \overline{p}-(\lambda(\vt)+2\mu(\vt))\Div \bfu\big)\varrho
-\big( p(\vr_\ep,\vt_\ep)-(\lambda(\vt_\ep)+2\mu(\vt_\ep))\Div \bfu_\varepsilon\big)\,\varrho_\varepsilon\Big)\dxt
\\
+&\liminf_{\varepsilon\rightarrow0}\int_{I\times\Omega_{\eta_\epsilon}}\psi\big( p(\vr_\ep,\vt_\ep)\vr_\ep- \overline{p}\varrho\big)\dxt
\\
=&\liminf_{\varepsilon\rightarrow0}\int_{I\times\Omega_{\eta_\epsilon}}\psi \big(p(\vr_\ep,\vt_\ep)- \overline{p}\big)\big(\varrho_\varepsilon-\varrho\big)\dxt\geq 0
\end{align*}
using \eqref{eq:fluxpsi} as well as \eqref{m257a} (together with \eqref{m105} and the uniform bounds \eqref{Nbv2} and \eqref{Nbv4}).
As $\psi$ is arbitrary and $\mu$ strictly positive by \eqref{m105} we conclude
\begin{align}\label{8.12}
\overline{\Div \bfu\,\varrho}\geq \Div \bfu\,\varrho \quad\text{a.e. in }\quad I\times\Omega_\eta,
\end{align}
where 
\begin{align*}
\Div \bfu_\epsilon\,\varrho_\epsilon\weakto^\eta \overline{\Div \bfu\,\varrho}\quad\text{in}\quad L^1(\Omega;L^1(\Omega_{\eta_\varepsilon})),
\end{align*}
recall \eqref{eq:convu1} and \eqref{eq:convrho1}. Now, we compute both sides of
\eqref{8.12} by means of the corresponding continuity equations. Due to Theorem
\ref{lem:warme} (b) with $\theta(z)=z\ln z$ and $\psi=\mathbb{I}_{(0,t)}$ we have
\begin{align}\label{8.15}
\int_0^t\int_{\R^3}\Div \bfu_\epsilon\,\varrho_\epsilon\dxs\leq\int_{\R^3}\varrho_0\ln(\varrho_0)\dx
-\int_{\R^3}\varrho_\varepsilon(t)\ln(\varrho_\varepsilon(t))\dx.
\end{align}
Similarly, equation \eqref{8.4} yields
\begin{align}\label{8.14}
\int_0^t\int_{\R^3}\Div \bfu\,\varrho\dxs=\int_{\R^3}\varrho_0\ln(\varrho_0)\dx
-\int_{\R^3}\varrho(t)\ln(\varrho(t))\dx.
\end{align}
Combining \eqref{8.12}--\eqref{8.14} shows
\begin{align*}
\limsup_{\varepsilon\rightarrow0}\int_{\R^3}\varrho_\varepsilon(t)\ln(\varrho_\varepsilon(t))\dx\leq \int_{\R^3}\varrho(t)\ln(\varrho(t))\dx
\end{align*}
for any $t\in I$.
This gives the claimed convergence $\varrho_\varepsilon\rightarrow\varrho$ in $L^1(I\times\R^3)$ by convexity of $z\mapsto z\ln z$. Consequently, we have $\overline p=p(\vr,\vt)$.

 \subsection{Compactness of the shell energy}
 \label{sec:shelleps}
 All the forthcoming effort is to prove
\begin{align}\label{eq:1401a}
\lim_{\ep\rightarrow 0}\int_I\int_\omega|\partial_t\eta_\ep(t)|^2\,\dd y\dt&=\int_I\int_\omega |\partial_t\eta(t)|^2\,\dd y\dt,\\
\label{eq:1401b}\lim_{\ep\rightarrow 0}\int_I K_\ep(\eta_\ep(t))\dt&=\int_I  K(\eta(t))\dt,
\end{align}
as $\ep\rightarrow0$ at least for a subsequence. This will allow us to pass to the limit in the energy balance as well as in the nonlinear term of the shell equation.
In the following we derive a framework
to prove \eqref{eq:1401b} based on fractional estimates. The same approach will be subsequently used
in the limit passage $\delta\rightarrow0$ in Section \ref{sec:6}. The difference is that the bounds on the density will be more restrictive.
%Since the limit procedure of the energy on the $\delta$-level will follow the same lines as on the $\epsilon$-level, 
We develop the theory here using only these weaker estimates to have it ready for the final limit procedure as well.\\
%However, our situation becomes more involved as we loose the regularising effect of the $\kappa$-layer. A major difficulty is that the boundary of the time dependent domain $\mathscr R_{1/n}\eta_n$ is less regular. In particular, the a priori estimates from the last subsection are not even enough to obtain Lipschitz regularity. Fortunately, a careful regularity analysis provides this information nevertheless.  
%In particular, we cannot work with the extension $\mathscr F_\eta$ any more sine
%the pressure only belongs to $L^1$ (this will become even more apparent in the proceeding limits) whereas (the extension of) $\eta_n$ fails being bounded in $W^{1,\infty}$.
%We overcome this problem by using the divergence-free extension $\mathscr F_{\eta}^{\Div}$ from Theorem \ref{thm:MuScprop3.3}.
A first observation is that $\tr_{\eta_\ep}(\bfu_\ep)=\partial_t\eta_\ep\nu$ implies
\begin{align}
\partial_t\eta_\ep&\rightharpoonup\partial_t\eta\quad\text{in}\quad L^2(I;W^{1-1/r,r}(\omega)),\label{eq:conetatA2}
\end{align}
for all $r<2$ by \eqref{Nbv4} in combination with Lemma \ref{lem:2.28}.
In the following we are going to prove that
\begin{align}\label{eq:0806}
\int_I\|\eta_n\|^2_{W^{2+s,2}(\omega)}\dt
\end{align}
is uniformly bounded for some $s>0$ using an appropriate test-function in the shell equation. On account of the coupling we need a suitable test-function for the momentum equation for it as well.
Hence we set
\begin{align*}
({\bfphi}_\ep,\phi_\ep)=\big(\Test^{\mathrm div}_{\eta_\ep}(\Delta_{-h}^s\Delta_h^s \eta_\ep-\mathscr K_{\eta_\ep}(\Delta_{-h}^s\Delta_h^s \eta_\ep)),\Delta_{-h}^s\Delta_h^s \eta_\ep-\mathscr K_{\eta_\ep}(\Delta_{-h}^s\Delta_h^s \eta_\ep)\big),
\end{align*}
where $\Test^{\mathrm div}_{\eta}$ and $\mathscr K_{\eta}$ have been introduced in Proposition \ref{prop:musc}. Here $\Delta_s^hv(y)=h^{-s}(v(y+he_\alpha)-v(y))$ is the fractional difference quotient in direction $e_\alpha$ for $\alpha\in\{1,2\}$.  We obtain 
\begin{align*}
\int_I & K_\ep'(\eta_\ep)\, \phi_\ep\dt\\
&= \int_I  \int_{\Omega_{\eta_\ep(t)}}\big(  \vr_\ep\bu_\ep \otimes \bu_\ep
-\mathbf S(\vt_\ep,\bu_\ep)-\ep\overline\bfP_\ep\big):\nabla {\bfphi}_\ep
+\bff\cdot{\bfphi}_\ep \big) \dx\dt
\\
&+\int_I  \int_{\Omega_{\eta^\varrho(t)}} \vr_\ep\bu_\ep \cdot \partial_t  {\bfphi}_\ep\dxt-\int_I  \frac{\mathrm{d}}{\dt}\bigg(\int_{\Omega_{\eta_\ep(t)}}\vr_\ep\vu_\ep  \cdot {\bfphi}_n\dx
+\int_{\omega} \partial_t \eta_\ep \, \phi_\ep \dH\bigg)\dt\\
&+
\int_I \int_{\omega} \big(\partial_t \eta_\ep\, \partial_t\phi_\ep +g\, \phi_\ep \big)\dH\dt
=:(I)_\ep+(II)_\ep+(III)_\ep+(IV)_\ep
\end{align*}
recalling that the function $\bfphi_\ep$ is divergence-free such that the pressure term disappears.
Since $\eta_\ep\in L^\infty(I,W^{2,2}(\omega))$ uniformly, we have
\begin{align*}
\int_I&\|\Delta_h^s \nabla^2\eta_\ep\|_{L^2(\omega)}^2\dt\lesssim 1+ \int_I K'(\eta_\ep)\, \phi_\ep\dt
\end{align*} 
for every $h>0$ and $s\in (0,\frac{1}{2})$ due to \cite[Lemma 4.5]{MuSc}. Consequently,
it holds
\begin{align*}
\int_I\|\Delta_h^s \nabla^2\eta_\ep\|_{L^2(\omega)}^2\dt+\ep\int_I\|\Delta_h^s \nabla^3\eta_\ep\|_{L^2(\omega)}^2\dt\lesssim 1+(I)_\ep+(II)_\ep+(III)_\ep+(IV)_\ep
\end{align*}
% The last term is controlled by the assumptions on $\eta_0$ such that
and our task consists in establishing uniform estimates for the 
terms $(I)_\ep,\dots,(IV)_\ep$.
As far as $(I)_n$ is concerned the most critical term is the convective term $\vr_\ep\bu_\ep \otimes \bu_\ep$ with integrability $\frac{6\gamma}{\gamma+6}>1$. By \eqref{musc1} and \eqref{eq:conveta1} 
\begin{align}\label{eq:2907}
\begin{aligned}
\|\bfphi_\ep\|_{L^q(I;W^{1,p}(\omega))}&\leq\,\|\Delta_{-h}^s\Delta_h^s \eta_\ep\|_{L^q(I;W^{1,p}(\omega))}+\|\Delta_{-h}^s\Delta_h^s \eta_\ep\nabla \eta_\ep\|_{L^q(I;L^{p}(\omega))}\\
&\leq\,\|\eta_\ep\|_{L^q(I;W^{1+2s,p}(\omega))}+\|\Delta_{-h}^s\Delta_h^s \eta_\ep\|_{L^\infty(I;L^{2p}(\omega))}\|\nabla\eta_\ep\|_{L^\infty(I;L^{2p}(\omega))}\\
&\leq\,\|\eta_\ep\|_{L^q(I;W^{1+2s,p}(\omega))}+\|\eta_\ep\|_{L^\infty(I;W^{2s,2p}(\omega))}\|\nabla \eta_\ep\|_{L^\infty(I;L^{2p}(\omega))}\\
&\leq\,\|\eta_\ep\|_{L^q(I;W^{1+2s,p}(\omega))}+c_p
\end{aligned}
\end{align}
for all $s<\frac{1}{2}$, $p<\infty$ and $q\in[1,\infty]$. For $p=\frac{6\gamma}{6\gamma-\gamma-6}$ we can choose $s>0$ small enough such that $W^{2,2}(\omega)\hookrightarrow W^{1+2s,p}(\omega)$. Using \eqref{eq:conveta1} again implies that $\bfphi_\ep$ is
uniformly bounded in $L^\infty_t(W^{1,p}_x)$.
We conclude that 
$(I)_\ep$ is uniformly bounded in $\ep$ and $h$ if we choose $s$ small enough. The most critical term is in fact $(II)_\ep$. We note that \eqref{eq:convu1} and \eqref{eq:convrho1} imply
\begin{align*}
\vr_\ep\bfu_\ep\in L^2(I;L^{q_3}(\Omega_{\eta_n}))
\end{align*}
uniformly for all $q_3<\frac{6\gamma}{\gamma+6}$.
Due to the assumption $\gamma>\frac{12}{7}$ we can choose in the above $q_3>\frac{4}{3}$.
On the other hand we have
\begin{align*}
\|\partial_t\bfphi_\ep\|_{L^2(I;L^{q_3'}( S_{L/2} \cup \Omega))}&\lesssim \|\partial_t\Delta_{-h}^s\Delta_h^s \eta_n\|_{L^2(I;L^{q_3'}(\omega))}+\|\Delta_{-h}^s\Delta_h^s \eta_n\partial_t \eta_\ep\|_{L^2(I;L^{q_3'}(\omega))}\\
&\lesssim \|\partial_t\eta_\ep\|_{L^2(I;W^{2s,q_3'}(\omega))}+\|\Delta_{-h}^s\Delta_h^s \eta_\ep\|_{L^\infty(I\times\omega)}\|\partial_t \eta_\ep\|_{L^2(I;L^{q_3'}(\omega))}.
\end{align*}
Thus, we can choose $s$ small enough such that $\partial_t\bfphi_\ep$ is uniformly bounded in $L^2_t(L_x^{q_3'})$ thanks to \eqref{eq:conveta1} and \eqref{eq:conetatA2} (together with Sobolev's embedding and $q_3'<4$). We conclude boundedness of $(II)_\ep$. As far as $(III)_\ep$ is concerned, uniform bounds for the first term are easily obtained
 from \eqref{eq:2907} (choosing $p>\seb{\beta}>2$ and using Sobolev's embedding) in combination with
 \eqref{conv:rhov2}. For the second term we use
 \begin{align*}
 \|\phi_\ep\|_{L^2(I;L^2(\omega))}\lesssim  \|\eta_\ep\|_{L^2(I;W^{2s,2}(\omega)}\lesssim  \|\eta_\ep\|_{L^2(I;W^{1,2}(\omega))}
 \end{align*}
 together with \eqref{eq:conveta1}. The second term in $(IV)_\ep$ is analogous. Finally, we can use again \eqref{eq:conetatA2} to control the first term in $(IV)_\ep$ and the proof of \eqref{eq:1401b} is complete. Moreover we have shown
 \begin{align*}
\ep\int_I\|\eta_n\|^2_{W^{3+s,2}(\omega)}\dt\leq\,c
\end{align*}
uniformly in $\ep$. This, interpolated with \eqref{eq:0806}, yields
$\ep\mathcal L(\eta_\ep)\rightarrow0$
as $\ep\rightarrow0$, which completes the proof of \eqref{eq:1401b}.
Finally we observe that the convergence in \eqref{eq:1401a} follows exactly as was done in Subsection~\ref{ssec:comp} (in particular, \eqref{eq:0806} implies the required to obtain a counterpart of \eqref{3112b}). The proof is even slightly simpler since we do not need to project into a discrete space when proving the equi-continuity. 

\subsection{Proof of Theorem~\ref{thm:ap}}
We have collected all convergences that are necessary to pass to the limit with all involved terms. In particular, we may pass to the limit to obtain \ref{D2}, \ref{D3} and (by the methods established for the cylindrical domains, see \cite[Chapter 3]{F}) with \ref{D2'}. 
In order to pass to the limit with the weak momentum equation, fix a pair of smooth test-functions for the limit geometry $(b,\bfphi)\in C^\infty(\omega)\times C^\infty(\overline{I}\times\R^3)$ with $\mathrm{tr}_{\eta}\bfphi=b\nu$. Now since $\nabla \bfPsi_{\eta_\ep},\nabla \bfPsi_{\eta_\ep}^{-1}$ are strongly convergent in $L^\infty(I;L^q(\Omega_{\eta_\ep}))\cap L^2(I;L^\infty(\Omega_{\eta_\ep}))$ for all $q<\infty$ and $\partial_t\bfPsi_{\eta_\ep},\partial_t\bfPsi_{\eta_\ep}^{-1}$ are strongly convergent in $L^2(I;L^a(\Omega_{\eta_\ep}))$ for all $a<4$ we find that $(b,\bfphi\circ \bfPsi_\eta \circ \bfPsi_{\eta_\ep}^{-1})$ is an admissible test function for the approximate weak momentum equation with respective convergence properties. Hence we may pass to the limit with the approximate momentum equation and optain \ref{D1}.

\subsection{Proof of Theorem~\ref{thm:MAIN}.}

\label{sec:6}
In this section we are ready to prove the main result of this paper by passing to the limit $\delta\rightarrow0$ in the system \ref{D1}--\ref{D3} from Section \ref{subsec:del}.
Large parts of
the proof are very similar to their counterparts in the limit $\ep\to 0$. In particular, the compactness arguments from \ref{sec:shelleps} and \ref{ssec:comp} have been written in such a way that they are directly adaptable for the final layer here (using only the more restrictive bounds on $\gamma$). The main exception is the analysis related to the limit passage in the molecular pressure. This can, however, be adapted from \cite[Section 7]{BrSc}. As there, we can localise the argument for fixed boundaries from \cite{F}.
Consequently, parts of the argument are independent from the variable domain and the fluid-structure interaction. Nevertheless we sketch the main steps of the proof for the convenience of the reader.\\
Given initial data $(\bfq_0,\vr_0,\vt_0)$ and $H$ belonging to the function spaces stated in Theorem \ref{thm:MAIN}
it is standard to find regularized versions $(\bfq_{0,\delta},\vr_{0,\delta},\vt_{0,\delta})$ and $H_\delta$ such that for all $\delta>0$
\begin{align*}
 \vr_{0,\delta},\vt_{0,\delta}\in C^{2,\alpha}(\overline\Omega_{\eta_0}),\ \vr_{0,\delta},\vt_{0,\delta}\ \text{strictly positive},\  H_\delta\in C^{1,\alpha}(\overline I\times\R^3),\ H_\delta\geq0,
 \end{align*}
 as well as
 \begin{align*}
 \int_{\Omega_{\eta_0}}\Big(\frac{1}{2} \frac{| {\bf q}_{0,\delta} |^2}{\vr_{0,\delta}} &+ \varrho_{0,\delta} e(\varrho_{0,\delta},\vartheta_{0,\delta})\Big)\dx\rightarrow  \int_{\Omega_{\eta_0}}\Big(\frac{1}{2} \frac{| {\bf q}_{0} |^2}{\vr_{0,\delta}} + \varrho_{0} e(\varrho_{0},\vartheta_{0})\Big)\dx,\\
 \ H_\delta&\rightarrow H\ \text{in}\ L^\infty(\overline I\times\R^3),
 \end{align*}
 as $\delta\rightarrow0$.
For a given $\delta$ we gain a weak solution $(\eta_\delta,\bfu_\delta,\varrho_\delta,\vt_\delta)$ to \eqref{eq:apu}--\eqref{eq:apvarrho0} with this data by Theorem \ref{thm:ap}. It is defined in the interval $(0,T_*)$, where $T_*$ is restricted by the data only. The counterpart of the total dissipation balance from (\ref{NTDB}), that can be derived exactly as in Section \ref{subsec:del}, provides the following
uniform bounds:
\begin{equation} \label{wWS4eta}
 \sup_{t \in I} \| \partial_t{\eta_\delta} \|^{2}_{L^{2}(\omega)}  +
  \sup_{t \in I}  \|  {\eta_\delta} \|^2_{W^{2,2}(\omega)}  
\leq c,
\end{equation}
\begin{equation} \label{wWS47}
 \sup_{t \in I} \| \vrd \|^{\gamma}_{L^{\gamma}(\Omega_{\eta_\delta})}  +
  \sup_{t \in I} \delta \|  \vrd \|^\beta_{L^\beta(\Omega_{\eta_\delta})}  
\leq c,
\end{equation}
\begin{equation} \label{wWS48}
\begin{split}
   \sup_{t \in I} \left\| \vrd |\vud|^2 \right\|_{L^1(\Omega_{\eta_\delta})} +
 \sup_{t \in I} \left\| \vrd \vud \right\|^\frac{2\gamma}{\gamma+1}_{L^{\frac{2\gamma}{\gamma+1}}(\Omega_{\eta_\delta})}   \leq c,
\end{split}
\end{equation}
\begin{equation} \label{wWS49}
 \left\| \vud \right\|^{2}_{L^2(I\times\Omega_{\eta_\delta})}  +\left\| \mathcal D(\vud) \right\|^{2}_{L^2(I\times\Omega_{\eta_\delta})}  \leq c,
\end{equation}
\begin{equation} \label{wWS217aSS}
 \sup_{t \in I} \| \vtd \|^4_{L^4(\Omega_{\eta_\delta})} + \left\| \nabla \vtd \right\|^2_{L^2(I\times \Omega_{\eta_\delta})}  \leq c,
\end{equation}
\begin{equation} \label{wWS217SS}
 \left\| \frac{\varkappa_\delta(\vartheta_\delta)}{\vartheta_\delta}\nabla\vartheta_\delta \right\|^{2}_{L^2(I\times \Omega_{\eta_\delta})}  \leq c.
\end{equation}
\\
Finally, we report the conservation of mass principle
\begin{equation} \label{wWS411}
\| \vrd(\tau, \cdot) \|_{L^1(\Omega_{\eta_\delta})} = \int_{\Omega_{\eta_\delta}} \vr(\tau, \cdot) \dx = \int_{\Omega} \vr_0 \dx  \quad \mbox{for all}\ \tau\in[0,T].
\end{equation}
%The estimate from Theorem \ref{thm:ap} holds uniformly with respect to $\delta$. 
Hence we may take a subsequence, such that
%by Lemma~\ref{thm:weakstrong}, we find 
%for $s\in (1,2)$ and 
for some $\alpha\in (0,1)$ we have
\begin{align}
\eta_\delta&\rightharpoonup^\ast\eta\quad\text{in}\quad L^\infty(I;W^{2,2}(\omega))\label{eq:conveta}\\
\eta_\delta&\rightharpoonup^\ast\eta\quad\text{in}\quad W^{1,\infty}(I;L^2(\omega)),
\label{eq:conetat}\\
\eta_\delta&\to\eta\quad\text{in}\quad C^\alpha(\overline{I}\times \omega),
\label{eq:conetata}
\\
\label{eq:convu}
\mathcal D(\bfu_\delta)&\rightharpoonup^\eta\mathcal D(\bfu)\quad\text{in}\quad L^2(I;L^{2}(\Omega_{\eta_\delta})),\\
\bfu_\delta&\rightharpoonup^\eta\bfu\quad\text{in}\quad L^2(I;L^{2}(\Omega_{\eta_\delta})),%\label{eq:convu}
\\
\varrho_\delta&\rightharpoonup^{\ast,\eta}\varrho\quad\text{in}\quad L^\infty(I;L^\gamma(\Omega_{\eta_\delta})),\label{eq:convrho}\\
\vt_\delta&\rightharpoonup^{\ast,\eta}\vt\quad\text{in}\quad L^\infty(I;L^4(\Omega_{\eta_\delta})),\label{eq:convt}\\
\vt_\delta&\rightharpoonup^\eta\vt\quad\text{in}\quad L^2(I;W^{1,2}(\Omega_{\eta_\delta})).\label{eq:convt2}
\end{align}
%Since $\delta \varrho_\delta^\beta$ is uniformlyt bounded, we find that
%\begin{align}
%\delta^\frac1\beta\varrho_\delta\rightarrow 0\quad\text{in}\quad L^\beta(I\times\Omega_{\eta_\delta }),\label{eq:van}.
%\end{align}
%Note, that \eqref{con:epseta} and \eqref{con:epsetat} imply
%\begin{align}\label{eq:convetah}
%\eta_\delta\rightarrow\eta\quad\text{in}\quad C(\overline{I}\times \partial\Omega).
%\end{align}
By Lemma~\ref{thm:weakstrong}, arguing as in Sections \ref{subsec:teb} and \ref{subsec:del}, we find for all $q\in (1,\frac{6\gamma}{\gamma+6})$ that
\begin{align}
\varrho_\delta\bu_\delta&\rightharpoonup^\eta  {\varrho}  {\bfu}\qquad\text{in}\qquad L^2(I, L^q(\Omega_{\eta_\delta}))\label{conv:rhov2delta}\\
{\varrho}_\delta  {\bfu}_\delta\otimes  {\bfu}_\delta&\rightarrow^\eta  {\varrho}  {\bfu}\otimes  {\bfu}\qquad\text{in}\qquad L^1(I;L^1(\Omega_{\eta_\delta})).\label{conv:rhovv2delta}\\
\sqrt{{\varrho}_\delta}  {\bfu}_\delta&\rightarrow^\eta  \sqrt{\varrho}  {\bfu}\qquad\text{in}\qquad L^2(I;L^2(\Omega_{\eta_\delta})).\label{conv:rhovv2delta'}
\end{align}
As in Section \ref{sec:5} we also obtain again
\begin{align}
\bfS(\vt_\delta,\nabla\bfu_\delta)&\rightharpoonup^\eta  \overline{\bfS}\qquad\text{in}\qquad L^{4/3}(I, L^{4/3}(\Omega_{\eta_\delta}))\label{conv:S}\\
\label{conv:K'}
K'(\eta_\delta)&\rightharpoonup^* \overline K'\quad \text{in}\quad L^\infty(I;W^{-2,r}(\omega))
\end{align}
for any $r<2$ with some limit objects $\overline{\bfS}$ and $\overline K$.
As before in Proposition \ref{prop:higher} we have higher integrability of the density (see \cite[Lemma 7.3]{BrSc} for the proof).
% due to $\gamma>3$ and \eqref{eq:convetah'}.
\begin{lemma}\label{prop:higher'}
Let $\gamma>\frac{3}{2}$ ($\gamma>1$ in two dimensions).
Let $Q=J\times B\Subset I\times\Omega_\eta$ be a parabolic cube and $0<\Theta\leq\frac{2}{3}\gamma-1$. The following holds for any $\delta\leq \delta_0(Q)$
\begin{equation}\label{eq:gamma+1'}
\int_{Q}p_\delta(\vr_\delta,\vt_\delta)\varrho_\delta^{\Theta}\,\dif x\,\dif t\leq C(Q)
\end{equation}
with constant independent of $\delta$.
\end{lemma}

Similarly to \cite[Lemma 7.4]{BrSc} we can exclude concentrations of the pressure at the moving boundary. Here, we need the assumption $\gamma>\frac{12}{7}$.
\begin{lemma}\label{prop:higherb'}
Let $\gamma>\frac{12}{7}$ ($\gamma>1$ in two dimensions).
Let $\kappa>0$ be arbitrary. There is a measurable set $A_\kappa\Subset I\times\Omega_\eta$ such that we have for all $\delta\leq\delta_0$
\begin{equation}\label{eq:gamma+1b'}
\int_{I\times\R^3\setminus A_\kappa}p_\delta(\varrho_\delta,\vt_\delta)
\chi_{\Omega_{\eta_\delta}}\dif x\,\dif t\leq \kappa.
\end{equation}
\end{lemma}
Lemma~\ref{prop:higher'} and Lemma~\ref{prop:higherb'} imply equi-integrability of the sequence  $p_\delta(\varrho_\delta,\vt_\delta)\chi_{\Omega_{\eta_\delta}}$. This yields the existence of a function $\overline p$ such that (for a subsequence)
\begin{align}\label{eq:limp'}
p_\delta(\vr_\delta,\vt_\delta)\rightharpoonup\overline p\quad\text{in}\quad L^{1}(I\times\R^3),\\
\label{1301}
\delta\varrho_\delta^{\beta}\rightarrow0\quad\text{in}\quad L^1(I\times\R^3).
\end{align}
Similarly to Corollary \ref{prop:higherb0} we have the following. 
\begin{corollary}\label{prop:higherb0'}
Let $\kappa>0$ be arbitrary. There is a measurable set $A_\kappa\Subset I\times\Omega_\eta$ such that
\begin{equation}\label{eq:gamma+1b0'}
\int_{I\times\R^3\setminus A_\kappa}\overline p\,\dif x\,\dif t\leq \kappa.
\end{equation}
\end{corollary}

Using \eqref{eq:limp'} and the convergences \eqref{eq:conveta}--\eqref{conv:K'} we can pass to the limit in \eqref{eq:apu} and \eqref{eq:apvarrho0} and obtain
\begin{align}\label{eq:apulim'}
\begin{aligned}
&\int_I\frac{\dd}{\dt}\int_{\Omega_{ \eta}}\varrho\bfu \cdot\bfphi\dx-\int_{\Omega_{\eta}} \Big(\varrho\bfu\cdot \partial_t\bfphi +\varrho\bfu\otimes \bfu:\nabla \bfphi\Big)\dxt
\\
&+\int_I\int_{\Omega_\eta}\overline{\bfS}:\nabla\bfphi \dxt-\int_I\int_{\Omega_{ \eta }}
\overline p\,\Div\bfphi\dxt\\
&+\int_I\bigg(\frac{\dd}{\dt}\int_\omega \partial_t \eta b\dH-\int_\omega \partial_t\eta\,\partial_t b\dH + \int_\omega \overline K'\,b\dH\bigg)\dt
\\&=\int_I\int_{\Omega_{\eta}}\varrho\bff\cdot\bfphi\dxt+\int_I\int_\omega g\,b\,\dd x\dt
\end{aligned}
\end{align} 
for all test-functions $(b,\bfphi)$ with $\mathrm{tr}_\eta\bfphi=\partial_t\eta\nu$, $\bfphi(T,\cdot)=0$ and $b(T,\cdot)=0$. Moreover, the following holds
\begin{align}\label{eq:apvarrholim}
\int_I\int_{\Omega_{\eta}}\varrho\,\partial_t\psi\dxt-\int_I\int_{\Omega_{\eta}}\Div(\varrho\,\bfu)\,\psi\dxt=\int_{\Omega_{\eta_0}}\varrho_0\,\psi(0,\cdot)\dx
\end{align}
for all $\psi\in C^\infty(\overline{I}\times\R^3)$.
It remains to show strong convergence of $\vt_\delta$, $\vr_\delta$ and $\nabla^2\eta_\delta$. As in the last section the proof of the convergence of $\vt_\delta$ is entirely based on local arguments. Consequently the shell is not seen and we can follow the arguments in 
\cite[Chapter 3, Section 3.7.3]{F} to conclude
\begin{align}\label{m257}
\vt_\delta \to^\eta \vt \quad \mbox{in} \quad L^4(I;L^4(\Omega_\delta)).
\end{align}
Consequently we have $\overline\bfS=\bfS(\vartheta,\nabla \vu)$ in \eqref{eq:apulim'}. Moreover, we can pass to the limit in the entropy balance and obtain \ref{O2'}. Next we aim to prove strong convergence of the density.
We define the $L^\infty$-truncation
\begin{align}\label{eq:Tk'}
T_k(z):=k\,T\Big(\frac{z}{k}\Big)\quad z\in\mr,\,\, k\in\N.
\end{align}
Here $T$ is a smooth concave function on $\mr$ such that $T(z)=z$ for $z\leq 1$ and $T(z)=2$ for $z\geq3$.
Now we have to show that 
\begin{align}\label{eq:flux'}
\begin{aligned}
\int_{I\times\Omega_{\eta_\delta}}&\big( a\varrho_\delta^\gamma+\delta\varrho_\delta^\beta-(\lambda(\vt)+2\mu(\vt))\Div \bfu_\delta\big)\,T_k(\varrho_\delta)\dxt\\&\longrightarrow\int_{I\times\Omega_{\eta}} \big( \overline{p}-(\lambda(\vt)+2\mu(\vt))\Div \bfu\big)\,T^{1,k}\dxt.
\end{aligned}
\end{align}
For this step we are able to use the theory established in \cite{Li2} on a local level. Similarly to
\cite[Subsection 7.1]{BrSc} (see \cite[Chapter 3, Section 3.7.4]{F} about how to include the temperature) we first prove a localised version of \eqref{eq:flux'}
and then use Lemma \ref{prop:higherb'} and Corollary \ref{prop:higherb0'} to deduce the global version.
The next aim is to prove that $\varrho$ is a renormalized solution (in the sense of Definition~\ref{def:ren}).
In order to do so it suffices to use the continuity equation and \eqref{eq:flux'} again on the whole space.
Following line by line the arguments from \cite[Subsection 7.2]{BrSc} we have
\begin{align}\label{eq:Tk''}
\partial_t T^{1,k}+\Div\big( T^{1,k}\bfu\big)+T^{2,k}= 0
\end{align}
in the sense of distributions on $I\times\R^3$. Note that we extended
$\varrho$ by zero to $\R^3$.
The next step is to show
\begin{align}\label{eq.amplosc''}
\limsup_{\delta\rightarrow0}\int_{I\times\mt}|T_k(\varrho_\delta)-T_k(\varrho)|^{q}\dxt\leq C,
\end{align}

where $C$ does not depend on $k$ and $q>2$ will be specified later. The proof of \eqref{eq.amplosc''} follows exactly the arguments from the setting with fixed boundary
 (see \cite[Chapter 3, Section 3.7.5]{F}) using \eqref{eq:flux'} and the uniform bounds on $\bfu_\delta$ (with the only exception that we do not localise).
Using  \eqref{eq.amplosc''} and arguing as in \cite[Sec. 7.2]{BrSc} we obtain the renormalised continuity equation.
As in \cite[Sec. 7.3]{BrSc} we can use the latter one to show strong convergence of the density. %\seb{In order to pass to the limit with the weak momentum equation, fix a pair of smooth test-functions for the limit geometry $(b,\bfphi)\in C^\infty(\omega)\times C^\infty(\overline{I}\times\R^3)$ with $\mathrm{tr}_{\eta}\bfphi=b\nu$. Now since $\nabla \bfpsi_{\eta_\delta},\nabla \bfpsi_{\eta_\delta}^{-1}$ are uniformly in $L^\infty(I;L^q(\Omega_\eta))$ for all $q<\infty$ and $\partial_t\bfpsi_{\eta_\delta},\partial_t\bfpsi_{\eta_\delta}$ is uniformly in $L^2(I;L^a(\Omega_\eta))$ for all $a<4$ we find that $(b,\bfphi\circ \bfpsi_{\eta_\delta}\circ \bfpsi_{\eta}^{-1})$ is an admissible test function for the approximate weak momentum equation, with respective convergence properties.} 
Now we can pass to the limit with the approximate equations and and obtain the weak solution, as it was explained in the previous subsection.\\\
%:
%The renormalized continuity equation implies that on the $\delta$-level the energy equality is satisfied for $\mathcal E$ and $\mathcal E_\delta$. We can therefore pass to the limit (with $\mathcal E$) and obtain \ref{O3}.
% In particular establishing the weak momentum equation and the entropy inequality follows precisely as in the $\ep$-limit.

\centerline{\bf Acknowledgement}
\noindent{
S. Schwarzacher thanks the support of the research support programs PRIMUS/19/SCI/01 and UNCE/SCI/023 of Charles University. S.\ Schwarzacher thanks the support of the program GJ19-11707Y of the Czech national grant agency (GA\v{C}R). }
The authors wish to thank the referee for the careful reading of the paper and the valuable suggestions which helped to significantly improve the paper.
%
%\centerline{\bf Declaration}
%\noindent{
%The authors declare that there are no conflicts of interest.}

%\nocite{*}


\begin{thebibliography}{20}
\bibitem{Ac} G. Acosta, R. Dur\'an, and A. Lombardi: Weighted Poincar\'e and Korn inequalities for H\"older $\alpha$ domains. Math. Meth. Appl. Sci. 29, 387--400. (2006)
%\bibitem{AdFo} R. A. Adams and J. Fournier: Sobolev spaces. 2nd Edition. Academic Press. (2003)
%\bibitem{BrDS}  D. Breit, L. Diening \& S. Schwarzacher, Solenoidal Lipschitz truncation for parabolic PDE's, Math. Mod. Meth. Appl. Sci. 23, 2671--2700, 2013.
 \bibitem{Be} H. Beir\~{a}o da Veiga: On the existence of strong solutions to a coupled fluid-structure
evolution problem. J. Math. Fluid Mech.6, 21--52. (2004)
\bibitem{Bo} M. Boulakia: Existence of weak solutions for an interaction problem between an
elastic structure and a compressible viscous fluid. J. Math. Pures Appl. (9) 84(11),
1515--1554. (2005)
\bibitem{BGN} T. Bodn\'ar, G. Galdi and v. Ne\v{c}asov\'a, Fluid-Structure Interaction and Biomedical Applications, Springer, Basel. (2014)
\bibitem{BrMe} D. Breit \& P. R. Mensah: An incompressible polymer fluid interacting with a Koiter shell.  J. Nonlinear Sci. 31, 25. (2021)
\bibitem{BrSc} D. Breit \& S. Schwarzacher: Compressible fluids interacting with a linear-elastic shell. Arch. Rational Mech. Anal. 228, 495--562. (2018)
%\bibitem{BuMu} M. Bukac, B. Muha: Stability and Convergence Analysis of the Extensions of the Kinematically Coupled Scheme for the Fluid-Structure Interaction.  SIAM J. Num. Anal. 54, No. 5, 3032--3061. (2016)
%\bibitem{ByWa} S.-S. Byun, L. Wang (2005): $L^p$ estimates for parabolic equations in Reifenberg
%domains. J. Funct. Anal. 223, 44--85.
\bibitem{Ci1} P.G. Ciarlet: Mathematical elasticity. Vol. II. Theory of plates. In: Studies in Mathematics
and its Applications, vol. 27. North-Holland, Amsterdam. (1997)
\bibitem{Ci2} P.G. Ciarlet: Mathematical elasticity, vol. III. Theory of shells. In: Studies in Mathematics
and its Applications, vol. 29. North-Holland Publishing Co., Amsterdam. (2000)
\bibitem{Ci3} P.G. Ciarlet: An Introduction to Differential Geometry with Applications to Elasticity.
Springer, Dordrecht. Reprinted from J. Elasticity 78/79, no. 1--3. (2005)
\bibitem{ciarlet2001justification}
Ciarlet, P.G., Roquefort, A.: Justification of a two-dimensional nonlinear
  shell model of {K}oiter's type.
\newblock Chinese Ann. Math. Ser. B {22}(2), 129--144 (2001).
\bibitem{Su} S. K. Chakrabarti: The theory and practice of hydrodynamics and vibration. River Edge, N.J. : World Scientific. (2002)
\bibitem{Ch} A. Chambolle, B. Desjardins, M. J. Esteban, C. Grandmont: Existence of weak
solutions for the unsteady interaction of a viscous fluid with an elastic plate. J. Math.
Fluid Mech.7(3), 368--404. (2005)
\bibitem{ChSk} C. H. A. Cheng, S. Shkoller: The interaction of the 3D Navier--Stokes equations with
a moving nonlinear Koiter elastic shell. SIAM J. Math. Anal.42(3), 1094--1155. (2010)
%\bibitem{CoSh1} D. Coutand, S. Shkoller: Motion of an elastic solid inside an incompressible viscous
%fluid. Arch. Rational Mech. Anal.176(1), 25--102. (2005)
\bibitem{CoSh2} D. Coutand, S. Shkoller: The interaction between quasilinear elastodynamics and
the Navier--Stokes equations. Arch. Rational Mech. Anal.179(3), 303--352. (2006)
\bibitem{DL} R. J. DiPerna, P.-L. Lions: Ordinary differential equations, transport theory and
Sobolev spaces. Invent. Math. 98, 511--547. (1998)
\bibitem{Do} W. Dowell: A Modern Course in Aeroelasticity. Fifth Revised and Enlarged Edition. Solid Mechanics and Its Applications 2017. Springer (2015)
\bibitem{DuFe} B. Ducomet, E. Feireisl:
On the Dynamics of Gaseous Stars.
Arch. Rational Mech. Anal. 174, 221--266. (2004)
%\bibitem{DRS} L. Diening, M.
% R\r{u}\v{z}i\v{c}ka,
%  K. Schumacher (2009): A Decomposition technique for John
%  domains. Ann. Acad. Scientiarum Fennicae 35, 87--114.
%\bibitem{feireisl2} E. Feireisl: On compactness of solutions to the compressible isentropic Navier--Stokes
%equations when the density is not square integrable. Comment. Math. Univ. Carolinae 42
%(1), 83--98. (2001)
\bibitem{Fe} E. Feireisl: On the motion of rigid bodies in a viscous compressible fluid. Arch. Rational Mech. Anal. 167 (3), 281--308. (2003)
\bibitem{fei3} E. Feireisl: Dynamics of Compressible Flow. Oxford Lecture  Series in Mathematics and its Applications, Oxford University Press, Oxford. (2004)
%\bibitem{feiKNNS} E. Feireisl, O. Kreml, S. Ne\v{c}asov\'a, J. Neustupa, J. Stebel:
%Weak solutions to the barotropic Navier--Stokes system with slip boundary conditions in time dependent domains. J. Diff. Equ. 254 (1), 125--140. (2013)
\bibitem{FN} E. Feireisl, A. Novotn\'y: On a simple model of reacting compressible
flows arising in astrophysics. Proc. Royal Soc. Edinburgh, 135A, 1169--1194.
(2005)
\bibitem{F} E. Feireisl, A. Novotn\'{y} (2009): \emph{Singular limits in thermodynamics of viscous fluids.} Birkh\"auser-Verlag, Basel, 2009.
\bibitem{feireisl1} E. Feireisl, A. Novotn\'y, H. Petzeltov\'a: On the existence of globally defined weak solutions to the Navier--Stokes equations of compressible isentropic fluids. J. Math. Fluid. Mech. 3, 358--392. (2001) 
%\bibitem{GD} A. Granas, J. Dugundji: Fixed point theory. Springer Monographs in Mathematics.
%Springer, New York. (2003)
\bibitem{Gr} C. Grandmont: Existence of weak solutions for the unsteady interaction of a viscous
fluid with an elastic plate. SIAM J. Math. Anal.40(2), 716--737. (2008)
\bibitem{GraHil} C. Grandmont, M. Hillairet: Existence of Global Strong Solutions
to a Beam-Fluid Interaction System. Arch. Rational Mech. Anal. 220, 1283--1333. (2016)
%\bibitem{sarka} A. Hundertmark-Zau\v{s}kov\'a, M. Luk\'a\v{c}ov\'a-Medvid'ov\'a, S. Ne\v{c}asov\'a:
%On the existence of weak solution to the coupled fluid-structure interaction problem for non-Newtonian shear-dependent fluid.
%J. Math. Soc. Japan 68 (1), 2016, 193--243. 
\bibitem{Koi} W.T. Koiter: A consistent first approximation in the general theory of thin elastic
shells. In: Proc. Sympos. Thin Elastic Shells (Delft, 1959) (Amsterdam), pp. 12--33.
North-Holland, Amsterdam. (1960)
\bibitem{Koi2} W.T. Koiter: On the nonlinear theory of thin elastic shells. I, II, III. Nederl. Akad.
Wetensch. Proc. Ser. B69, 1--17, 18--32, 33--54. (1966)
%\bibitem{Ku} A. Kufner, O. John, S. Fuc\'ik: Function spaces. In: Monographs and Textbooks on
%Mechanics of Solids and Fluids; Mechanics: Analysis. Noordhoff International Publishing, Leyden. (1977)
%\bibitem{KuTu} I. Kukavica, A. Tuffaha: Well-posedness for the compressible Navier--Stokes--Lam\'e system with
%a free interface. Nonlinearity 25, 3111--3137. (2012)
\bibitem{Kuk} P. Kuku\v{c}ka: On the existence of finite energy weak solutions to the Navier--Stokes equations in irregular domains. Math. Meth. Appl. Sci. 32, Issue 11, 1428--1451. (2009)
%\bibitem{LaUrSo} Ladyzhenskaya, O.A., Solonnikov, V.A., Ural'tseva, N.N.: Linear and quasi-linear equations of parabolic type. Translations of Mathematical Monographs. 23. Amer. Math. Soc., Providence, RI. (1967)
%\bibitem{Le} D. Lengeler, Weak solutions for an incompressible, generalized Newtonian fluid interacting with a linearly elastic Koiter type shell. SIAM J. Math. Anal. 46 (2014), no. 4, 2614--2649.
\bibitem{LSU} O. A. Ladyzhenskaya, V. A. Solonnikov, and N. N. Ural'tseva.
Linear and quasilinear equations of parabolic type. Translated from
the Russian by S. Smith. Translations of Mathematical Monographs,
Vol. 23. American Mathematical Society, Providence, R.I., 1967. 
\bibitem{Ler} J. Leray: Sur le mouvement d'un liquide visqueux emplissant l'espace.
%Acta Math. 63, 193--248. (1934)
%\bibitem{Lee} J.M.L. Lee: Introduction to smooth manifolds. Graduate Texts in Mathematics, vol.
%218. Springer, New York. (2003)
\bibitem{Le} D. Lengeler: Weak solutions for an incompressible, generalized Newtonian fluid interacting with a linearly elastic Koiter type shell. SIAM J. Math. Anal. 46, no. 4, 2614--2649. (2014)
%\bibitem{Leq} J. Lequeurre: Existence of strong solutions to a fluid-structure system. SIAM J. Math.
%Anal.43(1), 389--410. (2011)
\bibitem{LeRu} D. Lengeler, M. \Ruzicka: Weak Solutions for an Incompressible Newtonian Fluid Interacting with a Koiter Type Shell. Arch. Ration. Mech. Anal. 211, no. 1, 205--255. (2014)
%\bibitem{Lialt} J. L. Lions: Quelques m\'{e}thodes de r\'{e}solution des probl\`{e}mes aux limites non lin\'{e}aires. Dunod, Gauthier-Villars, Paris. (1969)
%\bibitem{Li1} P.\,L. Lions, Mathematical topics in fluid mechanics. Vol. 1. Inompressible models. Oxford Lecture Series in Mathematics and its Applications, 3. Oxford Science Publications, The Clarendon Press, Oxford University Press, New York. (1996)
\bibitem{Li2} P.\,L. Lions: Mathematical topics in fluid mechanics. Vol. 2. Compressible models. Oxford Lecture Series in Mathematics and its Applications, 10. Oxford Science Publications, The Clarendon Press, Oxford University Press, New York. (1998)
\bibitem{Li} G. M. Lieberman. Second order parabolic differential equations. World Scientific Publishing Co., Inc., River Edge, NJ. (1996) 
\bibitem{Ma} D. Maity, A. Roy, T. Takahashi: Existence of strong solutions for a system of interaction between a compressible viscous fluid and a wave equation. Preprint at hal.archives-ouvertes.fr
\bibitem{Ma2} D. Maity, T. Takahashi: Existence and uniqueness of strong solutions for the system of interaction between a compressible Navier-Stokes-Fourier fluid and a damped plate equation. Preprint arXiv:2006.00488
\bibitem{Mi} S. Mitra: Local Existence of Strong Solutions of a Fluid--Structure Interaction Model. J. Math. Fluid Mech. 22: 60. (2020)
\bibitem{MuCa1} B. Muha, S. Canic: Existence of a weak solution to a nonlinear fluid-structure interaction problem modeling the flow of an incompressible, viscous fluid in a cylinder with deformable walls. Arch. Rational Mech. Anal. 207 (3), 919--968. (2013)
\bibitem{MuCa}
Muha, B., \v{C}ani\'{c}, S.: Existence of a solution to a
  fluid-multi-layered-structure interaction problem.
\newblock J. Differential Equations {256}(2), 658--706. (2014)
%\bibitem{MuCa2} B. Muha, S. Canic: Fluid-structure interaction between an incompressible, viscous 3D fluid and
%an elastic shell with nonlinear Koiter membrane energy.  Interfaces and Free Boundaries 17, 465--495.  (2015)
\bibitem{MuSc} B. Muha, S. Schwarzacher: Existence and regularity for weak solutions for a fluid interacting with a non-linear shell in 3D, to appear in Ann. I. H. Poincar\'e -- AN. (2021)
%   \bibitem{MNRR} J. M{\'a}lek, J. Nec\v{a}s, M. Rokyta, M. R\r{u}\v{z}i\v{c}ka: Weak and measure valued solutions to evolutionary PDEs. Chapman \& Hall, London-Weinheim-New York. (1996)
\bibitem{NRL}P. N\"agele, M. \Ruzicka, and D. Lengeler: Functional setting for unsteady problems in moving domains and applications. Comp. Var. Ell. Syst., 62(1):66--97. (2016)
%\bibitem{novot} A. Novotn\'y, I. Stra\v{s}kraba: Introduction to the Mathematical Theory of Compressible Flow, Oxford Lecture Series in Mathematics and its Applications, Oxford University Press, Oxford. (2004)

\bibitem{SchShe20}
S. Schwarzacher, B. She:
\newblock On numerical approximations to fluid-structure interactions involving compressible fluids.
\newblock {\em arXiv preprint:} arXiv:2002.04636. (2020)


\bibitem{SchSro20}
S. Schwarzacher, M. Sroczinski:
\newblock Weak-strong uniqueness for an elastic plate interacting with the Navier Stokes equation.
\newblock {\em arXiv preprint:} arXiv:2003.04049. (2020)

\bibitem{Tr} S. Trifunovi\'c, Y.-G. Wang: On the interaction problem between a compressible viscous fluid and a nonlinear thermoelastic plate. Preprint arXiv:2010.01639. (2020)


%\bibitem{resh} Y.G. Reshetnyak. Estimates for certain
%differential operators with finite dimensional kernel.
%Sibirskii Math. Zh. 2, 414--418. (1970)
\end{thebibliography}
\end{document}